\documentclass[a4paper, 10pt]{article}
\usepackage{lineno}
\usepackage{hyperref}
\usepackage{enumitem}
\usepackage[hyperref, dvipsnames]{xcolor}
\hypersetup{
  colorlinks=true,
}
\usepackage{amsmath, amsthm, amssymb, color, mathtools, bm, mathrsfs, authblk}
\modulolinenumbers[5]
\makeatletter
\def\ps@pprintTitle{%
  \let\@oddhead\@empty
  \let\@evenhead\@empty
  \def\@oddfoot{\reset@font\hfil\thepage\hfil}
  \let\@evenfoot\@oddfoot
}
\makeatother

\DeclareMathOperator*{\esssup}{ess\,sup}
\DeclareMathOperator*{\sign}{sign}
\DeclareMathOperator*{\supp}{supp}
\DeclareMathOperator{\Lip}{Lip}
\newcommand{\norm}[2]{\left\lVert #1\right\rVert_{#2}}
\newcommand{\seminorm}[2]{\left\vert #1\right\vert_{#2}}

\newcommand{\weakstar}{\stackrel{*}\rightharpoonup}
\newcommand{\pme}{\Psi}
\newcommand{\md}{\partial^\bullet}

\newcommand{\cts}{\hookrightarrow}

\newcommand{\compact}{\xhookrightarrow{c}}

\newtheorem{theorem}{Theorem}[section]
\newtheorem{defn}[theorem]{Definition}
\newtheorem{lem}[theorem]{Lemma}

\newtheorem{ass}[theorem]{Assumption}

\newtheorem{cor}[theorem]{Corollary}
\newtheorem{remark}[theorem]{Remark}

\newcommand{\C}{\mathcal{C}}

\newcommand{\nablabg}{\nabla_{\bar g}}

\newcommand{\nablabgt}{\nabla_{\bar g(t)}}
\newcommand{\nablabgs}{\nabla_{\bar g(s)}}
\newcommand{\Deltabg}{\Delta_{\bar g}}

\newcommand{\Deltabgt}{\Delta_{\bar g(t)}}
\newcommand{\sgrad}{\nabla_\Gamma}
\newcommand{\sgradt}{\nabla_{\Gamma(t)}}
\newcommand{\sgrads}{\nabla_{\Gamma_0}}
\newcommand{\halflapt}{(-\slapt)^{1\slash 2}}

\newcommand{\slapt}{\Delta_{\Gamma(t)}}
\newcommand{\weaklyto}{\rightharpoonup}

\newcommand*\mean[1]{\overline{#1}}
\newcommand*\meanu[1]{\overline{#1}}
\newcommand*\ext[1]{\overline{#1}}

\begin{document}
\hypersetup{
  urlcolor     = blue, 
  linkcolor    = Bittersweet, 
  citecolor   = Cerulean
}

%\begin{frontmatter}

\title{Well-posedness of a fractional porous medium equation on an evolving surface}

\author{Amal Alphonse and Charles M. Elliott}
\affil{Mathematics Institute\\ University of Warwick\\ Coventry CV4 7AL\\ United Kingdom}

%\author[mymainaddress]{Amal Alphonse\corref{mycorrespondingauthor}}
%\ead{a.c.alphonse@warwick.ac.uk}
%\author[mymainaddress]{Charles M. Elliott}
%\ead{c.m.elliott@warwick.ac.uk}




%\address[mymainaddress]{Mathematics Institute, University of Warwick, Conventry CV4 7AL, United Kingdom}
\maketitle
\begin{abstract}
We investigate the existence, uniqueness, and $L^1$-contractivity of weak solutions to a porous medium equation with fractional diffusion on an evolving hypersurface. To settle the existence, we reformulate the equation as a local problem on a semi-infinite cylinder, regularise the porous medium nonlinearity and truncate the cylinder. Then we pass to the limit first in the truncation parameter and then in the nonlinearity, and the identification of limits is done using the theory of subdifferentials of convex functionals.  

In order to facilitate all of this, we begin by studying (in the setting of closed Riemannian manifolds and Sobolev spaces) the fractional Laplace--Beltrami operator which can be seen as the Dirichlet-to-Neumann map of a harmonic extension problem. A truncated harmonic extension problem will also be examined and convergence results to the solution of the harmonic extension will be given. For a technical reason, we will also consider some related extension problems on evolving hypersurfaces which will provide us with the minimal time regularity required on the harmonic extensions in order to properly formulate the moving domain problem. This functional analytic theory is of course independent of the fractional porous medium equation and will be of use generally in the analysis of fractional elliptic and parabolic problems on manifolds. 
\end{abstract}

\begin{center}
{\em Dedicated to Juan Luis V{\'a}zquez on the occasion of his 70\textsuperscript{th} birthday}
\end{center}
%\begin{keyword}
%porous medium equation\sep fractional laplacian \sep evolving hypersurface
%\MSC[2010] 35R01\sep 35R37\sep 35K65  \sep 35R11  
%\end{keyword}

%\end{frontmatter}

%\linenumbers
%\tableofcontents
\section{Introduction}
For each $t \in [0,T]$, let $\Gamma(t) \subset \mathbb{R}^{d+1}$ be a smooth and compact $d$-dimensional hypersurface without boundary evolving with a given velocity field $\mathbf w$. In this paper, we are interested in the well-posedness of the fractional porous medium equation
\begin{equation}\label{eq:pme}
\begin{aligned}
\dot u(t) + \halflapt (u^m(t)) + u(t)\sgradt \cdot \mathbf w(t) &=0&&\text{on $\Gamma(t)$}\\
u(0) &= u_0&&\text{on $\Gamma_0 := \Gamma(0)$}
\end{aligned}
\end{equation}
for $m \geq 1$, where $u_0$ is a given initial data, $u^m := |u|^{m-1}u$ as usual, and $\halflapt$ is the square root of the Laplace--Beltrami operator on $\Gamma(t)$, which is a nonlocal first order elliptic pseudodifferential operator \cite{Seeley, Shubin, Wunsch, TaylorPDEII}.

%and $-\Delta_{\Gamma(t)}$ is the Laplace--Beltrami operator on $\Gamma(t)$. The half-Laplacian $\halflapt$ is a nonlocal first order elliptic pseudodifferential operator \cite{Seeley, Shubin, Wunsch, TaylorPDEII}.

If the fractional Laplacian in \eqref{eq:pme} is replaced with the ordinary Laplace--Beltrami operator $-\slapt$, \eqref{eq:pme} would be a porous medium equation on an evolving surface. Porous medium equations on stationary domains have, of course, attracted a considerable and well-developed literature. We refer the reader to the book \cite{VazquezBook} by V{\'a}zquez which is a comprehensive study of the mathematical analysis of the equation (and it also contains many references) and results on the porous medium equation on manifolds can be found in \cite[\S 11.5]{VazquezBook} and \cite{BonforteGrillo}. We will also say a few words about the non-fractional moving case in the conclusion of this paper. The investigation of \emph{fractional} porous medium equations was instituted in \cite{DePablo2011} where the authors examined such a problem on $\mathbb{R}^d$ involving the square root of the Laplacian and gave a complete theory of the equation, and indeed, our work is motivated by the results in that paper. There, the existence was proved by discretisation in time of a localised formulation of the equation and then the application of the Crandall--Liggett theorem \cite{CrandallLiggett}. Those results were generalised in \cite{Pablo2012} to a wider range of fractional powers of the Laplacian $(-\Delta)^{s}$ with exponent $s \in (0,1)$ on a stationary domain $\Omega \subseteq \mathbb{R}^d$ using the extension method introduced by Caffarelli and Silvestre in \cite{CaffarelliSilvestre}. Existence was proved in \cite{Bonforte} (for a more general nonlinearity) in a different way through the theory of semigroups and maximal monotone operators. Our model \eqref{eq:pme} differs from all of the aforementioned works since it is on a moving space. 

Other related works in the literature include variants of nonlocal porous medium equations such as those with variable density \cite{Punzo, PunzoGeneral} and different fractional operators \cite{Karch}. We also mention \cite{Alves2014, Capella2010, Stinga2014, Montefusco2012} where elliptic fractional problems are studied in the setting of the Laplacian on a bounded domain with Neumann boundary conditions, and \cite{Imbert2011} where a degenerate parabolic equation arising in crack dynamics is considered, again in the Neumann setting. One can also find numerical and finite element analysis for elliptic and parabolic problems in \cite{Nochetto2013, Nochetto2014}. As is evident, there has been an extraordinary amount of activity in fractional diffusion problems in the last decade or so. A good survey of recent and current output involving nonlinear fractional diffusion can be found in the articles \cite{VazquezAbel, VazquezRecentProgress}.

In terms of the analysis, a common preliminary step when working with half-Laplacians is to rewrite the problem locally using a Dirichlet-to-Neumann map \cite{CabreTan, Banica2012, ValdinociSire, ChangGonzalez}. We will also reformulate \eqref{eq:pme} using such a map; this step is likewise performed in \cite{DePablo2011, Pablo2012} but from here on, the type of approaches taken in \cite{DePablo2011, Pablo2012} are problematic in our setting because of the additional complexity engendered by the evolving domain. For example, one could attempt to pull back the problem onto a reference domain (the resulting expression is not too cumbersome if the evolution of $\Gamma(t)$ is prescribed particularly agreeably) and try to employ an appropriate time-dependent version of Crandall--Liggett \cite{Crandall1972, Evans1977,Lin2011} to the resulting equation (which will have time-dependent coefficients) but these theorems are difficult to apply even when the evolution of the domains is highly simplified. Therefore, we choose a different way to approach this problem, which we shall outline below, starting from the foundations. To our knowledge, the type of approach developed in this paper has not been used before in the fractional setting, even in the stationary case. The challenges and peculiarities that arise due to the moving domain will be highlighted in due course. 

Before we proceed, let us remark that fractional Laplace--Beltrami operators on various classes of manifolds have been studied in \cite{Banica2012, ValdinociSire, ChangGonzalez} through extension problems in the style of Caffarelli--Silvestre \cite{CaffarelliSilvestre}, but a convenient work detailing all the relevant properties of the half-Laplacian on closed manifolds in a Sobolev space setting appears lacking, so this paper is useful also in this respect. With this in mind, it is worth emphasising that the first part of this paper, comprising of \S \ref{sec:fractionalLaplacianOnCompactManifolds}--\ref{sec:fractionalLaplacianOnL2X}, is independent of the second part which consists of \S \ref{sec:nondegenerateProblem} and \S \ref{sec:fpme}, and indeed the reader can read the first part in isolation. The first part can be of use for other fractional diffusion problems on (evolving) manifolds and the second part can be thought of as an application of the first part. See the outline below for more details.
\subsection{Reformulation of the equation and main results}
A natural way to define $\halflapt$ is through a spectral definition which we describe now in greater generality. Indeed, suppose that
\begin{equation}\label{ass:onM}
(M,g)\text{ is a connected closed smooth Riemannian manifold}\tag{$A_M$}
\end{equation}
and let $(\varphi_k, \lambda_k)_{k \in \mathbb{N}}$ be the normalised eigenpairs of the Laplacian $-\Delta_M$ so that $-\Delta_M \varphi_k = \lambda_k \varphi_k$ for each $k$; it follows that $0=\lambda_0 < \lambda_1 \leq \lambda_2 \leq ... \nearrow \infty $ and $\varphi_0 \equiv |M|^{- 1\slash 2}$ \cite[Theorem 3.2.1]{Jost}. The $\varphi_k$ form an orthonormal basis of $L^2(M)$ and are orthogonal in $H^1(M)$. For smooth functions $u$, define
\begin{equation}\label{eq:defnFractionalLaplacian}
(-\Delta_M)^{1\slash 2}u := \sum_{k=1}^\infty \lambda_k^{1\slash 2}(u,\varphi_k)_{L^2(M)}\varphi_k.
\end{equation}
The operator $(-\Delta_M)^{1 \slash 2}$ can  be defined in a weaker sense through the action
%This definition can be extended to a weak fractional Laplacian via the action
\begin{equation}\label{eq:spectralH12Action}
\langle (-\Delta_M)^{1\slash 2}u, v \rangle := \sum_{k=1}^\infty \lambda_k^{1\slash 2}(u,\varphi_k)_{L^2(M)}(v,\varphi_k)_{L^2(M)}
\end{equation}
which is sensible whenever $u$ and $v$ belong to the Hilbert space
\begin{equation}\label{eq:spaceHM}
H(M) := \left\{ u \in L^2(M) \mid \sum_{k=1}^\infty \lambda_k^{1\slash 2}|(u,\varphi_k)_{L^2(M)}|^2 < \infty\right\}
\end{equation} 
endowed with the inner product
\[(u,v)_{H(M)} := (u,v)_{L^2(M)} + \sum_{k=1}^\infty \lambda_k^{1\slash 2}(u,\varphi_k)_{L^2(M)}(v,\varphi_k)_{L^2(M)}.\]
It is useful to have a Sobolev characterisation of the space $H(M)$; in Lemma \ref{lem:HMEquivalence}, we will see that
\[H(M) = H^{1\slash 2}(M) = B^{1\slash 2}_{22}(M) = (L^2(M), W^{1,2}(M))_{1\slash 2},\]
i.e., $H(M)$ is exactly the fractional Sobolev space $H^{1\slash 2}(M)$  %defined as a real interpolation space 
%\[H^{1\slash 2}(M) = B^{1\slash 2}_{22}(M) = (L^2(M), W^{1,2}(M))_{\frac 12},\]
(see \cite[\S 7.2.2, \S 7.3.1, \S 7.4.5]{Triebel} for more details on the second and third equalities). In the later sections, we will be working on hypersurfaces so it is convenient for our purposes to introduce the Sobolev--Slobodecki\u\i\ space $W^{1 \slash 2, 2}(\Gamma)$ (where $\Gamma$ is a sufficiently smooth hypersurface) defined using the Gagliardo norm (see \cite{Alphonse2014a} and references therein):
\[W^{1\slash 2, 2}(\Gamma) := \left\{ u \in L^2(\Gamma) \mid  \norm{u}{W^{1\slash 2, 2}(\Gamma)}^2 := \int_{\Gamma}|u(x)|^2\;\mathrm{d}\sigma(x) +  \int_{\Gamma}\int_{\Gamma}\frac{|u(x) - u(y)|^2}{|x-y|^{n}}\;\mathrm{d}\sigma(x)\mathrm{d}\sigma(y) < \infty\right\}.\]
Of course, this space is equivalent to $H^{1\slash 2}(\Gamma)$ with an equivalence of norms (see \cite[\S I.4.2 and Theorem 5.2 of \S I.5.1]{Wloka}, \cite[Theorem 7.7, Chapter 1]{LionsMagenes}, \cite[Chapter 1, \S 15]{LionsMagenes} and \cite[\S 1.3.3]{Grisvard}), 
%Of course, this space is equivalent to $H^{1\slash 2}(\Gamma)$ with an equivalence of norms: for $s \in (0,1)$, $W^{s,2}(\Gamma)$ is equivalent to the fractional Sobolev space $\tilde W^{s,2}(\Gamma)$ \cite[\S I.4.2]{Wloka} defined via charts and partitions of unity using the norm in the Sobolev--Slobodecki\u\i\ space $W^{s,2}(\mathbb{R}^r) \equiv H^s(\mathbb{R}^r)$ (see \cite[Theorem 5.2, \S I.5.1]{Wloka} for the equivalence to $H^s(\mathbb{R}^r)$, which is defined through Fourier transforms), and $\tilde W^{s,2}(\Gamma)$ is equivalent to the interpolation space $H^s(\Gamma) = (L^2(\Gamma), W^{1,2}(\Gamma))_{1-s}$ \cite[Theorem 7.7, Chapter 1]{LionsMagenes}  (in this connection, see \cite[Chapter 1, \S 15]{LionsMagenes}), see also \cite[\S 1.3.3]{Grisvard}. 
but it is important to distinguish between these spaces when $\Gamma=\Gamma(t)$ is time-dependent because the constants in the equivalence of norms will depend on $t$ in an unknown way. %These comments will be expanded upon in \S \ref{sec:functionSpaces}.

The spectral definition of $(-\Delta_M)^{1\slash 2}$ in \eqref{eq:defnFractionalLaplacian} is not particularly amenable to a convenient theory of weak solutions; however, there is a way to localise the fractional Laplacian (see \cite{Banica2012, ValdinociSire, ChangGonzalez}). With $\C := M \times [0,\infty)$ and $\bar g$ denoting the trivial product metric on $\C$, consider the problem
\begin{equation}\label{eq:prelim1}
\begin{aligned}
\Deltabg v&=0&&\text{on $\C$}, \qquad v|_{\partial \C} = u
\end{aligned}
\end{equation}
where $\partial \C = M \times \{0\}$. Whenever $u$ belongs to $H(M)$, the equation has a unique weak solution $v = \ext{\mathcal{E}}u$, called the harmonic extension of $u$. This harmonic extension $\ext{\mathcal{E}}u$ belongs in general not to $H^1(\C)$ but to the larger space 
\begin{equation}\label{eq:spaceXC}
X(\C) := \overline{H^1(\C)}^{\norm{\cdot}{X(\C)}} \text{ where }\norm{v}{X(\C)}^2 := \norm{\nablabg v}{L^2(\C)}^2+ \norm{\mathcal{T} v}{L^2(M)}^2\quad\text{for $v \in H^1(\C)$}
\end{equation}
with $\mathcal{T}\colon H^1(\C) \to H(M)$ denoting the trace map onto $M \times \{0\}$, so that $\ext{\mathcal{E}}\colon H(M) \to X(\C)$ (this type of space $X(\C)$ was first defined in a different setting by Stinga and Volzone in \cite{Stinga2014}). As we shall see in Lemma \ref{lem:fracDTN}, the fractional Laplacian is  recovered as a Dirichlet-to-Neumann map:
\begin{equation*}\label{eq:prelimDtn}
\langle (-\Delta_M)^{1\slash 2}u, v \rangle_{H(M)^*, H(M)} = \left\langle \frac{\partial (\ext{\mathcal{E}}u)}{\partial \nu}\Big|_{y=0}, v \right\rangle_{H(M)^*, H(M)},
\end{equation*}
where $\nu = (0,-1)$ is the outward normal to $\C$. All of this will be laid out in detail in \S \ref{sec:fractionalLaplacianOnCompactManifolds}. 

Setting $\pme(r):= |r|^{m-1}r$ and $\C(t) := \Gamma (t) \times [0,\infty)$, the above characterisation implies that one can rewrite \eqref{eq:pme} as
\begin{equation}\label{eq:localEquation2}
\begin{aligned}
\dot u(t) + u(t)\sgradt \cdot \mathbf w(t) + \frac{\partial v(t)}{\partial \nu(t)} &=0&&\text{on $\partial \C(t)$}\\
v(t)&=\ext{\mathcal{E}}_t(\pme(u(t)))\\%&&\text{on $\partial \C(t)$}%\\
u(0) &= u_0 &&\text{on $\Gamma_0$}
\end{aligned}\tag{$\textbf{P}$}
\end{equation}
where $\ext{\mathcal{E}}_t$ is the map $\ext{\mathcal{E}}$ with the manifold $M$ chosen to be $\Gamma(t)$ and $\nu(t) = (0,-1)$ is outward normal to $\C(t)$. Regarding the regularity of $\{\Gamma(t)\}_{t \in [0,T]}$, we will assume Assumption \ref{ass:onHypersurfaces} on p.~\pageref{ass:onHypersurfaces} and that
\begin{equation}\label{eq:eigenvalueEstimate}
\text{there exists a constant $\lambda_1 > 0$ such that }\lambda_1(t) \geq \lambda_1\text{ for all $t \in [0,T]$}\tag{$A_{\lambda}$}
\end{equation}
where $\lambda_k(t)$ denotes the $k$-th eigenvalue of $-\Delta_{\Gamma(t)}$; see Remark \ref{rem:lots}. A proper weak formulation of this problem requires the use of appropriate functional spaces. In \cite{Alphonse2014, Alphonse2014b} (see also \cite{Alphonse2014a}), we defined generalisations of the Bochner spaces $L^p(0,T;Y)$ to handle functions defined on evolving spaces: given a family of Banach spaces $Y \equiv \{Y(t)\}_{t \in [0,T]}$, a family of uniformly bounded linear homeomorphisms $\{\phi_t \colon Y_0 \to Y(t)\}_{t \in [0,T]}$ with uniformly bounded inverses $\{\phi_{-t}\colon Y(t) \to Y_0\}_{t \in [0,T]},$ and $t \mapsto \norm{\phi_t u}{Y(t)}$ measurable for all $u \in Y_0$, we, generalising some work by Vierling \cite{Vierling}, defined the Banach spaces $L^p_Y$ as
\begin{align*}
L^p_Y &= \begin{cases}
\{u:[0,T] \to \bigcup_{t \in [0,T]}Y(t) \times \{t\},\quad \!\!t \mapsto (\hat u(t), t)\quad\!\!\!\! \mid \quad\!\!\!\!\phi_{-(\cdot)} \hat u(\cdot) \in L^p(0,T;Y_0 )\}\quad\!\!\!& \text{for $p \in [1,\infty)$}\\
\{ u \in L^2_{Y} \mid \esssup_{t \in [0,T]} \norm{u(t)}{Y(t)} < \infty\} &\text{for $p=\infty$}
\end{cases}
\end{align*}
endowed with the norm
\begin{equation*}
\begin{aligned}\label{eq:ip}
\norm{u}{L^p_Y} &=\begin{cases}
\left({\int_0^T \norm{u(t)}{Y(t)}^p}\right)^{1\slash p} &\text{for $p \in [1,\infty)$}\\
\esssup_{t \in [0,T]}\norm{u(t)}{Y(t)} &\text{for $p=\infty$}.
\end{cases}
\end{aligned}
\end{equation*}
(Note that we made an abuse of notation after the definition of the first space and identified $u(t) = (\hat u(t), t)$ with $\hat u(t)$.) The space $\mathbb{W}(Y,Z) := \{u \in L^2_Y \mid \dot u \in L^2_Z\}$ with $\dot u$ the weak time or material derivative refers to an evolving space version of a Sobolev--Bochner space; this notion will be properly defined in \S \ref{sec:functionSpaces} where we shall also make clear the assumptions on the evolution of the hypersurface $\Gamma(t)$. This theory will allow us to define the following spaces (amongst others) after we make and check the relevant assumptions in \S \ref{sec:functionSpaces}.
\begin{center}
  \begin{minipage}{0.9\textwidth} 
  \centering
        \begin{tabular}{rl}
    \hline
   Space $L^p_Y$ &Formed from $\{Y(t)\}_{t \in [0,T]}$\\
    \hline\\[-3mm]
    $L^p_{L^q}$ &$\{L^q(\Gamma(t))\}_{t\in [0,T]}$    \\
    $L^2_{W^{1\slash 2, 2}}$ &$\{W^{1\slash 2, 2}(\Gamma(t))\}_{t\in [0,T]}$ 
    \end{tabular}\qquad\qquad
        \begin{tabular}{rl}
    \hline
   Space $L^p_Y$ &Formed from $\{Y(t)\}_{t \in [0,T]}$\\
    \hline\\[-3mm]
$L^2_{L^2(\C)}$ &$\{L^2(\C(t))\}_{t\in [0,T]}$  \\
$L^2_{H^1(\C)}$ &$\{H^1(\C(t))\}_{t\in [0,T]}$  \\
$L^2_{X(\C)}$ &$\{X(\C(t))\}_{t\in [0,T]}$  	
    \end{tabular}     
%    \begin{tabular}{rl}
%    \hline
%    Space &Formed from\\
%    \hline\\[-3mm]
%	$L^2_{L^2(\C_R)}$ &$\{L^2(\C_R(t))\}_{t\in [0,T]}$\\
%    $L^2_{H^1(\C_R)}$ &$\{H^1(\C_R(t))\}_{t\in [0,T]}$\\
%%	$L^2_{H^1_0(\C_R)}$ &$\{H^1_0(\C_R(t))\}_{t\in [0,T]}$
%    \end{tabular}
  \end{minipage}
\end{center}
In order to obtain measurability in time of $t \mapsto \ext{\mathcal{E}_t}(\pme(u(t)))$ for $u \in L^2_{W^{1\slash 2, 2}}$ (recall that each $\ext{\mathcal{E}}_t$ was defined individually at each moment in time as the harmonic extension on  $\Gamma(t)$), we will consider in \S \ref{sec:fractionalLaplacianOnL2X} the ``$L^2_{X(\C)}$ harmonic extension" problem: given $u \in L^2_{W^{1\slash 2, 2}}$, find $\ext{\mathbb{E}}u = v \in L^2_{X(\C)}$ such that
\begin{equation}\label{eq:introL2XHarmonicExtensionProblem}
\begin{aligned}
\Deltabg v&=0, %&&\text{in $(L^2_{H^1_0(\C)})^*$ },
 \quad \ext{\mathbb{T}}v = u&&\text{}
\end{aligned}
\end{equation}
holds with $\ext{\mathbb{T}}\colon L^2_{X(\C)} \to L^2_{W^{1\slash 2, 2}}$ the trace map. Then we will show that 
%\begin{equation}\label{eq:linkBetweenDifferentEs}
$(\ext{\mathbb{E}}u)(t) = \ext{\mathcal{E}}_tu(t)$ for almost all $t$, 
%\end{equation}
which gives the desired measurability. Of course, in the stationary setting, this issue of measurability would not arise and there would be no need to consider \eqref{eq:introL2XHarmonicExtensionProblem}.  Now we can think about what we mean by a weak solution. In what follows, given $\eta \in L^2_{W^{1\slash 2, 2}}$, we denote by $E\eta \in L^2_{H^1(\C)}$ an arbitrary extension of $\eta$ that satisfies $\ext{\mathbb{T}}{E\eta} = \eta$.
\begin{defn}[Weak solution]\label{defn:weakSolutionToFPME}A weak solution of \eqref{eq:localEquation2} is a function $u \in L^{\infty}_{L^\infty}$ with $\ext{\mathbb{E}}(\pme(u)) \in L^2_{X(\C)}$ satisfying 
\begin{equation*}\label{eq:fpmeWeakForm}
-\int_0^T \int_{\Gamma(t)}\dot \eta(t)u(t)\;\mathrm{d}\sigma_t\mathrm{d}t + \int_0^T \int_{\C(t)}\nablabgt \ext{\mathcal{E}}_t(\pme(u(t)))\nablabgt (E\eta)(t)\;\mathrm{d}\sigma_t\mathrm{d}t = \int_{\Gamma_0}u_0\eta(0)\;\mathrm{d}\sigma_0
\end{equation*}
for all $\eta \in \mathbb{W}(W^{1\slash 2, 2}, L^2)$ with $\eta(T)=0.$ Here, $\mathrm{d}\sigma_t$ means the surface measure on $\Gamma(t)$.
\end{defn}From now on, for brevity, we will omit the measures in any integrals. We will prove the following theorem in \S \ref{sec:fpme}, which is the main result of our paper. 
\begin{theorem}[Well-posedness of the fractional porous medium equation]\label{thm:wellPosednessOfFPME}
Under Assumption \ref{ass:onHypersurfaces} and \eqref{eq:eigenvalueEstimate}, given $u_0 \in L^\infty(\Gamma_0)$, there exists a unique weak solution $u \in L^\infty_{L^\infty} \cap L^2_{W^{-1\slash 2, 2}}$ to \eqref{eq:localEquation2} with $\ext{\mathbb{E}}(\pme(u)) \in L^2_{X(\C)}$ (in the sense of Definition \ref{defn:weakSolutionToFPME}).
%\begin{equation}\label{eq:fpmeWeakForm}
%-\int_0^T \int_{\Gamma(t)}\dot \eta(t)u(t) + \int_0^T \int_0^\infty  \int_{\Gamma(t)}\nablabgt \ext{\mathcal{E}}_t(\pme(u))\nablabgt E(t)\eta(t) = \int_{\Gamma_0}u_0\eta(0).
%\end{equation}
%for all $\eta \in W(H^{1\slash 2}, L^2)$ with $\eta(T)=0.$ 
Furthermore, we have the following properties:
\begin{enumerate}
\item Boundedness: for all $t \in [0,T]$, $u(t) \in L^\infty(\Gamma(t))$.
\item Conservation of mass: for all $t \in [0,T]$,
\[\int_{\Gamma(t)} u(t) = \int_{\Gamma_0}u_0.\]
\item $L^1$-contraction principle: if $u_{01}$ and $u_{02}$ are two pairs of initial data in $L^\infty(\Gamma_0)$, then the respective solutions $u_1$ and $u_2$ satisfy
\[\int_{\Gamma(t)}(u_1(t)-u_2(t))^+ \leq \int_{\Gamma_0}(u_{01}-u_{02})^+\quad\text{for all $t \in [0,T]$}.\]
\end{enumerate}
\end{theorem}
An immediate consequence of the contraction principle is the following.
\begin{cor}[$L^1$-continuous dependence and comparison principle] If $u_{01}$ and $u_{02}$ are two pairs of initial data in $L^\infty(\Gamma_0)$, then the respective weak solutions $u_1$ and $u_2$ of Theorem \ref{thm:wellPosednessOfFPME} satisfy the $L^1$-continuous dependence result
\[\int_{\Gamma(t)}|u_1(t)-u_2(t)| \leq \int_{\Gamma_0}|u_{01}-u_{02}|\quad\text{for all $t \in [0,T]$}.\]
If $u_{01} \leq u_{02}$ a.e., then $u_1(t) \leq u_2(t)$ a.e. in $\Gamma(t)$ for all $t$.
\end{cor}
Let us discuss how these results compare to those in the stationary case considered in \cite{DePablo2011, Pablo2012}. Theorem \ref{thm:wellPosednessOfFPME} and its corollary correspond to parts i, ii, iv and v of Theorem 2.2 of \cite{DePablo2011} and to Theorem 7.2 of \cite{Pablo2012} in the half-Laplacian setting. In terms of the proof, our methods are quite different, as already discussed earlier. Let us sketch the proof now. 
\subsection{Plan of the proof}
In order to solve \eqref{eq:localEquation2} and prove Theorem \ref{thm:wellPosednessOfFPME}, we will first approximate the nonlinearity $\pme$ by well-behaved smooth approximations $\pme_k$ and seek to solve \eqref{eq:localEquation2} with $\pme$ replaced by $\pme_k$. This directs us to study the non-degenerate problem 
%\begin{equation}\label{eq:fullApproximatedPDE1}
%\begin{aligned}
%\Deltabg v_k(t) &= 0 &&\text{on $\C(t)$}\\
%v_k(t)|_{\Gamma(t) \times \{0\}} &= \varphi_k(u_k(t))&&\text{on $\partial \C(t)$}\\
%\dot u_k + u_k \sgrad \cdot \mathbf w + \frac{\partial v_k}{\partial \nu} &= 0&&\text{on $\partial \C(t)$}
%\end{aligned}.
%\end{equation}
\begin{equation}\label{eq:betaProblem}
\begin{aligned}
\dot u_\beta(t) + u_\beta(t) \sgradt \cdot \mathbf w(t) + \frac{\partial v_\beta(t)}{\partial \nu(t)} &= 0&&\text{on $\partial \C(t)$}\\
v_\beta(t) &= \ext{\mathcal{E}}_t(\beta(u_\beta(t)))\\
u_\beta(0) &= u_0 &&\text{on $\Gamma_0$}
\end{aligned}\tag{$\textbf{P}_{\beta}$}
\end{equation}
where $\beta\colon \mathbb{R} \to \mathbb{R}$ satisfies
\begin{equation}\label{eq:assumptionsOnBeta}
\begin{aligned}
&\beta(0) = 0, \beta \text{ is } C^2(\mathbb{R}) \text{ (and Lipschitz)}\\
%&\beta(0) = 0,\\
&\beta', (\beta^{-1})', (\beta^{-1})'' \in L^\infty(\mathbb{R}), \text{ and}\\
&\text{there exist constants $C_{\beta'}, C_{\beta'_{inv}} > 0$ with } \beta' \geq C_{\beta'} \text{ and } (\beta^{-1})' \geq C_{\beta'_{inv}}.
\end{aligned}\tag{$A_\beta$}
\end{equation}
To show well-posedness of \eqref{eq:betaProblem} one could try a Galerkin method but a complication involving the unbounded cylinder $\C(t)$ arises due to the surface evolution, see Remark \ref{rem:needForTruncationOfCylinder}; this suggests truncating the cylinder $\C(t)$ in the unbounded direction. %to $\C_R(t) := \Gamma(t) \times [0,R]$. 
So we consider in \S \ref{sec:truncatedHarmonicExtension} a truncated harmonic extension problem and show that its solution approximates the (untruncated) harmonic extension in some sense: given $u \in H(M)$, with $\ext{\mathcal{E}}_{R}u=v_R$ denoting the weak solution of
\begin{equation}\label{eq:prelim2}
\begin{aligned}
\Deltabg v_R &=0&&\text{on $\C_R:=M \times [0,R]$}, \qquad v_R|_{M \times \{0\}} = u, \qquad
v_R|_{M \times \{R\}} = 0,
\end{aligned}
\end{equation}
we will show in \S \ref{sec:convergenceOfTruncatedSolution} that $\nablabg \ext{\mathcal{E}}_{R}u \to \nablabg \ext{\mathcal{E}}u$ in $L^2(\C)$ as $R \to \infty$. As with $\mathcal{E}_t$, we define $\ext{\mathcal{E}}_{R,t}$ as $\ext{\mathcal{E}}_{R}$ with $M=\Gamma(t)$ and $\C_R(t) := \Gamma(t) \times [0,R]$, and consider the following problem as an approximation of \eqref{eq:betaProblem}:
%\begin{equation}\label{eq:betaRProblem}
%\begin{aligned}
%\Deltabg v_{k,R}(t) &= 0 &&\text{on $\C_R(t):=\Gamma(t)\times [0,R]$}\\
%v_{k,R}(t)|_{\Gamma(t) \times \{0\}} &= \varphi_k(u_{k,R}(t))\\
%v_{k,R}(t)|_{\Gamma(t) \times \{R\}} &= 0\\
%\dot u_{k,R} + u_{k,R} \sgrad \cdot \mathbf w + \frac{\partial v_{k,R}}{\partial \nu} &= 0&&\text{on $\partial \C_R(t)$}
%\end{aligned}
%\end{equation}
\begin{equation}\label{eq:betaRProblem}
\begin{aligned}
\dot u_{\beta R}(t) + u_{\beta R}(t) \sgradt \cdot \mathbf w(t) + \frac{\partial v_{\beta R}(t)}{\partial \nu(t)} &= 0&&\text{on $\Gamma(t) \times \{0\}$}\\
v_{\beta R}(t) &= \ext{\mathcal{E}}_{R,t}(\beta(u_{\beta R}(t)))\\
u_{\beta R}(0) &= u_0 &&\text{on $\Gamma_0$}.
\end{aligned}\tag{$\textbf{P}_{\beta R}$}
\end{equation}
We can define the spaces $L^2_{L^2(\C_R)}$ and $L^2_{H^1(\C_R)}$ on the truncated cylinder just like before, and consideration of an ``$L^2_{H^1(\C_R)}$ truncated harmonic extension" problem like \eqref{eq:introL2XHarmonicExtensionProblem} in \S \ref{sec:fractionalLaplacianOnL2X} will lead to a map $\ext{\mathbb{E}}_{R}$ and show the measurability in time of $\ext{\mathcal{E}}_{R,t}$. We will use the Galerkin method to solve \eqref{eq:betaRProblem} in \S \ref{sec:truncatedProblem}, see Remark \ref{rem:projectionOperator} where we explain the choice of our Galerkin approximation; this requires emphasis due to a technical difficulty in the evolution-dependent projection operators associated to the Galerkin basis. Then we will pass to the limit in $R$ in \S \ref{sec:passToLimitInR} in order to settle \eqref{eq:betaProblem} and %In fact, we work in a slightly more general setting in \S \ref{sec:nondegenerateProblem} where the specific nonlinearity $\pme_k$ is replaced with a function $\beta\colon \mathbb{R} \to \mathbb{R}$ satisfying
%\begin{equation}\label{eq:assumptionsOnBeta}
%\begin{aligned}
%&\beta \text{ is } C^2(\mathbb{R}) \text{ and Lipschitz}\\
%&\beta(0) = 0,\\
%&\beta', (\beta^{-1})', (\beta^{-1})'' \in L^\infty(\mathbb{R}), \text{ and}\\
%&\text{there exist constants $C_{\beta'}, C_{\beta'_{inv}} > 0$ with } \beta' \geq C_{\beta'} %\text{ and } (\beta^{-1})' \geq C_{\beta'_{inv}},% ,\\
%%(\beta^{-1})' &\geq C_{\beta'_{inv}} 
%\end{aligned}\tag{$A_\beta$}
%\end{equation}
the following theorem will be proved.
%for constants $C_{\beta'}$ and $C_{\beta'_{inv}}.$
\begin{theorem}\label{thm:wellPosednessOfNonDegenerateProblem}Under Assumption \ref{ass:onHypersurfaces}, \eqref{eq:eigenvalueEstimate}, and \eqref{eq:assumptionsOnBeta}, given $u_0 \in L^\infty(\Gamma_0)$, there exists a unique solution $u_\beta \in \mathbb{W}(W^{1\slash 2, 2}, W^{-1\slash 2, 2})$ to \eqref{eq:betaProblem} with $u_\beta(0)=u_0$ and $\ext{\mathbb{E}}(\beta(u_\beta)) \in L^2_{X(\C)}$ satisfying 
\begin{equation}\label{eq:nondegenerateProblem}
\int_0^T\langle \dot u_\beta(t), \eta(t)\rangle + \int_0^T\int_{\Gamma(t)}u_\beta(t)\eta(t) \sgradt \cdot \mathbf w(t) + \int_0^T\int_{\C(t)}\nablabgt \ext{\mathcal{E}}_t(\beta(u_\beta(t)))\nablabgt (E\eta)(t) = 0
\end{equation}
for all $\eta \in L^2_{W^{1\slash 2, 2}}$, where the duality pairing is between $W^{-1\slash 2, 2}(\Gamma(t))$ and $W^{1\slash 2, 2}(\Gamma(t))$. Furthermore, mass is conserved and the $L^1$-contraction principle holds for almost all $t \in [0,T]$. %If $u_{01}$ and $u_{02}$ are two pairs of initial data in $L^\infty(\Gamma_0)$, then the respective solutions $u_1$ and $u_2$ satisfy the $L^1$-contraction property
%\[\int_{\Gamma(t)}(u_1(t)-u_2(t))^+ \leq \int_{\Gamma_0}(u_{01}-u_{02})^+\quad\text{for almost all $t \in [0,T]$}.\]
\end{theorem}
%In this theorem, $E\eta \in L^2_{H^1(\C)}$ has trace $\mathbb{T}E\eta = \eta$.
With $\beta$ chosen to be the regularisation $\pme_k$, this theorem gives us a sequence $\{u_k\}_{k \in \mathbb{N}}$ where $u_k \in \mathbb{W}(W^{1\slash 2, 2}, W^{-1\slash 2, 2})$ satisfies $u_k(0)=u_0$, $\ext{\mathbb{E}}(\pme_k(u_k)) \in L^2_{X(\C)}$, and the equation \eqref{eq:nondegenerateProblem} with $\beta$ replaced by $\pme_k$ and $u_\beta$ replaced by $u_k$.
%\begin{equation}
%\int_0^T\langle \dot u_k(t), \eta(t)\rangle_{W^{-1\slash 2, 2}(\Gamma(t)), W^{1\slash 2, 2}(\Gamma(t))} + \int_0^T\int_{\Gamma(t)}u_k(t)\eta(t) \sgradt \cdot \mathbf w(t) + \int_0^T\int_{\C(t)}\nablabgt \ext{\mathcal{E}}_t(\pme_k(u_k(t)))\nablabgt (E\eta)(t) = 0
%\end{equation}
%for all $\eta \in L^2_{W^{1\slash 2, 2}}$.
%This will give us a sequence of solutions $\{u_k\}_{k}$ of \eqref{eq:betaProblem} where each $u_k$ is associated to the regularisation $\beta= \pme_k$. 
Then we pass to the limit in $k$ using energy estimates and the identification of limits is handled with the theory of subdifferentials of convex functionals in \S \ref{sec:fpme} where the proof of Theorem \ref{thm:wellPosednessOfFPME} is concluded.

In \cite{DePablo2011, Pablo2012}, the authors prove results for existence with integrable data too, as well as other properties besides, including regularity, smoothing effects and extinction of solutions. As the next step to our results, studying regularity in time would be natural (and useful) but it appears difficult in our setting. We comment on this in more detail in the conclusion.
%\paragraph{Survey of literature}Let us now discuss what else is present in the literature. 
\subsection{Outline}It is clear that we need to properly study the harmonic extension maps $\ext{\mathcal{E}}_t$ and $\ext{\mathcal{E}}_{R,t}$, which we take care of in \S\ref{sec:fractionalLaplacianOnCompactManifolds} in the general setting of closed Riemannian manifolds. 
In \S \ref{sec:functionSpaces} we shall check that the spaces $L^p_Y$ listed above are well-defined and prove some preliminary functional analytic results. We then study the maps  $\ext{\mathbb{E}}$ and $\ext{\mathbb{E}}_R$ in \S \ref{sec:fractionalLaplacianOnL2X}. After this, we tackle the non-degenerate problem \eqref{eq:betaProblem} in \S \ref{sec:nondegenerateProblem} and then prove the main theorem in \S \ref{sec:fpme}. We will finish with some concluding remarks in \S \ref{sec:conclusions}. Let us emphasise that \S \ref{sec:fractionalLaplacianOnCompactManifolds} is useful more generally for fractional problems on closed manifolds and \S \ref{sec:functionSpaces}--\ref{sec:nondegenerateProblem} are useful for fractional diffusion problems on (evolving) hypersurfaces. Only in \S \ref{sec:fpme} do we specialise to the porous medium equation.

\subsection{Notation}
We use the overline $\bar{\cdot}$ in different contexts. When applied to functions $u$, it means the spatial mean value: typically $\mean u = \frac{1}{|M|}\int_{M}u$ or $\mean u= \frac{1}{|\Gamma(t)|}\int_{\Gamma(t)}u$. When applied to symbols like $\mathbb{E}$ or $\mathcal{E}$ like $\ext{\mathbb{E}}$ or $\ext{\mathcal{E}}$, the meaning usually is that the map with the overline is a linear extension, for example, $\ext{\mathcal{E}}$ is a linear extension of $\mathcal{E}$ to a larger space. Symbols of the blackboard bold style like $\mathbb{E}$ refer to maps between the evolving Bochner spaces $L^2_Y$, whilst symbols of the calliographic style like $\mathcal{E}$ refer to maps between Sobolev spaces of the form $H^s(M)$. The notation $\seminorm{\cdot}{}$ denotes a seminorm; usually the $L^2$ part of the corresponding norm is omitted.% from the seminorm.

As a convenience for the reader, we give here a list of the major notations and symbols that we use in this paper along with the page number of definition or first usage.
\begin{center}
  \begin{minipage}{0.9\textwidth}
  \centering
    \begin{tabular}{lr}
    	\hline
       Notation & Page \\
       \hline \\[-3mm]
	\eqref{ass:onM} &p.~\pageref{ass:onM}\\
	$H(M)$ &p.~\pageref{eq:spaceHM} \\
		$X(\C)$ &p.~\pageref{eq:spaceXC} \\
	    \eqref{eq:eigenvalueEstimate} & p.~\pageref{eq:eigenvalueEstimate}\\    
	$\mathcal{E}$, $\ext{\mathcal{E}}$ &p.~\pageref{thm:harmonicExtensionExistence} \\
%	$H^1_0(\C_R)$ &p.~\pageref{defn:H10}\\
	$\mathcal{E}_R$, $\ext{\mathcal{E}}_R$ &p.~\pageref{thm:truncatedHarmonicExtensionExistence} \\
	$\mathscr{Z}_R$ &p.~\pageref{defn:Z} \\
    \end{tabular}\qquad\qquad
        \begin{tabular}{lr}
    	\hline
       Notation & Page \\
       \hline \\[-3mm]     
    $\mathcal{E}_t$, $\mathcal{E}_{R,t}$, $\ext{\mathcal{E}}_t$, $\ext{\mathcal{E}}_{R,t}$ &p.~\pageref{defn:harmonicExtensionsWitht} \\
         $\mathcal{T}_t$, $\mathcal{T}_{R,t,y=0}$, $\mathcal{T}_{R,t,y=R}$, $\ext{\mathcal{T}}_t$ &p.~\pageref{defn:harmonicExtensionsWitht} \\
    $\mathbb{T}$,   $\ext{\mathbb{T}}$ &p.~\pageref{lem:existenceOfTraceMapOnL2H1AndL2X} \\
	$\mathbb{T}_{R,y=0}$, $\mathbb{T}_{R,y=R}$ &p.~\pageref{lem:existenceTraceMapsTruncatedL2H1CR} \\
    $\mathbb{E}$, $\ext{\mathbb{E}}$ &p.~\pageref{thm:wellPosednessHarmonicExtensionL2X} \\
    $\mathbb{E}_R$, $\ext{\mathbb{E}}_R$ &p.~\pageref{thm:wellPosednessTruncatedHarmonicExtensionL2X} \\
    $\langle \cdot, \cdot \rangle_{}$, $\langle \cdot, \cdot \rangle_{0}$ &p.\pageref{notation:dualityProducts}
    \end{tabular}
  \end{minipage}
\end{center}
\section{The fractional Laplacian on compact Riemannian manifolds}\label{sec:fractionalLaplacianOnCompactManifolds}
Throughout this section, we assume that $(M,g)$ is a Riemannian manifold as given in \eqref{ass:onM}. One aim of this section is to realise the fractional Laplacian on a closed Riemannian manifold as the Dirichlet-to-Neumann map of a harmonic extension problem in a Sobolev space setting.  We will define an operator $\ext{\mathcal{E}}\colon H(M) \to X(\C)$ for this purpose. We also study the truncated harmonic extension by means of an operator $\ext{\mathcal{E}}_R\colon H(M) \to H^1(\C_R)$, and then prove that $\mathcal{E}_R$ approximates $\mathcal{E}$. 
\begin{remark}
We do not consider the case where $M$ is an open manifold (i.e., a manifold with boundary). If $\partial M \neq \emptyset$ and we place Neumann boundary conditions then most of what follows should be similar, since the eigenvalues and eigenfunctions behave similarly to the closed setting. If instead Dirichlet boundary conditions are taken then an analogue of the following results will hold; in particular one probably would not need to worry about differentiating between functions with mean value zero and those without, and the space $H(M) = (L^2(M), W^{1,2}_0(M))_{1\slash 2}$ will be rather different. %In either case the validity of the Poincar\'e inequality would need to be checked. 
\end{remark}
We will often be integrating or manipulating infinite series of functions term by term which can be justified by Abel's test or the Weierstrass M-test. More details of this and lengthier calculations of what follows can be found in \cite{Thesis}. First, we begin with a brief discussion of Sobolev spaces on (semi-infinite) cylinders.

\subsection{Sobolev spaces on semi-infinite cylinders}\label{sec:sobolevSpacesOnSemiInfiniteCylinders}
%Let $(M,g)$ be a closed smooth Riemannian manifold. 
%It appears difficult to find literature where Sobolev spaces on semi-infinite cylinders of the form $\C:= M \times [0,\infty)$ are treated. 
We can use the space $H^1(\C)$ (utilised already in the introduction) defined in \cite{Amann2012} as the linear subspace of $L^1_{\text{loc}}(\C)$ consisting of all $v$ such that $v$ and $\nablabg v$ belong to $L^2(\C)$, and endowed with the natural norm. Equivalently, it can be defined as the linear subspace of $L^2(0,\infty;H^1(M))$ consisting of all $v$ such that $v_y \in L^2(0,\infty;L^2(M))$. This is precisely the type of Sobolev--Bochner space whose theory was developed by Lions and Magenes \cite[Chapter 1, \S 2.2]{LionsMagenes}. There is a bounded linear surjective trace operator $\mathcal{T} \colon H^1(\C) \to H^{1\slash 2}(M \times \{0\})$ %(where we identify $\partial\C = M \times \{0\}$ with $M$ itself), 
 \cite[Theorem 18.1]{Amann2012}, \cite[Theorem 3.2, Chapter 1]{LionsMagenes}, possessing a continuous right inverse. Similarly, the spaces $H^1(\C_R)$ can be defined on the truncated cylinder $\C_R=M \times [0,R]$. Theorem 3.1 of \cite[Chapter 1]{LionsMagenes} gives $H^1(\C_R) \cts C^0([0,R];H^{1\slash 2}(M)),$ so that the linear trace operators $\mathcal{T}_{R,y=0}$, $\mathcal{T}_{R,y=R}  \colon H^1(\C_R) \to H^{1\slash 2}(M)$ defined by $(\mathcal{T}_{R,y=0}v)(\cdot) := v(\cdot,0)$ and $(\mathcal{T}_{R,y=R}v)(\cdot) := v(\cdot,R)$ are also bounded. Furthermore $\mathcal{T}_{R,y=0}$ is surjective \cite[Theorem 3.2, Chapter 1]{LionsMagenes}.
\begin{lem}\label{lem:meanValueContinuous}
If $v \in H^1(\C)$, then $y \mapsto \mean v(y) = \frac{1}{|M|}\int_M v(y)$ is an element of $H^1(0,\infty)$ and thus $\mean v \in C^0([0,\infty)).$
\end{lem}
\begin{proof}
A calculation verifies that $\mean v \in H^1(0,\infty)$, and Theorem 8.2 in \cite{brezis2010functional} proves that each function in $H^1(0,\infty)$ has a unique continuous representative in $C^0([0,\infty))$.
\end{proof}
\subsection{Fractional Sobolev spaces and the fractional Laplacian}
%Let $(M,g)$ be a connected closed smooth oriented Riemannian manifold. 
The setting of a closed manifold is similar to the setting of Neumann boundary conditions on a bounded domain (see \cite{Stinga2014, Montefusco2012, StingaTorrea}), and now we motivate the definition of the half-Laplacian like \cite[\S 2]{Stinga2014}. As mentioned in the introduction, let $(\lambda_k, \varphi_k)$ be the normalised eigenelements of the Laplace--Beltrami operator $-\Delta_M$. For $k \neq 0$, since $(\varphi_k, \varphi_0)_{L^2(M)} = 0$, $\mean{\varphi_k}=0$. We also have $\norm{\varphi_k}{H^1(M)}^2 = 1+\lambda_k$ which implies that %. If $u \in H^1(M)$ is $u=\sum u_j \varphi_j$ where $u_j = (u, \varphi_j)_{L^2(M)}$, then 
%\begin{align*}
%\norm{u}{H^1(M)}^2 &= \sum_j \sum_k u_ju_k(\varphi_j, \varphi_k)_{H^1(M)} = \sum_k |u_k|^2(1+\lambda_k)
%\end{align*}
%by orthogonality of the eigenfunctions in $H^1(M)$. Therefore,
\begin{equation*}\label{eq:defnH1}
H^1(M) = \left\{ u \in L^2(M) \mid \norm{u}{H^1(M)}^2 = 	\sum_{k=0}^\infty (1+\lambda_k)|(u, \varphi_k)_{L^2(M)}|^2 < \infty \right\},
\end{equation*}
and for $u \in H^1(M)$, one has
\begin{equation*}\label{eq:Deltau}
-\Delta_M u = \sum_{k=1}^\infty \lambda_k (u,\varphi_k)_{L^2(M)}\varphi_k\quad\text{in $H^{-1}(M)$}
\end{equation*}
with
\begin{equation}\label{eq:actionH1}
\langle -\Delta_M u, v \rangle_{H^{-1}(M), H^1(M)} = \sum_{k=1}^\infty \lambda_k (u,\varphi_k)_{L^2(M)}(v,\varphi_k)_{L^2(M)}.
\end{equation}
With the Hilbert space $H(M)$ as in \eqref{eq:spaceHM}, the previous two identities inspire us to 
%\begin{equation}
%H(M) = \{ u \in L^2(M) \mid \sum_{k=1}^\infty \lambda_k^{1\slash 2}|(u,\varphi_k)_{L^2(M)}|^2 < \infty\}
%\end{equation}
%with the inner product
%\[(u,v)_{H(M)} = (u,v)_{L^2(M)} + \sum_{k=1}^\infty \lambda_k^{1\slash 2}(u,\varphi_k)_{L^2(M)}%(v,\varphi_k)_{L^2(M)}.\]
%it is natural to 
define $(-\Delta_M)^{1\slash 2}\colon H(M) \to H(M)^*$ by \eqref{eq:defnFractionalLaplacian} %:
%\[(-\Delta_M)^{1\slash 2}u := \sum_{k=1}^\infty \lambda_k^{1\slash 2}(u,\varphi_k)\varphi_k\]
with the action \eqref{eq:spectralH12Action}.
%Notice that $\ker(-\Delta_M)^{1\slash 2}$ contains $\mathbb{R}$. 
For $u$, $v \in H(M)$, it is easy to see the integration by parts formula
\[\langle (-\Delta_M)^{1\slash 2}u, v \rangle_{H(M)^*, H(M)} = \int_M (-\Delta_M)^{1\slash 4}u(-\Delta_M)^{1\slash 4}v\]
where %$(-\Delta_M)^{1\slash 4}u := \sum_{k=1}^\infty \lambda_k^{1\slash 4}(u,\varphi_k)_{L^2(M)}\varphi_k$
%with the action
$\langle (-\Delta_M)^{1\slash 4}u, v \rangle := \sum_{k=1}^\infty \lambda_k^{1\slash 4}(u,\varphi_k)_{L^2(M)}(v,\varphi_k)_{L^2(M)},$ and we have
\[\norm{(-\Delta_M)^{1\slash 4}u}{L^2(M)}^2 = \sum_{k=1}^\infty \lambda_k^{1\slash 2}|(u,\varphi_k)_{L^2(M)}|^2 = |u|_{H(M)}^2.\]
\subsection{The harmonic extension problem}
%Recall the problem \eqref{eq:prelim1}: given $u \in H(M)$, we want to find $v\colon \C \to \mathbb{R}$ satisfying
%\begin{equation}\label{eq:pointwiseExtensionProblem}
%\begin{aligned}
%\Delta_{\bar g}v &=0&&\text{on $\C$}\\
%v(\cdot,0) &= u(\cdot)
%\end{aligned}
%\end{equation}
%in the weak sense:
%\begin{equation}\label{eq:harmonicExtensionWeakSense}
%\begin{aligned}
%\int_\C \nablabg v \nablabg \eta &=0 &&\text{for all $\eta \in H^1(\C)$ with $\mathcal{T}\eta = 0$}\\
%\lim_{y \to 0^+}v(\cdot,y) &= u(\cdot) &&\text{in %$L^2(M)$}.
%end{aligned}
%\end{equation}
Recall the problem \eqref{eq:prelim1}. If $u \equiv 1$, then its harmonic extension is $v \equiv 1$, so $u \mapsto v$ does not map into $H^1(\C)$ in general. Therefore, we will work in the bigger space $X(\C)$, defined in \eqref{eq:spaceXC}. 
%\begin{defn}
%Define $X(\C) = \overline{H^1(\C)}^{\norm{\cdot}{\tilde X(\C)}}$ where
%\[\norm{v}{\tilde X(\C)}^2 = \norm{\nablabg v}{L^2(\C)}^2+ \norm{\mathcal{T} v}{L^2(M)}^2\quad\text{for $v \in H^1(\C)$}.\]
%\end{defn}
%Note that $H^1(\C) \ctsDense X(\C)$. 
\begin{remark}\label{rem:constants}
The constant functions belong to $X(\C)$. To see this, take $c \in \mathbb{R}$ and $c_n \in H^1(\C)$ with
\begin{align*}
c_n(x,y) = \begin{cases}
c &: y \in (0,n]\\
\frac{c}{n}(2n-y) &: y \in (n, 2n]\\
0 &: y \in (2n, \infty)
\end{cases}
\end{align*}
which satisfies $\nabla_M c_n = 0$ and $\partial_y c_n = -c/ n \chi_{(n, 2n)}(y)$. Note that 
\begin{align*}
\norm{c_n-c_m}{ X(\C)}^2 %&= \norm{\partial_y(u_n-u_m)}{L^2(\C)}^2 + \norm{\mathcal{T}(u_n-u_m)}{L^2(M)}^2\\
%&= \norm{\partial_y(u_n-u_m)}{L^2(\C)}^2\\
&\leq 2\left(\norm{\partial_yc_n}{L^2(\C)}^2 + \norm{\partial_yc_m}{L^2(\C)}^2\right)= 2\left(\int_n^{2n}\int_M \frac{c^2}{n^2} + \int_m^{2m}\int_M \frac{c^2}{m^2}\right)= 2|M|\left(\frac{c^2}{n} + \frac{c^2}{m}\right),
\end{align*}
so $\mathbf c := (c_n)$ is a Cauchy sequence in the $ X(\C)$ norm, and it follows that 
\begin{align*}
\norm{\mathbf c}{X(\C)}^2 := \lim_{n \to \infty} \norm{c_n}{ X(\C)}^2 = \lim_{n \to \infty}\int_n^{2n} \int_M \frac{c^2}{n^2} + \int_M c^2 = \lim_{n \to \infty} |M|\left(\frac{c^2}{n}+c^2\right) = |M|c^2.
\end{align*}
Then $\mathbf c$ can be identified with the constant $c$.
%Consider the constant sequence $c:= (c).$ We see that 
%\[\norm{u_n-c}{ X(\C)} = \norm{\partial_y (u_n-c)}{L^2(C)}^2 = \int_n^{2n}\int_M c^2/n^2  = |M|c^2/n\] which tends to zero as $n \to \infty,$ and this shows that $u \sim c$.
\end{remark}	
\begin{lem}[Extension of the gradient to $X(\C)$]The gradient $\nablabg \colon H^1(\C) \to L^2(\C)$ extends to a bounded linear map $\ext{\nablabg} \colon X(\C) \to L^2(\C)$ which satisfies $\ext{\nablabg}|_{H^1(\C)} = \nablabg$ and $\ext{\nablabg}v = \lim_{n \to \infty}\nablabg v_n$ for $v_n \in H^1(\C)$ such that $v_n \to v$ in $X(\C)$.
\end{lem}
\begin{proof}
Clearly $\nablabg \colon H^1(\C) \to L^2(\C)$ satisfies $\norm{\nablabg v}{L^2(\C)} \leq \norm{v}{X(\C)}$ for all $v \in H^1(\C)$. Since $H^1(\C)$ is dense in $X(\C)$, the bounded linear transformation (BLT) theorem provides the result.
\end{proof}
\begin{theorem}\label{thm:harmonicExtensionExistence}
For every $u \in H(M)$, there exists a unique weak solution $\ext{\mathcal{E}}u = v \in X(\C)$ to the harmonic extension problem \eqref{eq:prelim1} satisfying $(\ext{\mathcal{E}}u)(\cdot,0) = u(\cdot)$ in $L^2(M)$ and
\[\int_{\C}\ext{\nablabg} (\ext{\mathcal{E}}u) \nablabg \eta= 0\quad\text{for all $\eta \in H^1(\C)$ with $\mathcal{T}\eta = 0$}.\]
When $\mean u=0$, we write the solution as $\mathcal{E}u$ which is such that $\frac{1}{|M|}\int_{M}(\mathcal{E}u)(y) =0$ for all $y \in [0,\infty)$. The map $\ext{\mathcal{E}}\colon H(M) \to X(\C)$ satisfies $\ext{\mathcal{E}}u = \mathcal{E}(u-\mean u) + \mean u$ and $\ext{\nablabg}(\ext{\mathcal{E}}u)  = \nablabg \mathcal{E}(u-\mean u)$. Furthermore (if $\mean u=0$), $\mathcal{E}u \in C^0([0,\infty);L^2(M)) \cap C^\infty((0,\infty);H^1(M))$ and
\begin{align}
%\ext{\nablabg}(\ext{\mathcal{E}}u)  &= \nablabg \mathcal{E}(u-\mean u)\\
\norm{\mathcal{E}u}{L^2(\C)}^2 &\leq \frac{\norm{u}{L^2(M)}^2}{2\lambda_1^{1\slash 2}}\label{eq:harmonicExtensionL2Bound},\\
\norm{\nablabg\mathcal{E}u}{L^2(\C)}^2 &= \norm{(-\Delta_M)^{1\slash 4}u}{L^2(M)}^2 = |u|^2_{H(M)}
\label{eq:harmonicExtensionGradientBound}.
\end{align}
Finally, the harmonic extension $\mathcal{E}u$ (for $\mean{u}= 0$) is the unique minimiser of the energy
\[\mathcal{J}\colon \{ v \in H^1(\C) \mid \mathcal{T}v = u\} \to \mathbb{R}, \quad \mathcal{J}(v) := \frac{1}{2}\int_\C |\nablabg v|^2.\]
%over the set $H^1_u(\C) := \{ v \in H^1(\C) \mid \mathcal{T}v = u\}$. 
%\begin{align*}
%\int_\C \nabla_M v \nabla_M \eta + \int_\C v_y\eta_y &=0\\
%\lim_{y \to 0^+}v(x,y) &= u(x)
%\end{align*}
%for all $\eta \in H^1(\C)$ with $\mathcal{T}\eta = 0$.
\end{theorem}
\begin{proof}
The proof of the well-posedness is essentially the same as that of Theorem 2.1 in \cite{Stinga2014}.
Suppose for now that $\mean{u}= 0$. Set 
\[(\mathcal{E}u)(y) := v(y) := \sum_{k=1}^\infty e^{-y\lambda_k^{1\slash 2}}(u,\varphi_k)_{L^2(M)}\varphi_k,\] which is a sum that converges in $L^2(M)$ for each fixed $y \in [0,\infty)$. It satisfies
\begin{align*}
\int_M |v(y)|^2 &= \sum_{k\geq 1} e^{-2y\lambda_k^{1\slash 2}}|(u,\varphi_k)_{L^2(M)}|^2\quad\text{and}\quad\int_{M}|\nablabg v(y)|^2 = 2\sum_{k\geq 1} \lambda_ke^{-2y\lambda^{1\slash 2}}|(u,\varphi_k)_{L^2(M)}|^2
\end{align*}
where we used \eqref{eq:actionH1}.
%\begin{align*}
%\int_{M} |v_y(y)|^2 &= \int_M |\sum_{k\geq 1} -\lambda_k^{1\slash 2}e^{-y\lambda_k^{1\slash 2}}(u,\varphi_k)_{L^2(M)}\varphi_k|^2 = \sum_{k\geq 1} \lambda_ke^{-2y\lambda_k^{1\slash 2}}|(u,\varphi_k)_{L^2(M)}|^2\quad\text{and}\\
%\end{align*}
These expressions can be integrated over $y$ term by term since the sums converge uniformly, and doing so leads to properties \eqref{eq:harmonicExtensionL2Bound} and \eqref{eq:harmonicExtensionGradientBound}. %Since the partial sums $v_N(y) = \sum_{k=1}^Ne^{-y\lambda_k^{1\slash 2}}(u,\varphi_k)_{L^2(M)}\varphi_k$ converge in $L^2(M)$ to $v(y)$, we have that $0=(v_N(y),1)_{L^2(M)} \to (v(y),1)_{L^2(M)}$, 
%giving that $v(y)$ has mean value zero.

Now let $\eta \in H^1(\C)$ with $\mathcal{T} \eta =0$. For almost all $y$, $\eta(y) = \sum_{k=0}^\infty (\eta(\cdot,y),\varphi_k)_{L^2(M)}\varphi_k$; let us write the coefficients as $\eta_k(y)$. We see that
%\[\int_M v_{yy}(y)\eta(y) = \sum_{k=1}^\infty \lambda_k e^{-y\lambda_k^{1\slash 2}}(u,\varphi_k)_{L^2(M)}\eta_k(y)\]
%and
\begin{equation*}\label{eq:2}
\int_M \nabla_M v(y) \nabla_M \eta(y) 
= \sum_{k=1}^\infty \lambda_ke^{-y\lambda_k^{1\slash 2}}(u,\varphi_k)_{L^2(M)}\eta_k(y) = \int_M v_{yy}(y)\eta(y)
\end{equation*}
using \eqref{eq:actionH1}. By integrating by parts %$\int_0^\infty v_{yy}\eta = -\int_0^\infty v_y \eta_y + [v_y\psi]^{y=\infty}_{y=0}$ 
we get
\begin{align*}
\int_{\C}\nabla_M v \nabla_M \eta + v_y \eta_y &= \int_{\C}\nabla_M v \nabla_M \eta - v_{yy} \eta + \int_{\partial\C}v_y\eta%\\
%&=\int_{M}v_y(x,0)\eta(x,0)\tag{using \eqref{eq:2}}\\
= 0
\end{align*}
as $\eta$ has zero trace. This proves that $v$ is a weak solution. Uniqueness follows by taking the difference of the weak formulations satisfied by two solutions and testing with the difference of the two solutions (which has trace zero). Therefore, the map $\mathcal{E}\colon \{u \in H(M) \mid \mean{u} = 0\} \to H^1(\C)$ is well-defined. Now suppose $\mean{u}\neq 0$. Define 
\[\ext{\mathcal{E}}(u) := \mathcal{E}(u-\mean u) + \mean u.\] 
Note that $\ext{\nablabg}(\ext{\mathcal{E}}u) = \nablabg \mathcal{E}(u-\mean u) + \ext{\nablabg}\mean u$ by linearity and the fact that $\mathcal{E}(u-\mean u) \in H^1(\C)$. %We have $\ext{\nablabg}\mean{u}= \lim_{n \to \infty}\nablabg u_n$ where $u_n \in H^1(\C)$ converges to $\mean u$ in $X(\C)$, which is a constant. 
Let us choose $u_n=c_n \in H^1(\C)$ as in Remark \ref{rem:constants}, which tells us that $\lim_{n \to \infty}\nablabg u_n = \lim_{n \to \infty}-\frac{\mean u}{n}\chi_{(n, 2n)}(y) = 0$  in $L^2(\C)$, i.e., $\nablabg u = 0$. This proves that $\ext{\nablabg}(\ext{\mathcal{E}}u)  = \nablabg \mathcal{E}(u-\mean u).$

For the minimisation property, take $w \in H^1(\C)$ with $\mathcal{T}w=u$, test the weak form $\mathcal{E}u = v$ satisfies with $\eta = v-w$ and use Young's inequality:
%\[\int_\C |\nablabg v|^2  - \nablabg v \nablabg w = 0\]
%which gives
\[%2\mathcal J(v) = 
\norm{\nablabg v}{L^2(\C)}^2 \leq \frac{1}{2}\norm{\nablabg v}{L^2(\C)}^2 + \frac{1}{2}\norm{\nablabg w}{L^2(\C)}^2 \]
and rearranging shows $\mathcal J(v) \leq \mathcal J(w)$. Uniqueness follows since $\mathcal J$ is strictly convex.
\end{proof}
We will often (but not always) write $\nablabg$ instead of $\ext{\nablabg}$. From \eqref{eq:harmonicExtensionGradientBound}, we find that $\ext{\mathcal{E}}\colon H(M) \to X(\C)$ is an isometry:
\[\norm{\ext{\mathcal{E}}u}{X(\C)} = \norm{u}{H(M)}.\]
The next lemma is fundamental (see also \cite{CabreTan, ValdinociSire}).
%\begin{lem}\label{lem:uniformConvergenceOfE}
%The infinite sums that define $\norm{(\mathcal{E} u)(y)}{L^2(M)}$ and $\norm{\nablabg (\mathcal{E} u)(y)}{L^2(M)}$ are uniformly convergent on $[0,\infty)$ and $[\epsilon,\infty)$ respectively for any $\epsilon > 0$.
%%\[\norm{\nabla_M v(y)}{L^2(M)}^2 = 2\sum_{k =1}^\infty \lambda_k e^{-2y\sqrt{\lambda_k}}|(u,\varphi_k)|^2\]
%for any $y > 0$.
%\end{lem}
\begin{lem}\label{lem:fracDTN}The fractional Laplacian of $u \in H(M)$ is recovered through the Dirichlet-to-Neumann map:
\[(-\Delta_M)^{1\slash 2}u = -\lim_{y \to 0^+}\frac{\partial \ext{\mathcal{E}}u}{\partial y}\quad\text{in $H(M)^*$}.\]
\end{lem}
\begin{proof}
If $\mean{u}=0$ and $\eta \in H(M)$, taking the limit $y \to 0^+$ (using Abel's test) in
\begin{align*}
-\langle v_y(y), \eta \rangle_{H(M)^*, H(M)} = \sum_{k=1}^\infty \lambda_k^{1\slash 2}e^{-y\lambda_k^{1\slash 2}}(u,\varphi_k)_{L^2(M)}(\eta,\varphi_k)_{L^2(M)}
\end{align*}
and comparing the result to \eqref{eq:spectralH12Action} gives us what we expected. The case $\mean{u} \neq 0$ follows easily. %:
%\begin{align*}
%\lim_{y \to 0^+}-\langle v_y(y), \eta \rangle_{H(M)^*, H(M)} = \sum_{k=1}^\infty \sqrt{\lambda_k}(u,\varphi_k)_{L^2(M)}(\eta,\varphi_k)_{L^2(M)} = \langle (-\Delta_M)^{1\slash 2}u, \eta \rangle_{H(M)^*, H(M)}.
%\end{align*}
%We see that if $\mean u=0$,
%\begin{equation}\label{eq:rec}
%-\lim_{y \to 0^+}v_y(y) = \lim_{y \to 0^+}\sum_{k=1}^\infty \lambda_k^{1\slash 2}e^{-y\lambda_k^{1\slash 2}}(u,\varphi_k)_{L^2(M)}\varphi_k = \sum_{k=1}^\infty \lambda_k^{1\slash 2}(u,\varphi_k)_{L^2(M)}\varphi_k = (-\Delta_M)^{1\slash 2}u.
%\end{equation}
%If $\mean{u}\neq 0$,
%\[(-\Delta_M)^{1\slash 2}u = (-\Delta_M)^{1\slash 2}((u-\mean u) + \mean{u}) = (-\Delta_M)^{1\slash 2}(u-\mean u)= -\lim_{y \to 0^+}\partial_y\mathcal{E}(u-\mean u)= -\lim_{y \to 0^+} \partial_y\ext{\mathcal{E}}u.\]
%where the penultimate equality is from \eqref{eq:rec}.
\end{proof}
%\begin{remark}
%Note that $\mathcal{E}\colon H(M) \to H^1(\C) + \mathbb{R}$.
%\end{remark}
%\begin{remark}The following bounds hold:
%\begin{align}
%\norm{\nabla_{\mean g}\mathcal{E}u}{L^2(\C)}^2 = \norm{(-\Delta_M)^{1\slash 4}u}{L^2(M)}^2 = |u|^2_{H(M)} \qquad \forall u \in H(M)\label{eq:harmonicExtensionGradientBound}.
%\end{align}
%That the constant is independent of $M$ is very convenient for us.
%\end{remark}
\begin{lem}\label{lem:HMEquivalence}The space $H(M) = H^{1\slash 2}(M)$ with an equivalence of norms.
\end{lem}
\begin{proof}
%Let us see why $H(M) \subset H^{1\slash 2}(M)$. 
Given $u \in H(M)$ with $\mean u=0$, %it suffices to find $v \in H^1(\C)$ such that $\mathcal{T}v =u$, since $\mathcal{T}$ has range in $H^{1\slash 2}(M)$. We simply 
 define $v = \mathcal{E}u$, which we know belongs to 
%\[v(y) = \sum_{k \geq 1}e^{-\lambda_k^{1\slash 2}y}(u,\varphi_k)_{L^2(M)}\varphi_k\]
$H^1(\C)$ from Theorem \ref{thm:harmonicExtensionExistence} and so $\mathcal{T} v = u \in H^{1\slash 2}(M)$ since $\mathcal{T}$ has range in $H^{1\slash 2}(M)$. For the case $\mean{u}\neq 0$, we have that $u = u-\mean{u}+ \mean{u}\in H^{1\slash 2}(M)$. Now we prove the reverse inclusion. Recall that a function $u \in L^2(M) + H^1(M)$ belongs to the interpolation space $H^{1\slash 2}(M)$ as defined by the $K$-method if the following norm is finite:%, given the $K$-functional (for $t > 0$)
\[\norm{u}{H^{1\slash 2}(M)} := \left(\int_0^\infty (t^{-1\slash 2}K(t,u))^2\frac{\mathrm{d}t}{t}\right)^{1\slash 2}\quad\text{where}\quad K(t,u) := \inf_{\substack{u=u_0 + u_1\\ u_0 \in L^2(M)\\ u_1 \in H^1(M)}}\left(\norm{u_0}{L^2(M)}^2 + t^2\norm{u_1}{H^1(M)}^2\right)^{1\slash 2}.\]
%\[\norm{u}{H^{1\slash 2}(M)} = \left(\int_0^\infty (t^{-1\slash 2}K(t,u))^2\frac{\mathrm{d}t}{t}\right)^{1\slash 2}.\]
See \cite[Chapter 1, \S15]{LionsMagenes}, \cite[Appendix B]{Bramble}, \cite[Appendix B]{McLean} for more information. We follow the ideas of the proof of Theorem B.2 in \cite{Bramble} now. Let $u=\sum_{k=0}^\infty u_k\varphi_k \in H^{1\slash 2}(M)$ and $v=\sum_{k=0}^\infty v_k\varphi_k \in H^1(M)$. Then 
\begin{align*}
K^2(t,u) &= \inf_{v \in H^1(M)}\left(\norm{u-v}{L^2(M)}^2 + t^2\norm{v}{H^1(M)}^2\right)\\
&= \inf_{v \in H^1(M)}\left(\sum_{k =0}^\infty |u_k-v_k|^2 + t^2\sum_{k=0}^\infty(1+\lambda_k)|v_k|^2\right)\\
&= \sum_{k=0}^\infty \frac{t^2(1+\lambda_k)}{1+t^2(1+\lambda_k)}|u_k|^2
\end{align*}
because the expression in the infimum is minimised when 
$v_k={u_k}\slash (1+t^2(1+\lambda_k))$. 
%\[K^2(t,u) = \sum_{k=0}^\infty \frac{t^2(1+\lambda_k)}{1+t^2(1+\lambda_k)}|u_k|^2.\]
Therefore,
\begin{align*}
\nonumber \norm{u}{H^{1\slash 2}(M)}^2 = \sum_{k=0}^\infty (1+\lambda_k)|u_k|^2\int_0^\infty  \frac{1}{1+t^2(1+\lambda_k)}\mathrm{d}t
= \frac{\pi}{2}\sum_{k=0}^\infty \sqrt{1+\lambda_k}|u_k|^2\label{eq:kmethodSum2}
%\nonumber &\geq \frac{\pi}{2}\sum_{k=0}^\infty \sqrt{\lambda_k}|u_k|^2\\
\geq \frac{\pi}{2}\seminorm{u}{H(M)}^2
%\nonumber &\leq \frac{\pi}{2}\sum_{k=0}^\infty (1+\lambda_k^{1\slash 2})|u_k|^2\\
%\nonumber &= \frac{\pi}{2}\norm{u}{H(M)}^2.
\end{align*}
which implies that $\norm{u}{H(M)} \leq \pi^{-1\slash 2}(2+\pi)^{1\slash 2}\norm{u}{H^{1\slash 2}(M)}$.
In the above calculation, using $\sqrt{1+\lambda_k} \leq 1+ \sqrt{\lambda_k}$ shows that
%\begin{align*}
$\norm{u}{H^{1\slash 2}(M)} \leq 2^{-1\slash 2}{\pi^{1\slash 2}}\norm{u}{H(M)}.$
%\end{align*}
%and %adding $\frac{\pi}{2}\norm{u}{L^2(M)}^2$ to both sides, and 
%using $H^{1\slash 2}(M) \cts L^2(M)$ on the left hand side concludes the equivalence of norms.%, we find
%\[\norm{u}{H^{1\slash 2}(M)}^2 \geq C\norm{u}{H(M)}^2.\]
\end{proof}
We could also have proved this lemma via the $J$-method of interpolation \cite[Appendix B]{McLean} and Weyl's law \cite[Chapter 3, equation (3.2.24)]{Jost}, as is done in \cite[\S 3.1.3]{Bonforte} on a bounded domain. Another approach, relying explicitly on the Gagliardo norm on $H^{1\slash 2}(M)$ when $M$ is a hypersurface, can also work with the use of two-sided Gaussian estimates on the heat kernel, similar to \cite[\S 2.2]{Stinga2014} for the case of the Neumann Laplacian on a bounded domain.

We introduce the following cut-off function which will be useful here and in \S \ref{sec:fpme}.
\begin{defn}[Cut-off function]\label{defn:cutoffFunction}
For any $\rho > 0$, there exists a smooth cut-off function $\psi_\rho$ such that
\begin{equation*}
\psi_\rho(y) = \begin{cases}1 &: y \in [0,\rho]\\
0 &: y \in [2\rho, \infty)
\end{cases}
\end{equation*}
and $-\frac{C}{\rho}\sqrt{\psi(1-\frac y\rho)} \leq \psi_\rho'(y) \leq 0$ on $[\rho,2\rho]$, with $C$ not depending on $\rho$. 
%It follows that $\psi_r(y) \to 1$ pointwise, $|\psi_r(y)| \leq 1$, $\psi'_r(y) \to 0$ pointwise and $|\psi_r'(y)| \leq \frac{1}{r}C\sqrt{\psi(1-\frac yr)} \leq C$ (for $r \geq 1$) on $[r,2r]$.
\end{defn}
Define a map $\mathcal{N}\colon H^{1\slash 2}(M) \to H^{-1\slash 2}(M)$ by %$\mathcal{N}u = \frac{\partial \ext{\mathcal{E}}u}{\partial \nu}\big|_{y=0}.$ It satisfies %(see the proof of Theorem \ref{thm:harmonicExtensionExistence})
\begin{equation}\label{eq:dtn}
\langle \mathcal{N}u, h \rangle_{H^{-1\slash 2}(M), H^{1\slash 2}(M)} %= \langle (-\Delta_M)^{1\slash 2}u, h \rangle_{H^{-1\slash 2}(M), H^{1\slash 2}(M)} =\langle \frac{\partial (\ext{\mathcal{E}}u)}{\partial \nu}\Big|_{y=0}, h \rangle_{H^{-1\slash 2}(M), H^{1\slash 2}(M)} 
:= \int_{\C}\ext{\nablabg}(\ext{\mathcal{E}}u) \nablabg\tilde h %= \int_{\C}\nablabgv \nablabg\tilde h
\end{equation}
where $\tilde h \in H^1(\C)$ is any extension of $h$ (i.e., $\mathcal{T}\tilde h =h$).  %The last equality follows by Green's formula. %the formula
%\[\int_{\C} w\Deltabg \ext{\mathcal{E}}u = -\int_{\C}\nablabg \ext{\mathcal{E}}u\nabla w + \int_{\partial \C}\frac{\partial \ext{\mathcal{E}}u}{\partial \nu}w.\]
This map is well-defined since if we had two arbitrary extensions $\tilde h_1$ and $\tilde h_2$, then 
\[\int_{\C}\ext{\nablabg}(\ext{\mathcal{E}}u) \nablabg\tilde h_1 - \int_{\C}\ext{\nablabg}(\ext{\mathcal{E}}u) \nablabg\tilde h_2 = \int_{\C}\ext{\nablabg}(\ext{\mathcal{E}}u) \nablabg(\tilde h_1-\tilde h_2) = 0\]
by definition of $\ext{\mathcal{E}}u$ and since $\mathcal{T}(\tilde h_1 -\tilde h_2) = 0$. The fact that the extension can be arbitrary will be extremely useful later on. Furthermore, by choosing in \eqref{eq:dtn} $\tilde h = \mathcal{E}(h-\mean h) + \psi_\rho \mean h \in H^1(\C)$, one can see that $\mathcal{N}u$ is linear and that it is indeed in the dual space of $H^{1\slash 2}(M)$. We can write $\mathcal{N}u = {\partial \ext{\mathcal{E}}u}\slash{\partial \nu}\big|_{y=0},$ i.e., $\mathcal{N}$ is the Dirichlet-to-Neumann map; this notation is justified since, if for example ${\Deltabg} {\mathcal{E}}u \in L^2(\C)$, the standard Green's formula implies $\int_{\partial \C}w{\partial {\mathcal{E}}u}\slash {\partial \nu}=\int_{\C} w{\Deltabg} {\mathcal{E}}u + \int_{\C}\nablabg {\mathcal{E}}u\nabla w = \int_{\C}\nablabg {\mathcal{E}}u\nabla w.$
\subsection{Trace maps}
The trace map can be extended to the space $X(\C)$.
\begin{lem}\label{lem:existenceTraceMapXC}There exists a bounded linear trace map $\ext{\mathcal{T}}\colon X(\C) \to H(M)$ such that 
\begin{equation*}\label{eq:traceMapBound}
\norm{\ext{\mathcal{T}}v}{H(M)} \leq \norm{v}{X(\C)} \quad \text{for $v \in X(\C)$},
\end{equation*}
$\ext{\mathcal{T}}w = \mathcal{T} w$ if $w \in H^1(\C) \subset X(\C)$, and $\ext{\mathcal{T}}w := \lim_{n \to \infty} \mathcal{T} w_n$ for $w_n \in H^1(\C)$ converging to $w$ in $X(\C)$. 
\end{lem}
\begin{proof}
This is similar to \cite[Lemma 2.4]{Stinga2014}. Let $w \in H^1(\C)$ be arbitrary with $\mathcal{T} w =: w_0$. If $\meanu{w_0} = 0$, by \eqref{eq:harmonicExtensionGradientBound},
\begin{align*}
\seminorm{w_0}{H(M)}^2 %&= \norm{\nablabg\mathcal{E}w_0}{L^2(\C)}^2 \tag{by \eqref{eq:harmonicExtensionGradientBound}}
= 2\mathcal J(\mathcal{E}w_0)
\leq 2\mathcal J(w)%\tag{since the harmonic extension minimises $J$}\\
= \norm{\nablabg w}{L^2(\C)}^2,
%&= \norm{w}{X(\C)}^2.
\end{align*}
and since this inequality involves seminorms, we can drop the assumption $\meanu{w_0} = 0$. Adding $\norm{w_0}{L^2(M)}^2$ to both sides shows that %$\norm{w_0}{H(M)}^2 \leq \norm{w}{X(\C)}^2$, that is,
%\begin{align*}
%\norm{w_0}{L^2(M)}^2 + \norm{(-\Delta_M)^{1\slash 4}w_0}{L^2(M)}^2 &\leq  \norm{\nablabgw}{L^2(\C)}^2 + \norm{w_0}{L^2(M)}^2 = \norm{w}{X(\C)}^2.
%\end{align*}
%That is, 
$\mathcal{T}\colon H^1(\C) \to H(M)$ satisfies $\norm{\mathcal{T}w}{H(M)} \leq \norm{w}{X(\C)}.$ Then the BLT theorem gives the result.
%Let now $w \in X(\C)$, and take $w_n \in H^1(\C)$ converge to $w$ in $X(\C)$. We have 
%\[\norm{\mathcal{T}w_n - \mathcal{T}w_m}{H(M)} \leq \norm{w_n - w_m}{X(\C)} \leq \norm{w_n - w_m}{H^1(\C)}\]
%so $\mathcal{T}w_n$ converges in $H(M)$. Set 
%\[Tw := \lim_{n \to \infty} \mathcal{T} w_n\]
%so $T$ is defined now on $X(\C)$.
%\norm{(-\Delta_M){\frac 14}w_0}{L^2(M)}^2 = \sum \lambda_k^{1\slash 2}|(w_0, \varphi_k)|^2 
\end{proof}
%\begin{remark}The trace map $\ext{\mathcal{T}}$ between $\{v \in X(\C) \mid \Delta_{\bar g}v = 0 \}$ and $H^{1\slash 2}(M)$ is invertible with inverse $\ext{\mathcal{E}}\colon H^{1\slash 2}(M) \to X(\C)$.
%\end{remark}
The following lemma is a seminorm boundedness property of the trace map; note the Gagliardo seminorm on the left hand side (the proof of Lemma \ref{lem:existenceTraceMapXC} had the $H(M)$ seminorm on the left hand side instead).
\begin{lem}\label{lem:traceBoundedBySeminorm}Let $M=\Gamma$ be a hypersurface of class $C^1$. For every $v \in H^1(\C)$,
\[\seminorm{\mathcal{T}v}{W^{1\slash 2,2}(\Gamma)} \leq  C\norm{\nablabg v}{L^2(\C)}.
\]
\end{lem}
\begin{proof}
If $v \in H^1(\C)$ satisfies $\mean{v}(y) = 0$ for all $y$, then using the trace theorem and Poincar\'e's inequality on $\Gamma$,% we calculate
\begin{align*}
\norm{\mathcal{T}v}{W^{1\slash 2, 2}(\Gamma)} 
&\leq C_1\left(\norm{\nabla_\Gamma v}{L^2(\C)} + \norm{\nablabg v}{L^2(\C)}\right) \leq C_2\norm{\nablabg v}{L^2(\C)}.
\end{align*}
Now suppose that $v \in H^1(\C)$ does have not have spatial mean value zero for a.a. $y$. Then define 
$\hat v(x,y) = v(x,y) - \mean{v(y)}$ which satisfies $\frac{1}{|\Gamma|}\int_{\Gamma}\hat v(y) = 0$
and $\hat v \in H^1(\C)$ by Lemma \ref{lem:meanValueContinuous}. Then, by the above inequality,
\begin{align*}
\norm{\mathcal{T} \hat v}{W^{1\slash 2, 2}(\Gamma)} %&\leq C_2\norm{\nablabg \left(v - \frac{1}{|\Gamma|}\int_{\Gamma}v(y)\right)}{L^2(\C)}\\
%&\leq C_3\left(\norm{\nablabg v}{L^2(\C)}+ \norm{\frac{1}{|\Gamma|}\partial_y\int_{\Gamma}v(y)}{L^2(\C)}\right)\\
&\leq C_2\left(\norm{\nablabg v}{L^2(\C)}+ \norm{\frac{1}{|\Gamma|}\int_{\Gamma}\partial_y v(y)}{L^2(\C)}\right) \leq C_3\norm{\nablabg v}{L^2(\C)},
\end{align*} 
but, using Lemma \ref{lem:meanValueContinuous}, the left hand side is greater than $\seminorm{\mathcal{T}v}{W^{1\slash 2, 2}(\Gamma)}$
%\begin{align*}
%\end{align*}
because the seminorm does not see constants.
%and so we have proved the lemma.
\end{proof}
\subsection{The truncated harmonic extension problem}\label{sec:truncatedHarmonicExtension}
%With $\C_R := M \times [0,R]$, let $\mathcal{T}_R \colon H^1(\C_R) \to H^{1\slash 2}(\partial \C_R)$ be the trace operator (which was defined in \S \ref{sec:sobolevSpacesOnSemiInfiniteCylinders}), so $(\mathcal{T}_R v)\colon M \times \{0, R\} \to \mathbb{R}$. Let $(\mathcal{T}_{R,y=0}v)(x) := (\mathcal{T}_R v)(x,0)$ and $(\mathcal{T}_{R,y=R}v)(x) := (\mathcal{T}_R v)(x,R)$  denote the traces of $v$ at $y=0$ and $y=R$ respectively; these maps  $\mathcal{T}_{R,y=0}, \mathcal{T}_{R,y=R} \colon H^1(\C_R) \to H^{1\slash 2}(M)$ are bounded between the respective spaces since
%\[\norm{\mathcal{T}_R v}{H^{1\slash 2}(\partial\C_R)}^2 = \norm{\mathcal{T}_{R,y=0}v}{H^{1\slash 2}(M)}^2 + \norm{\mathcal{T}_{R,y=R}v}{H^{1\slash 2}(M)}^2\]
%(write out the integral definition to see this). 
%Given $u \in H^{1\slash 2}(M)$, the weak formulation of \eqref{eq:prelim2} is% consider the problem
%\begin{equation}\label{eq:truncatedHarmonicProblem}
%\begin{aligned}
%\Deltabg v(x,y) &= 0 &&\text{on $\C_R$}\\
%v(x,0) &= u&&\text{on $M \times \{0\}$}\\
%v(x,R) &= 0&&\text{on $M \times \{R\}$}%\\
%%\bar u &= 0
%\end{aligned}
%\end{equation}
%with weak formulation
%\begin{equation}\label{eq:truncatedHarmonicProblemWeakForm}
%\end{equation}
Define $H^1_0(\C_R) := \{ \eta \in H^1(\C_R) \mid \mathcal{T}_{R,y=0} \eta = \mathcal{T}_{R,y=R} \eta=0\}$\label{defn:H10}; this is a Hilbert space because it is a closed linear subspace of $H^1(\C_R)$. %The two boundary conditions in \eqref{eq:truncatedHarmonicProblem} can also be written as 
%\begin{align*}
%\mathcal{T}_{R,y=0}v = u \quad \text{and} \quad \mathcal{T}_{R,y=R}v = 0.%\qquad \text{or}\qquad \mathcal{T}_R v = v|_{\partial \C_R} = u\chi_{\{y =0\}}.
%\end{align*}
For $\eta \in H^1_0(\C_R)$, 
%\[\norm{\eta}{{H^1_0(\C_R)}}^2 = \int_{M}\int_0^R|\eta|^2 + |\sgrad \eta|^2 + |\eta_y|^2 < \infty,\]
it follows by Fubini--Tonelli that for almost all $x$, $\eta(x,\cdot) \in H^1_0(0,R)$. Thus the Poincar\'e inequality on $[0,R]$ implies for almost all $x$ that
\[\int_0^R |\eta(x)|^2 \leq C_P \int_0^R |\eta_y(x)|^2.\]
Using this fact in the definition of the norm of $H^1_0(\C_R)$ gives $\norm{\eta}{H^1_0(\C_R)} \leq C\norm{\nablabg \eta}{L^2(\C_R)}$ so that $\norm{\nablabg \cdot}{L^2(\C_R)}$ is an equivalent norm on $H^1_0(\C_R).$
\begin{theorem}\label{thm:truncatedHarmonicExtensionExistence}
For every $u \in H(M)$, there exists a unique weak solution $\ext{\mathcal{E}}_Ru = v \in H^1(\C_R)$ to the truncated harmonic extension problem \eqref{eq:prelim2} satisfying $(\ext{\mathcal{E}}_Ru)(\cdot,0) = u(\cdot)$ and $(\ext{\mathcal{E}}_Ru)(\cdot,R) = 0$ in $L^2(M)$ and\[\int_{\C_R} \nablabg v \nablabg \eta = 0\quad\text{for all $\eta \in H^1_0(\C_R)$}
.\]
When $\mean u=0$, we write the solution as $\mathcal{E}_R u$ which is such that $\frac{1}{|M|}\int_{M}(\mathcal{E}_Ru)(y) =0$ for all $y \in [0,R]$. The map $\ext{\mathcal{E}}_R\colon H(M) \to H^1(\C_R)$ satisfies $\ext{\mathcal{E}}_Ru = \mathcal{E}_R(u-\mean u) + \frac{R-y}{R}\mean u$. Furthermore, $\mathcal{E}_Ru \in C^0([0,R];L^2(M)) \cap C^0((0,R];H^1(M)),$ $\partial_y \mathcal{E}_R u \in C^0((0,R];L^2(M))$ and
\begin{align}
\norm{\ext{\mathcal{E}}_R u}{L^2(\C_R)}^2 &\leq \frac{1}{2\lambda_1^{1\slash 2}}\norm{u-\mean u}{L^2(M)}^2 + 4R|M||\mean u|^2,\label{eq:truncatedHarmonicExtensionL2Bound}\\
\norm{\nablabg \ext{\mathcal{E}}_Ru}{L^2(\C_R)}^2 &\leq \left(1 + {1}\slash {(2\sinh^2(\lambda_1^{1\slash 2}R))}\right)\seminorm{u-\mean u}{H(M)}^2 + \frac{|M||\mean u|^2}{R}.\label{eq:truncatedHarmonicExtensionGradientBound}
\end{align}
\end{theorem}
\begin{proof}
Suppose that $\mean{u}=0$ and define
\[(\mathcal{E}_Ru)(y) := v(y) := \sum_{k=1}^\infty \left(\alpha_1(k,R)e^{\sqrt{\lambda_k}y} + \alpha_2(k,R)e^{-\sqrt{\lambda_k}y}\right)(u,\varphi_k)_{L^2(M)}\varphi_k\]
where 
\[\alpha_1(k,R) = -\frac{e^{-\sqrt{\lambda_k}R}}{e^{\sqrt{\lambda_k}R} - e^{-\sqrt{\lambda_k}R}}\quad\text{and}\quad \alpha_2(k,R) = \frac{e^{\sqrt{\lambda_k}R}}{e^{\sqrt{\lambda_k}R} - e^{-\sqrt{\lambda_k}R}}.\]
The formula for $\mathcal{E}_Ru$ comes from separation of variables and the infinite sum converges in $L^2(M)$ for all $y \in [0,R]$. %Let $\mean u=0$ and set $v=\ext{\mathcal{E}}_R u$.
 We have
\begin{align}
\nonumber \int_{\C_R} |v(y)|^2 &= \int_0^R\sum_{k=1}^\infty \left(\alpha_1(k,R)e^{\sqrt{\lambda_k}y} + \alpha_2(k,R)e^{-\sqrt{\lambda_k}y}\right)^2|(u,\varphi_k)_{L^2(M)}|^2\\
\nonumber &= \sum_{k=1}^\infty \frac{1}{2\sqrt{\lambda_k}}\frac{{e^{2\sqrt{\lambda_k}R} }{-e^{-2\sqrt{\lambda_k}R}}{}-4R\sqrt{\lambda_k}}{e^{2\sqrt{\lambda_k}R} + e^{-2\sqrt{\lambda_k}R}-2} |(u,\varphi_k)_{L^2(M)}|^2\\
%&\leq \frac{1}{2\sqrt{\lambda_1}}\sum_{k=1}^\infty |(u,\varphi_k)|^2\\
&\leq \frac{1}{2\sqrt{\lambda_1}}\norm{u}{L^2(M)}^2\label{eq:pp1}
\end{align}
and
%\[\left({\mathcal{E}}_R(u-\mean u), \frac{R-y}{R}\mean u\right)_{L^2(M)}=\frac{R-y}{R}\mean u({\mathcal{E}}_R(u-\mean u),1)_{L^2(M)} = 0.\] 
%\end{proof}
%\begin{lem}\label{lem:gradBoundOnTruncated}
%We have %for $u$ such that $\mean{u}= 0$
%\[\norm{\nablabg \mathcal{E}_Ru}{L^2(\C_R)} \leq \sqrt{C(\lambda_0,R)}\seminorm{u}{H(M)}\]
%and for $\mean{u}\neq 0$:
%\begin{align*}
%\norm{\nablabg \ext{\mathcal{E}}_Ru}{L^2(\C_R)}^2 &\leq (1+{2}\slash (e^{2\sqrt{\lambda_1}R}-1))\seminorm{u-\mean u}{H(M)}^2 + \frac{|M||\mean u|^2}{R}.
%\end{align*}
%%where $C(\lambda_1,R) = .$
%%If $M=\Gamma(t)$, $C_1$ (and hence $C_2$) may depend on $t$.
%\end{lem}
%\begin{proof}
%For \eqref{eq:truncatedHarmonicExtensionGradientBound}, again let $\mean u=0$ initially. We have %from Lemma \ref{lem:uniformConvergesTruncated} that
\begin{align*}
\int_M |\nablabg v(y)|^2 =2\sum_{k=1}^\infty {\lambda_k}\left(\alpha_1(k,R)^2e^{2\sqrt{\lambda_k}y} + \alpha_2(k,R)^2e^{-2\sqrt{\lambda_k}y}\right)|(u,\varphi_k)_{L^2(M)}|^2,
\end{align*}
%and
%\begin{align*}
%\alpha_1(k,R)^2e^{2\sqrt{\lambda_k}y} + \alpha_2(k,R)^2e^{-2\sqrt{\lambda_k}y} &= \frac{e^{2\sqrt{\lambda_k}(R-y)}+e^{2\sqrt{\lambda_k}(y-R)}}{(e^{\sqrt{\lambda_k}R} - e^{-\sqrt{\lambda_k}R})^2}
%\end{align*}
and formally, integration over $[0,R]$ of the latter quantity yields
\begin{align}
\nonumber  \int_{\C_R}  |\nablabg v|^2 &= \sum_{k=1}^\infty \sqrt{\lambda_k}\frac{e^{2\sqrt{\lambda_k}R}-e^{-2\sqrt{\lambda_k}R}}{e^{2\sqrt{\lambda_k}R} + e^{-2\sqrt{\lambda_k}R}-2}|(u,\varphi_k)_{L^2(M)}|^2\\
\nonumber &\leq \sum_{k=1}^\infty \sqrt{\lambda_k}\frac{e^{2\sqrt{\lambda_k}R}+e^{-2\sqrt{\lambda_k}R}}{e^{2\sqrt{\lambda_k}R} + e^{-2\sqrt{\lambda_k}R}-2}|(u,\varphi_k)_{L^2(M)}|^2\\
%\end{align*}
%Note that
%\begin{align*}
%\frac{e^{2\sqrt{\lambda_k}(R-\epsilon)}-e^{-2\sqrt{\lambda_k}(R-\epsilon)}}{e^{2\sqrt{\lambda_k}R} + e^{-2\sqrt{\lambda_k}R}-2} &\leq \frac{e^{2\sqrt{\lambda_k}R}+e^{-2\sqrt{\lambda_k}R}}{e^{2\sqrt{\lambda_k}R} + e^{-2\sqrt{\lambda_k}R}-2} \leq  1 + \frac{1}{2\sinh^2(\sqrt{\lambda_1}R)}.
%\end{align*}
%Therefore,
%\[
&\leq \left(1 + \frac{1}{2\sinh^2(\sqrt{\lambda_1}R)}\right)\sum_{k=1}^\infty \sqrt{\lambda_k}|(u,\varphi_k)_{L^2(M)}|^2.\label{eq:pp2}
\end{align}
To make this rigorous, we should have integrated over $[\epsilon, R]$ instead of $[0,R]$ and then passed to the limit as $\epsilon \to 0$ using the monotone convergence theorem on the left hand side.

To see that $v$ is a weak solution, take a test function $\eta \in H^1_0(\C_R)$ with $\eta(y) = \sum_{k=0}^\infty (\eta(y),\varphi_k)\varphi_k$ and calculate (using \eqref{eq:actionH1})
\[\int_M \sgrad v\sgrad \eta = \sum_{k=1}^\infty \lambda_k \left(\alpha_1(k,R)e^{\sqrt{\lambda_k}y} + \alpha_2(k,R)e^{-\sqrt{\lambda_k}y}\right)(\eta(y), \varphi_k)_{L^2(M)}(u, \varphi_k)_{L^2(M)}=\int_M v_{yy}\eta.\]
%and since
%\[v_{yy}(y) = \sum_{k=1}^\infty  \lambda_k\left(\alpha_1(k,R)e^{\sqrt{\lambda_k}y} + \alpha_2(k,R)e^{-\sqrt{\lambda_k}y}\right)(u,\varphi_k)_{L^2(M)}\varphi_k,\]
%we have
%\[\int_M v_{yy}(y)\eta(y) =\sum_{k=1}^\infty (\eta(y),\varphi_k)_{L^2(M)}(u,\varphi_k)_{L^2(M)} \lambda_k\left(\alpha_1(k,R)e^{\sqrt{\lambda_k}y} + \alpha_2(k,R)e^{-\sqrt{\lambda_k}y}\right) = \int_M \sgrad v \sgrad \eta.\]
Then
\begin{align*}
\int_{\C_R}\sgrad v \sgrad \eta &= \int_{\C_R}v_{yy}\eta = -\int_{C_R}v_y\eta_y + \int_{\partial\C_R}v_y\eta = -\int_{C_R}v_y\eta_y
\end{align*}
with the last equality since $\eta$ vanishes on the boundary; this implies the result. For $u$ with $\mean u \neq 0$, we set $\ext{\mathcal{E}}_Ru := \mathcal{E}_R(u-\mean u) + \frac{R-y}{R}\mean u.$ This is a solution because
\[\int_0^R \int_M \nablabg\left(\frac{R-y}{R}\mean u\right)\nablabg \eta = -\frac{\mean u}{R}\int_0^R \int_M \partial_y \eta =-\frac{\mean u}{R}\int_0^R \frac{d}{dy}\int_M \eta = -\frac{\mean u}{R}\int_M (\eta(R) - \eta(0)) = 0.\]
If $\mean{u}\neq 0$, \eqref{eq:truncatedHarmonicExtensionL2Bound} and \eqref{eq:truncatedHarmonicExtensionGradientBound} follow from \eqref{eq:pp1} and \eqref{eq:pp2} by noting that ${\mathcal{E}}_R(u-\mean u) \perp (R-y)\mean u\slash {R}$ in $L^2(M)$ and $\nablabg {\mathcal{E}}_R(u-\mean u) \perp \nablabg(R-y)\mean u\slash {R}$ in $L^2(M)$ respectively, pointwise in $y$.
\end{proof}
\begin{remark}\label{rem:truncatedHarmonicExtensionArbitrary}
Define a form $a_R \colon H(M) \times H(M) \to \mathbb{R}$ by
\[a_R(u,\eta) = \int_{\C_R}\nablabg \ext{\mathcal{E}}_{R}u\nablabg \tilde \eta\]
where $\tilde \eta \in H^1_0(\C_R)$ is an (arbitrary) extension of $\eta$; %that satisfies $\mathcal{T}_{R,y=0}\tilde \eta = \eta$ and $\mathcal{T}_{R,y=R}\tilde \eta = 0$; 
the choice of extension does not matter, since for any two such extensions $\tilde \eta_1$ and $\tilde \eta_2$,
\begin{align*}
\int_{\C_R}\nablabg \ext{\mathcal{E}}_{R}u\nablabg \tilde \eta_1 &- \int_{\C_R}\nablabg \ext{\mathcal{E}}_{R}u\nablabg \tilde \eta_2 = \int_{\C_R}\nablabg \ext{\mathcal{E}}_{R}u\nablabg (\tilde \eta_1 -\tilde \eta_2)= 0
\end{align*}
by definition of the weak solution, because $\tilde \eta_1 - \tilde \eta_2 \in H^1_0(\C_R)$. %and $\mathcal{T}_{R,y=0}(\tilde \eta_1 - \tilde \eta_2) = \eta-\eta = 0$ and $\mathcal{T}_{R,y=R}(\tilde \eta_1 - \tilde \eta_2) = 0$, i.e., $\tilde \eta_1 - \tilde \eta_2 \in H^1_0(\C_R)$.
\end{remark}
\subsection{Decay and convergence of solutions of the truncated problem}\label{sec:convergenceOfTruncatedSolution}
%\begin{defn}\label{defn:Z}
In order to compare functions defined on $C_R$ and $\C$, we define the zero extension %$\mathscr{Z}_R\colon \{ \eta \in H^1(0,R) \mid \eta(R) = 0\} \to H^1(0,\infty)$ 
\begin{equation}\label{defn:Z}
\mathscr{Z}_R\colon \{ \eta \in H^1(0,R) \mid \eta(R) = 0\} \to H^1(0,\infty) \quad\text{by}\quad (\mathscr{Z}_R\eta)(y) = \begin{cases}
\eta(y) &:\text{if $y \leq R$}\\
0 &:\text{otherwise}
\end{cases}
\end{equation}
which is an isometry. %satisfies $\norm{\mathscr{Z}_R\eta}{H^1(0,\infty)} = \norm{\eta}{H^1(0,R)}$.
%\end{defn}
Clearly, we can also view $\mathscr{Z}_R$ as a map $\mathscr{Z}_R\colon \{ \eta \in H^1(\C_R) \mid \eta(x,R) = 0\} \to H^1(\C)$ and this is also an isometry. %satisfies $\norm{\mathscr{Z}_R\eta}{H^1(\C)} = \norm{\eta}{H^1(\C_R)}$. 
\begin{lem}\label{lem:expDecay}
For all $u \in H(M)$,
\[\norm{\nablabg(\ext{\mathcal{E}}u-\mathscr{Z}_R\ext{\mathcal{E}}_Ru)}{L^2(\C)}^2 \leq 3e^{-R\sqrt{\lambda_1}}\seminorm{u-\mean u}{H(M)}^2  + \frac{2}{R}e^{-2R\sqrt{\lambda_1}}\norm{u-\mean u}{L^2(M)}^2 + \frac{2|M||\mean u|^2}{R}.\]
Hence $\mathscr{Z}_R\ext{\mathcal{E}}_Ru \to \ext{\mathcal{E}}u$ in $X(\C)$ as $R \to \infty$.
\end{lem}
\begin{proof}
Firstly, let $\eta_R = (\ext{\mathcal{E}}u-\ext{\mathcal{E}}_Ru) - \ext{\mathcal{E}}u(R)y\slash R$ which satisfies $\eta_R(0) = \eta_R(R) = 0$, and consider the difference of the weak formulations of $\ext{\mathcal{E}}_Ru$ tested with $\eta_R$ and $\ext{\mathcal{E}}u$ tested with $\mathscr{Z}_R\eta_R$:
\begin{align*}
0 %&= \int_{\C_R} \nablabg (\ext{\mathcal{E}}u-\ext{\mathcal{E}}_Ru) \nablabg \left((\ext{\mathcal{E}}u-\ext{\mathcal{E}}_Ru) - \ext{\mathcal{E}}u(R)\frac{y}{R}\right)\\
&= \int_{\C_R} |\nablabg (\ext{\mathcal{E}}u-\ext{\mathcal{E}}_Ru)|^2 - \nablabg (\ext{\mathcal{E}}u-\ext{\mathcal{E}}_Ru)\nablabg\left(\frac{\ext{\mathcal{E}}u(R)y}{R}\right),
\end{align*}
so that
\begin{align*}
\int_{\C_R} |\nablabg (\ext{\mathcal{E}}u-\ext{\mathcal{E}}_Ru)|^2 
&\leq \int_{\C_R} |\nablabg (\ext{\mathcal{E}}u-\ext{\mathcal{E}}_Ru)||\nablabg(\ext{\mathcal{E}}u(R))| + |\partial_y(\ext{\mathcal{E}}u-\ext{\mathcal{E}}_Ru)||\ext{\mathcal{E}}u(R)|\frac 1R\\
&\leq \frac{1}{2}\norm{\nablabg(\ext{\mathcal{E}}u-\ext{\mathcal{E}}_Ru)}{L^2(\C_R)}^2  + \int_{\C_R} |\sgrad \ext{\mathcal{E}}u(R)|^2 + \frac{1}{R^2}|\ext{\mathcal{E}}u(R)|^2
\end{align*}
where we used $ab \leq a^2\slash 4 + b^2$ on both products. Now, recalling that $\ext{\mathcal{E}}u(R) = \sum_{k \geq 1} e^{-R\sqrt{\lambda_k}}(u-\mean u,\varphi_k)\varphi_k + \mean u,$ %and hence
\begin{align*}
\int_{\C_R} |\ext{\mathcal{E}}u(R)|^2 
&\leq Re^{-2R\sqrt{\lambda_1}}\norm{u-\mean u}{L^2(M)}^2 + R|M||\mean u|^2
\end{align*}
and
\begin{align*}
\int_{\C_R} |\sgrad \ext{\mathcal{E}}u(R)|^2 %&= \int_0^R\sum_{k \geq 1} \lambda_k e^{-2R\sqrt{\lambda_k}}|(u-\mean u,\varphi_k)|^2\\
&= \sum_{k \geq 1} R\lambda_k e^{-2R\sqrt{\lambda_k}}|(u-\mean u,\varphi_k)|^2
\leq \sum_{k \geq 1} \sqrt{\lambda_k} e^{-R\sqrt{\lambda_k}}|(u-\mean u,\varphi_k)|^2
\leq e^{-R\sqrt{\lambda_1}}\seminorm{u-\mean u}{H(M)}^2
\end{align*}
hold (where we used using $xe^{-2x} \leq e^{-x}$), giving
\begin{equation}\label{eq:add1}
\int_{\C_R} |\nablabg (\ext{\mathcal{E}}u-\ext{\mathcal{E}}_Ru)|^2 \leq 2e^{-R\sqrt{\lambda_1}}\seminorm{u-\mean u}{H(M)}^2  + \frac{2}{R}e^{-2R\sqrt{\lambda_1}}\norm{u-\mean u}{L^2(M)}^2 + \frac{2|M||\mean u|^2}{R}.
\end{equation}
Secondly, note that %(using the uniform convergence in Lemma \ref{lem:uniformConvergenceOfE} to integrate term-by-term)
\begin{align*}
\int_R^\infty \int_M |\nablabg \ext{\mathcal{E}}u|^2 &= 2\int_R^\infty \sum_{k \geq 1} \lambda_k e^{-2y\sqrt{\lambda_k}}|(u-\mean u,\varphi_k)|^2 = e^{-2R\sqrt{\lambda_1}}\seminorm{u-\mean u}{H(M)}^2.
\end{align*}
Adding this and \eqref{eq:add1} implies the result.
%Finally, this all implies
%\begin{align*}
%\int_0^\infty \int_M |\nablabg(\ext{\mathcal{E}}u-\mathscr{Z}_R\ext{\mathcal{E}}_Ru)|^2 &= \int_0^R\int_M |\nablabg(\ext{\mathcal{E}}u-\ext{\mathcal{E}}_Ru)|^2 + \int_R^\infty \int_M |\nablabg \bar{\mathcal{E}}u|^2\\
%&\leq 3e^{-R\sqrt{\lambda_1}}\seminorm{u-\bar u}{H(M)}^2  + \frac{4}{R}e^{-2R\sqrt{\lambda_k}}\norm{u-\bar u}{L^2(M)}^2 + \frac{4|M||\mean u|^2}{R}.
%\end{align*}
\end{proof}
\begin{lem}\label{lem:convergenceTruncationsL2}
For all $u \in H(M)$ with $\mean{u}=0$, 
\begin{align*}
\norm{\mathscr{Z}_R\mathcal{E}_R u-\mathcal{E}u}{L^2(\C)}^2 &\leq C_P\left(3e^{-R\sqrt{\lambda_1}}\seminorm{u}{H(M)}^2  + \frac{2}{R}e^{-2R\sqrt{\lambda_1}}\norm{u}{L^2(M)}^2\right)+ \frac{e^{-2R\sqrt{\lambda_1}}}{2\sqrt{\lambda_1}}\norm{u}{L^2(M)}^2
\end{align*}
(where $C_P$ is the Poincar\'e constant on $M$). Hence $\mathscr{Z}_R\mathcal{E}_R u \to \mathcal{E}u$ in $L^2(\C)$.
\end{lem}
\begin{proof}
If $\mean{u}=0$, then $\ext{\mathcal{E}u(y)}=\ext{\mathcal{E}_Ru(y)}=0$ for all $y$. Therefore, with Poincar\'e's inequality on $M$, %gives for $y > 0$
%\[\int_{M}|\mathcal{E}u(y)-\mathcal{E}_R u(y)|^2 \leq C_P\int_{M}|\nabla_M \mathcal{E}u(y)-\nabla_M \mathcal{E}_R u(y)|^2\]
%for a.a. $y$. Therefore,
\begin{align*}
\int_0^R \int_M |\mathcal{E}u-\mathcal{E}_R u|^2 &\leq C_P\int_0^R \int_M |\nablabg \mathcal{E}u-\nablabg \mathcal{E}_R u|^2 
%&= C_P\int_0^R \int_M |\nablabg \mathcal{E}u-\nablabg \mathscr{Z}_R\mathcal{E}_R u|^2\\
\leq C_P\int_0^\infty \int_M |\nablabg \mathcal{E}u-\nablabg \mathscr{Z}_R\mathcal{E}_R u|^2. %\leq C_PK(R, u)
\end{align*}
%where $K(R,u)$ is from the previous lemma. 
Over the interval $(R,\infty)$, we have
\begin{align*}
\int_R^\infty \int_M |\mathscr{Z}_R\mathcal{E}_R u-\mathcal{E}u|^2 &= \int_R^\infty \int_M |\mathcal{E}u|^2%\\
%&=\int_R^\infty \sum_{k=1}^\infty e^{-2y\lambda_k^{1\slash 2}}|(u,\varphi_k)_{L^2(M)}|^2\\
=\sum_{k=1}^\infty \frac{e^{-2R\sqrt{\lambda_k}}}{2\sqrt{\lambda_k}}|(u,\varphi_k)_{L^2(M)}|^2
\leq \frac{e^{-2R\sqrt{\lambda_1}}}{2\sqrt{\lambda_1}}\norm{u}{L^2(M)}^2.
\end{align*}
Adding these two estimates and using the previous lemma yields the result.
%Therefore,
%\begin{align*}
%\int_0^\infty \int_M |\mathscr{Z}_R\mathcal{E}_R u-\mathcal{E}u|^2 &=  \int_0^R \int_M |\mathcal{E}_R u-\mathcal{E}u|^2 + \int_R^\infty \int_M |\mathscr{Z}_R\mathcal{E}_R u-\mathcal{E}u|^2\\
%&\leq C_PK(R,u) + \frac{e^{-2R\lambda_1^{1\slash 2}}}{2\lambda_1^{1\slash 2}}\norm{u}{L^2(M)}^2.
%\end{align*}
%where we took the limit inside the infinite sum, which is justified because the sum uniformly converges due to Abel's test (the coefficient $\frac{e^{-2R\lambda_k^{1\slash 2}}}{2\lambda_k^{1\slash 2}}$ is monotone decreasing in $k$ and is bounded above on $ R \in (0,\infty)$ by $\frac{1}{2\sqrt{\lambda_k}}$).
\end{proof}
The next lemma describes \emph{continuous convergence}.
\begin{lem}\label{lem:ctsConvergence}If $u_R$, $u \in H(M)$ with $u_R \to u$ in $L^2(M)$ with $\meanu{u_R} = \mean u=0$, then $\mathscr{Z}_R{\mathcal{E}}_Ru_R \to {\mathcal{E}}u$ in $L^2(\C)$.
\end{lem}
\begin{proof}
Writing $\mathscr{Z}_R{\mathcal{E}}_Ru_R - {\mathcal{E}}u = \mathscr{Z}_R{\mathcal{E}}_R(u_R - u) + \mathscr{Z}_R{\mathcal{E}}_Ru- {\mathcal{E}}u$, using the triangle inequality and \eqref{eq:truncatedHarmonicExtensionL2Bound},
\begin{align*}
\norm{\mathscr{Z}_R{\mathcal{E}}_Ru_R - {\mathcal{E}}u}{L^2(\C)} %&=\norm{\mathscr{Z}_R{\mathcal{E}}_Ru_R - \mathscr{Z}_R{\mathcal{E}}_Ru + \mathscr{Z}_R{\mathcal{E}}_Ru- {\mathcal{E}}u}{L^2(\C)}\\
%&\leq \norm{\mathscr{Z}_R{\mathcal{E}}_R(u_R - u)}{L^2(\C)} + \norm{\mathscr{Z}_R{\mathcal{E}}_Ru- {\mathcal{E}}u}{L^2(\C)}\\
%&\leq \norm{{\mathcal{E}}_R(u_R - u)}{L^2(\C_R)} + \norm{\mathscr{Z}_R{\mathcal{E}}_Ru- {\mathcal{E}}u}{L^2(\C)}\\
&\leq C\norm{u_R - u}{L^2(M)} + \norm{\mathscr{Z}_R{\mathcal{E}}_Ru- {\mathcal{E}}u}{L^2(\C)}%\tag{by \eqref{eq:truncatedHarmonicExtensionL2Bound}}%\\
%&\to 0
\end{align*}
which tends to zero by Lemma \ref{lem:convergenceTruncationsL2}. 
\end{proof}
\section{Function spaces on evolving hypersurfaces and preliminary results}\label{sec:functionSpaces}
We start with conditions on the prescribed evolution, in addition to \eqref{eq:eigenvalueEstimate}.
\begin{ass}\label{ass:onHypersurfaces}
For each $t \in [0,T],$ let $\Gamma(t) \subset \mathbb{R}^{d+1}$ be a smooth and compact $d$-dimensional hypersurface without boundary, and assume the existence of a flow $\Phi\colon [0,T] \times \mathbb{R}^{d+1} \to \mathbb{R}^{d+1}$ such that for all $t \in [0,T]$, with $\Gamma_0 := \Gamma(0)$, the map $\Phi_t^0(\cdot):=\Phi(t,\cdot)\colon \Gamma_0 \to \Gamma(t)$ is a $C^3$-diffeomorphism that satisfies
$\frac{d}{dt}\Phi^0_t(\cdot) = \mathbf w(t,\Phi^0_t(\cdot))$ and
$\Phi^0_0(\cdot) = \text{Id}(\cdot)$
for a given $C^2$ velocity field $\mathbf w\colon [0,T]\times\mathbb{R}^{d+1} \to \mathbb{R}^{d+1}$, which we assume satisfies the uniform bound
$|\sgradt \cdot \mathbf w(t)| \leq C$ for all $t \in [0,T]$. A $C^2$ normal vector field on the hypersurfaces is denoted by $\nu^{\Gamma}\colon [0,T]\times \mathbb{R}^{d+1} \to \mathbb{R}^{d+1}$. 
\end{ass}
It follows that the Jacobian $J^0_t := \det \mathbf{D}\Phi^0_t$ is $C^2$ and uniformly bounded away from zero and infinity. We denote by $\Phi^t_0\colon \Gamma(t) \to \Gamma_0$ the inverse of $\Phi^0_t$ and define $|\Gamma| := \max_{t \in [0,T]}|\Gamma(t)|$.
\begin{remark}\label{rem:lots}The assumption \eqref{eq:eigenvalueEstimate} is satisfied if for example each $\Gamma(t)$ has non-negative Ricci curvature, or if the Ricci curvature of $\Gamma(t)$ is greater than $\rho(t) < 0$, where $-\rho(t) \leq \rho$ holds for all $t \in [0,T]$ with $\rho$ a constant. See Theorem 4.6.1 in \cite{Jost} and the discussion afterwards. Also, instead of assuming \eqref{eq:eigenvalueEstimate}, one could study the possible continuity of $t \mapsto \lambda_1(t)$ through the theory of perturbations of linear operators \cite{kato2013perturbation}. Let us furthermore remark that all functional analytic results in this section not involving the harmonic extension maps remain true for $\Gamma(t)$ of class $C^3$. 
\end{remark}
\subsection{Function spaces}
In order to define the spaces $L^p_Y$ mentioned in the introduction, we need simply to verify a few assumptions.%, which we shall do below. 
%\begin{center}
%  \begin{minipage}{0.9\textwidth}
%  \begin{comment}
%      \begin{tabular}{rl|rl|rl}
%    \hline
%    Space &Formed from  &Space &Formed from &Space &Formed from\\
%    \hline
%    $L^2_{L^2}$ &$\{L^2(\Gamma(t))\}_{t\in [0,T]}$     &$L^2_{L^2(\C)}$ &$\{L^2(\C(t))\}_{t\in [0,T]}$    &$L^2_{L^2(\C_R)}$ &$\{L^2(\C_R(t))\}_{t\in [0,T]}$\\
%    
%    $L^2_{H^{1\slash 2}}$ &$\{H^{1\slash 2}(\Gamma(t))\}_{t\in [0,T]}$ &$L^2_{H^1(\C)}$ &$\{H^1(\C(t))\}_{t\in [0,T]}$    &$L^2_{H^1(\C_R)}$ &$\{H^1(\C_R(t))\}_{t\in [0,T]}$\\
%
%    $L^2_{X(\C)}$ &$\{X(\C(t))\}_{t\in [0,T]}$  	&$L^2_{H^1_0(\C_R)}$ &$\{H^1_0(\C_R(t))\}_{t\in [0,T]}$\\
%    \end{tabular}
%    \end{comment}
%        \begin{tabular}{rl}
%    \hline
%    Space &Formed from \\
%    \hline\\[-3mm]
%    $L^2_{L^2}$ &$\{L^2(\Gamma(t))\}_{t\in [0,T]}$    \\
%    $L^\infty_{L^\infty}$ &$\{L^\infty(\Gamma(t))\}_{t \in [0,T]}$ \\
%    $L^2_{W^{1\slash 2, 2}}$ &$\{W^{1\slash 2, 2}(\Gamma(t))\}_{t\in [0,T]}$ 
%    \end{tabular}
%        \begin{tabular}{rl}
%    \hline
%    Space &Formed from  \\
%    \hline\\[-3mm]
%$L^2_{L^2(\C)}$ &$\{L^2(\C(t))\}_{t\in [0,T]}$  \\
%$L^2_{H^1(\C)}$ &$\{H^1(\C(t))\}_{t\in [0,T]}$  \\
%$L^2_{X(\C)}$ &$\{X(\C(t))\}_{t\in [0,T]}$  	
%    \end{tabular}     
%    \begin{tabular}{rl}
%    \hline
%    Space &Formed from\\
%    \hline\\[-3mm]
%	$L^2_{L^2(\C_R)}$ &$\{L^2(\C_R(t))\}_{t\in [0,T]}$\\
%    $L^2_{H^1(\C_R)}$ &$\{H^1(\C_R(t))\}_{t\in [0,T]}$\\
%	$L^2_{H^1_0(\C_R)}$ &$\{H^1_0(\C_R(t))\}_{t\in [0,T]}$
%    \end{tabular}
%  \end{minipage}
%\end{center}
\subsubsection{Spaces on the surface $\Gamma$}
For $u\colon \Gamma_0 \to \mathbb{R}$, define $(\phi_t u)(x):=(\phi_{\Gamma,t} u)(x) := u (\Phi^t_0(x))$. Fortunately, we already checked that the spaces $L^p_{L^q}$ and $L^2_{W^{1 \slash 2, 2}}$ are well-defined in \cite[\S 2.2.1]{Alphonse2014} and \cite[\S 4.1 and \S 5.4]{Alphonse2014a} respectively. Recall from \cite{Alphonse2014b} that $u \in L^2_{W^{1\slash 2, 2}}$ is said to have a \emph{weak material derivative} $\dot u \in L^2_{W^{-1\slash 2, 2}}$ if 
\[\int_0^T \langle \dot u(t), \zeta(t) \rangle_{W^{-1\slash 2, 2}(\Gamma(t)), W^{1\slash 2, 2}(\Gamma(t))} = - \int_0^T\int_{\Gamma(t)}\dot \zeta(t)u(t) - \int_0^T \int_{\Gamma(t)}u(t)\zeta(t)\sgradt \cdot \mathbf w(t)\]
holds for all $\zeta \in \{\zeta \in L^2_{W^{1\slash 2, 2}} \mid \phi_{\Gamma, -(\cdot)}\zeta(\cdot) \in \mathcal{D}((0,T);W^{1\slash 2, 2}(\Gamma_0))\}$, where $\zeta$ belonging to this set has a \emph{strong material derivative} defined by $\dot \zeta (t) := \phi_{\Gamma, t}({d}\slash {dt}(\phi_{\Gamma, -t}\zeta(t)))$ (see also \cite{Alphonse2014a}).
In \cite[\S 5.4.1]{Alphonse2014a} the evolving Sobolev--Bochner space 
\[\mathbb{W}(W^{1 \slash 2, 2}, W^{-1 \slash 2, 2}) := \{ u \in L^2_{W^{1 \slash 2, 2}} \mid \dot u \in L^2_{W^{-1 \slash 2, 2}} \}\] was shown to be well-defined and isomorphic (via $\phi_{\Gamma, -(\cdot)}$) with an equivalence of norms to 
\[\mathcal{W}(W^{1 \slash 2, 2}, W^{-1 \slash 2, 2}) := \{ \tilde u \in L^2(0,T;W^{1 \slash 2, 2}(\Gamma_0)) \mid \tilde u' \in L^2(0,T;W^{-1 \slash 2, 2}(\Gamma_0))\},\]
and this implies that $\mathbb{W}(W^{1 \slash 2, 2}, W^{-1 \slash 2, 2}) \compact L^2_{L^2}$. 

The following lemma (which is surprisingly non-trivial) is useful later on; the proof of the continuity is the same as in Lemma 2.5 of \cite{Bousquet} with the obvious modifications.%, adapted to our setting; we give the proof in the appendix for convenience. 
\begin{lem}\label{lem:nemytskiiH12}For a sufficiently smooth hypersurface $\Gamma$, if $\beta\colon \mathbb{R} \to \mathbb{R}$ is Lipschitz with $\beta(0)=0$, then the map $\beta \colon W^{1\slash 2,2}(\Gamma) \to W^{1\slash 2,2}(\Gamma)$ defined by composition is (sequentially) continuous and satisfies
\[\norm{\beta(u)}{W^{1\slash 2,2}(\Gamma)} \leq \Lip(\beta)\norm{u}{W^{1\slash 2,2}(\Gamma)}\quad\text{for all $u \in W^{1\slash 2,2}(\Gamma)$}.\]
\end{lem}
\subsubsection{Spaces on the cylinders $\C$ and $\C_R$}
Recall from the introduction that $\C(t) = \Gamma(t) \times [0,\infty)$, and set $\C_0 := \C(0).$ Given $v \in L^2(\C_0)$, define $(\phi_{\C,t} v)(x,y) := v (\Phi^t_0(x),y)$. We have
\begin{align}
\norm{\phi_{\C,t}v}{L^2(\C(t))}^2  = \int_0^\infty \int_{\Gamma(t)} |v (\Phi^t_0(x),y)|^2= \int_0^\infty \int_{\Gamma_0} |v (z,y)|^2J^0_t \leq \lVert{J^{(\cdot)}_0}\rVert_{\infty} \norm{v}{L^2(\C)}^2\label{eq:R1},
\end{align}
so $\phi_{\C,t}\colon L^2(\C_0) \to L^2(\C(t))$. The inverse mapping is $\phi_{\C,-t}\colon L^2(\C(t)) \to L^2(\C_0)$ given by $(\phi_{\C,-t} w)(x,y) = w (\Phi^0_t(x),y)$ and these maps are linear homeomorphisms. Also, we see that if $v \in H^1(\C_0)$,
\begin{align}
\nonumber \seminorm{\phi_{\C,t}v}{H^1(\C(t))}^2 &= \int_0^\infty \int_{\Gamma(t)} |\nablabg v (\Phi^t_0(x),y)|^2\\
\nonumber &= \int_0^\infty \int_{\Gamma(t)} |(\mathbf{D}\Phi^t_0 )^\intercal(x)(\nabla_{\Gamma_0}v(y)\circ \Phi^t_0(x)) +\partial_yv(\Phi^t_0(x),y)|^2\\
&= \int_0^\infty \int_{\Gamma_0} |(\mathbf{D}\Phi^t_0 )^\intercal \circ \Phi^0_t(z) (\nabla_{\Gamma_0}v(z,y)) + \partial_yv(z,y)|^2J^0_t\label{eq:R2}\\
%\nonumber &\leq C_1\int_0^\infty \int_{\Gamma_0} |\nabla_{\Gamma_0}u(z,y)|^2 +|\partial_yu(z,y)|^2\\
&\leq C\seminorm{v}{H^1(\C_0)}^2\label{eq:H1SeminormBounded}
\end{align}
%which shows 
%\begin{equation}
%\seminorm{\phi_{\C,t} u}{H^1(\C(t))} \leq C_2 %\seminorm{u}{H^1(\C_0)}.
%\end{equation}
Overall, we have shown that %$\norm{\phi_{\C,t} u}{H^1(\C(t))} \leq C_3 \norm{u}{H^1(\C_0)}$ for all $t$ and $u \in H^1(\C_0)$ and that 
$\phi_{\C,t}\colon H^1(\C_0) \to H^1(\C(t))$ is bounded uniformly and well-defined. Finally, we have from \eqref{eq:R1} and \eqref{eq:R2} that $t \mapsto \norm{\phi_{\C,t} v}{H^1(\C(t))}^2$
%\[t \mapsto \norm{\phi_{\C,t} u}{H^1(\C(t))}^2 = \int_0^\infty \int_{\Gamma_0} |u (z,y)|^2J^0_t + \int_0^\infty \int_{\Gamma_0} |(\mathbf{D}\Phi^t_0 )^\intercal \circ \Phi^0_t(z) (\nabla_{\Gamma_0}u(z,y))+\partial_yu(z,y)|^2J^0_t\]
is continuous. Using the theory in \cite[\S 2.2]{Alphonse2014b}, this allows us to define the spaces $L^2_{H^1(\C)}$ and $L^2_{L^2(\C)}$ (just ignore the gradient term).  Clearly, the same argument allows us to define $L^2_{L^2(\C_R)},$ $L^2_{H^1(\C_R)}$, and $L^2_{H^1_0(\C_R)}$ using a map $\phi_{\C_R, t}$ defined in the same way.
\begin{defn}\label{defn:harmonicExtensionsWitht}We denote by $\ext{\mathcal{E}}_t$ and $\ext{\mathcal{E}}_{R,t}$ the maps $\ext{\mathcal{E}}$ and $\ext{\mathcal{E}}_R$ defined in Theorems \ref{thm:harmonicExtensionExistence} and \ref{thm:truncatedHarmonicExtensionExistence} respectively with the manifold $M$ chosen to be $\Gamma(t)$ (and likewise without the overlines). Similarly, we denote by $\ext{\mathcal{T}}_t$, $\mathcal{T}_{R,t,y=0}$ and $\mathcal{T}_{R,t,y=R}$ %\colon X(\C(t)) \to H^{1\slash 2}(\Gamma(t))$ 
the trace maps $\ext{\mathcal{T}}$, $\mathcal{T}_{R,y=0}$ and $\mathcal{T}_{R,y=R}$ defined in Lemma \ref{lem:existenceTraceMapXC} and in \S \ref{sec:sobolevSpacesOnSemiInfiniteCylinders} respectively with the choice $M=\Gamma(t)$.
\end{defn}
\begin{lem}[Commutativity of the trace and pushforward maps]\label{lem:traceCommutesWithPhi}The following identity holds:
\begin{equation*}
\begin{aligned}
\mathcal{T}_t(\phi_{\C, t}v) &= \phi_{\Gamma, t}(\mathcal{T}_0v) &&\text{for all $v \in H^1(\C_0)$.}
\end{aligned}
\end{equation*}
\end{lem}
\begin{proof}
We have $\phi_{\C,t}v = v \circ \Phi^t_0 \in H^1(\C(t))$ and so $\mathcal{T}_t\phi_{\C,t}v = v(\Phi^t_0(\cdot),0),$ whilst on the other hand, $\phi_{\Gamma, t}\mathcal{T}_0 v = v(\cdot,0) \circ \Phi^t_0(\cdot) = v(\Phi^t_0(\cdot),0)$. 
\end{proof}
Lemma \ref{lem:traceCommutesWithPhi} implies that if $v \in H^1(\C(t))$, then, using the boundedness of $\phi_{\Gamma, t}$,
\[\norm{\mathcal{T}_t v}{W^{1\slash 2,2}(\Gamma(t))} \leq  C_1\norm{\mathcal{T}_0 \phi_{\C,-t}v}{W^{1\slash 2,2}(\Gamma_0)} \leq C_2\norm{\phi_{\C,-t}v}{H^1(\C_0)}\leq C_3\norm{v}{H^1(\C(t))}\]
because of the trace theorem and the equivalence of norms between $H^{1\slash 2}(\Gamma_0)$ and $W^{1\slash 2, 2}(\Gamma_0)$. This shows that $\mathcal{T}_t \colon H^1(\C(t)) \to W^{1\slash 2,2}(\Gamma(t))$ is bounded independently of $t$. By the same argument, the uniform boundedness of $\mathcal{T}_{R,t,y=0}, \mathcal{T}_{R,t,y=R}\colon H^1(\C_R(t)) \to W^{1\slash 2, 2}(\Gamma(t))$ 
also holds, and a version of Lemma \ref{lem:traceCommutesWithPhi} holds for these maps, which allows us to define $L^2_{H^1_0(\C_R)}$. Now, by Lemma \ref{lem:traceCommutesWithPhi} %for $u \in H^1(\C_0)$, $\norm{\mathcal{T}_t\phi_{\C,t} u}{L^2(\Gamma(t))} = \norm{\phi_{\Gamma,t} \mathcal{T}_0u}{L^2(\Gamma(t))} \leq C\norm{\mathcal{T}_0u}{L^2(\Gamma_0)}.$ Using this 
and \eqref{eq:H1SeminormBounded}, for $v \in H^1(\C_0)$, we have
\begin{align*}
\norm{\phi_{\C,t} v}{X(\C(t))}^2 %&= \left(\int_0^\infty\int_{\Gamma(t)} |\nablabg \phi_{\C,t} u|^2\right) + \int_{\Gamma(t)}|\mathcal{T}_t\phi_{\C,t} u|^2 
&= \seminorm{\phi_{\C,t} v}{H^1(\C(t))}^2 + \norm{\mathcal{T}_t\phi_{\C,t} v}{L^2(\Gamma(t))}^2
%&\leq C(\seminorm{u}{H^1(\C_0)}^2 + \norm{\mathcal{T}_0u}{L^2(\Gamma_0)}^2)\\
\leq  C\norm{v}{X(\C_0)}^2
\end{align*}
which shows that $\phi_{\C,t} \colon H^1(\C_0) \to X(\C(t))$ has a useful boundedness property which, by the BLT theorem, allows us to extend $\phi_{\C,t}$ to a bounded linear map $\bar\phi_{\C,t}\colon X(\C_0) \to X(\C(t))$ defined as
\[\bar\phi_{\C,t}x_0 := \lim_{n \to \infty} \phi_{\C,t}v_n\text{ in $X(\C(t))$ for $v_n \in H^1(\C_0)$ with $v_n \to x_0$ in $X(\C_0)$}.\]
We also have the measurability of $t \mapsto \norm{\bar\phi_{\C,t}x_0}{X(\C(t))} = \lim_{n \to \infty} \norm{\phi_{\C,t}v_n}{X(\C(t))}$. Thus $L^2_{X(\C)}$ is also well-defined. Similar arguments can be made for the inverse operator of $\bar \phi_{C,t}$, denoted $\bar \phi_{\C,-t}\colon X(\C(t)) \to X(\C_0).$
By a density argument, exploiting the continuity of the operators involved, we can show that
\begin{equation}\label{eq:traceOnXCCommutesWithPhi}
\ext{\mathcal{T}}_t(\bar \phi_{\C, t}v) = \phi_{\Gamma, t}(\ext{\mathcal{T}}_0v)
 \quad\text{for all $v \in X(\C_0)$}.
\end{equation}
\subsubsection{Superposition trace maps}
\begin{lem}\label{lem:existenceOfTraceMapOnL2H1AndL2X}There exist bounded linear trace operators $\mathbb{T}\colon L^2_{H^1(\C)} \to L^2_{W^{1\slash 2, 2}}$ and $\ext{\mathbb{T}}\colon L^2_{X(\C)} \to L^2_{W^{1\slash 2,2}}$ satisfying $(\mathbb{T}v)(t) = \mathcal{T}_tv(t)$ and $(\ext{\mathbb{T}}v)(t) = \ext{\mathcal{T}}_tv(t)$ for almost every $t$.
\end{lem}
\begin{proof}
For $v \in L^2_{X(\C)}$, define $(\ext{\mathbb{T}}v)(t) = \ext{\mathcal{T}}_tv(t)$. Then
%\begin{align*}
$(\ext{\mathbb{T}}v)(t)= \ext{\mathcal{T}}_t v(t) = \phi_{\Gamma,t}\ext{\mathcal{T}}_0 \bar \phi_{\C,-t}v(t)$ 
%\end{align*}
by \eqref{eq:traceOnXCCommutesWithPhi} which gives measurability in time, and we have the bound
\begin{align*}
\norm{(\ext{\mathbb{T}}v)(t)}{W^{1\slash 2, 2}(\Gamma(t))}%&= \norm{\phi_{\Gamma,t}\ext{\mathcal{T}}_0 \bar \phi_{\C,-t}v(t)}{W^{1\slash 2, 2}(\Gamma(t))}\\
&\leq C_1\norm{\ext{\mathcal{T}}_0 \bar \phi_{\C,-t}v(t)}{W^{1\slash 2, 2}(\Gamma_0)}%\\
%&\leq C_2\norm{\ext{\mathcal{T}}_0 \bar \phi_{\C,-t}v(t)}{H^{1\slash 2}(\Gamma_0)}\\
%\leq C_3\norm{\ext{\mathcal{T}}_0 \bar \phi_{\C,-t}v(t)}{H(\Gamma_0)}%\\
\leq C_3\norm{\bar \phi_{\C,-t}v(t)}{X(\C_0)}
\leq C_4\norm{v(t)}{X(\C(t))}
\end{align*}
where for the second inequality we used the equivalence of norms between $W^{1\slash 2, 2}(\Gamma_0)$ and $H(\Gamma_0)$ and Lemma \ref{lem:existenceTraceMapXC}. This proves that $\ext{\mathbb{T}} \colon L^2_{X(\C)} \to L^2_{W^{1\slash 2,2}}$ is well-defined as a bounded linear operator. The operator $\mathbb{T}$ can be seen as the restriction of $\ext{\mathbb{T}}$ to $L^2_{H^1(\C)}$.
\end{proof}
By the same reasoning as above, we can prove the following lemma.
\begin{lem}\label{lem:existenceTraceMapsTruncatedL2H1CR}There exist bounded linear trace operators $\mathbb{T}_{R,y=0}, \mathbb{T}_{R,y=R}\colon L^2_{H^1(\C_R)} \to L^2_{W^{1\slash 2, 2}}$ satisfying $(\mathbb{T}_{R,y=0}v)(t) = \mathcal{T}_{R,t,y=0}v(t)$ and $(\mathbb{T}_{R,y=R}v)(t) = \mathcal{T}_{R,t,y=R}v(t)$ for almost every $t$.
\end{lem}
\subsubsection{Some uniform bounds}
When we work with a time-dependent manifold $M=\Gamma(t)$, we would like the constants in the gradient bounds \eqref{eq:harmonicExtensionGradientBound} and \eqref{eq:truncatedHarmonicExtensionGradientBound} to be independent of time. The space $H^{1\slash 2}(\Gamma(t))$ is equivalent to $W^{1\slash 2, 2}(\Gamma(t))$ with an equivalence of norms, as we mentioned in the introduction. However, the constants in the equivalence of norms result will depend on $t$ and we have no information as to in what way the dependence is. This means that one has to be careful whenever one uses estimates from \S \ref{sec:fractionalLaplacianOnCompactManifolds} involving the $H^{1\slash 2}(\Gamma(t))$ or $H(\Gamma(t))$ norm in the evolving set-up. For this reason, we need the bounds in the next two lemmas. %Lemmas \ref{lem:boundOnGradientEL2C} and \ref{lem:graddBoundOnTruncatedWithT}. 

\begin{lem}\label{lem:boundOnGradientEL2C}
There exists a constant $C > 0$ such that for all $t \in [0,T]$ and all $u \in H^{1\slash 2}(\Gamma(t))$,
%\[\norm{\nablabgt \mathcal{E}_t u}{L^2(\C(t))} \leq C\norm{u}{H^{1\slash 2}(\Gamma(t))}\]
\[\norm{\nablabgt \ext{\mathcal{E}}_t u}{L^2(\C(t))} \leq C\norm{u-\mean u}{W^{1\slash 2, 2}(\Gamma(t))}.\]
\end{lem}
\begin{proof}
Let $\mean u=0$ and set $U(t) := \phi_{\C,t}\mathcal{R}_0\phi_{\Gamma, -t}u \in H^1(\C(t))$ where $\mathcal{R}_0\colon H^{1\slash 2}(\Gamma_0) \to H^1(\C_0)$ is the right continuous inverse of the trace operator. Note that %$t \mapsto \phi_{\C, -t}U(t)= \mathcal{R}_0\phi_{\Gamma, -t}u(t)$ is measurable and 
\[\norm{U}{H^1(\C(t))} \leq C_0\norm{\mathcal{R}_0\phi_{\Gamma, -t}u}{H^1(\C_0)}\leq C_1\norm{\phi_{\Gamma, -t}u}{H^{1\slash 2}(\Gamma_0)}\leq C_2\norm{\phi_{\Gamma, -t}u}{W^{1\slash 2, 2}(\Gamma_0)}\leq  C_3\norm{u}{W^{1\slash 2, 2}(\Gamma(t))}.\]
Also, we have $\mathcal{T}_tU = \mathcal{T}_t\phi_{\C,t}\mathcal{R}_0\phi_{\Gamma, -t}u = \phi_{\Gamma,t}\mathcal{T}_0\mathcal{R}_0\phi_{\Gamma, -t}u = u$ by Lemma \ref{lem:traceCommutesWithPhi}. So the function $\eta = \mathcal{E}_tu - U \in H^1(\C(t))$ can be taken as an admissible test function in the weak formulation for $\mathcal{E}_tu$, and doing so yields
\begin{align*}
\int_{\C(t)} |\nablabgt \mathcal{E}_tu|^2 &=\int_{\C(t)} \nablabgt \mathcal{E}_tu \nablabgt U
%&\leq  \frac{1}{2}\int_{\C(t)} |\nablabgt \mathcal{E}_tu|^2 + \frac{1}{2}\int_{\C(t)}|\nablabgt U|^2\\
\leq  \frac{1}{2}\int_{\C(t)}|\nablabgt \mathcal{E}_tu|^2 + \frac{C_4}{2}\norm{u}{W^{1\slash 2, 2}(\Gamma)}^2.
\end{align*}
%For general $u$ without mean value zero, we have
%\begin{align*}
%\norm{\nablabgt \ext{\mathcal{E}}_t u}{L^2(\C(t))} =  \norm{\nablabgt \mathcal{E}_t (u-\mean u)}{L^2(\C(t))} \leq C\norm{u-\mean u}{W^{1\slash 2, 2}(\Gamma(t))}.
%\end{align*}
\end{proof}
\begin{lem}\label{lem:graddBoundOnTruncatedWithT}There exist constants $C_1$, $C_2 > 0$ such that for all $t \in [0,T]$ and all $u \in H^{1\slash 2}(\Gamma(t))$,
%\[\norm{\nablabgt {\mathcal{E}}_{R,t}u}{L^2(\C_R(t))} \leq C\norm{u}{H^{1\slash 2}(\Gamma(t))}.\]
\[\norm{\nablabgt \ext{\mathcal{E}}_{R,t}u}{L^2(\C_R(t))}^2 \leq C_1\norm{u-\mean u}{W^{1\slash 2, 2}(\Gamma(t)}^2 + \frac{C_2}{R^2}\norm{u-\mean u}{L^2(\Gamma(t))}^2 + \frac{2|\mean u|^2}{R}|\Gamma|.\]
%where $C_1$ and $C_2$ are independent of $t$.
\end{lem}
\begin{proof}
Suppose $\mean u=0$ and let $\eta = {\mathcal{E}}_{R,t}u-\frac{R-y}{R}{\mathcal{E}}_t u \in H^1_0(\C_R(t))$ which we take as the test function in the weak formulation for $\mathcal{E}_{R,t}u$:
\begin{align*}
 \int_{\C_R(t)}|\nablabgt {\mathcal{E}}_{R,t}u|^2 %&\leq \int_{\C_R(t)}|\nablabgt {\mathcal{E}}_{R,t}u||\nablabgt \left(\frac{R-y}{R}{\mathcal{E}}_t u\right)|\\
&\leq \frac{1}{2}\int_{\C_R(t)}|\nablabgt {\mathcal{E}}_{R,t}u|^2 + \frac{1}{R^2}\left|\nablabgt \left((R-y){\mathcal{E}}_tu\right)\right|^2,
\end{align*}
and this gives
\begin{align*}
\int_{\C_R(t)}|\nablabgt {\mathcal{E}}_{R,t}u|^2 %&\leq \int_{\C_R(t)}\left|\frac{R-y}{R}\nablabgt {\mathcal{E}}_tu - \frac{1}{R}{\mathcal{E}}_tu\right|^2
\leq 2\int_{\C_R(t)}4|\nablabgt {\mathcal{E}}_tu|^2 + \frac{1}{R^2}|{\mathcal{E}}_tu|^2 \leq C_1\norm{u}{W^{1\slash 2, 2}}^2 + \frac{C_2}{R^2}\norm{u}{L^2(\Gamma(t))}^2,
\end{align*} 
where we replaced the integral over $\C_R(t)$ by one over $\C(t)$ (this is why we need $\mean u=0$) and used Lemma \ref{lem:boundOnGradientEL2C} and \eqref{eq:harmonicExtensionL2Bound} in conjunction with \eqref{eq:eigenvalueEstimate}. The $\mean{u}\neq 0$ case follows from the above and 
\begin{align*}
\int_{\C_R(t)}|\nablabgt \ext{\mathcal{E}}_{R,t}u|^2 &=\int_{\C_R(t)}\left|\nablabgt {\mathcal{E}}_{R,t}(u-\mean u) +\frac{1}{R} \nablabgt (R-y)\mean u\right|^2.
\end{align*}
\end{proof}
\subsection{Truncations}\label{sec:truncations}
Let $\Gamma$ be a smooth hypersurface. Define the truncation $T_k \colon \mathbb{R} \to \mathbb{R}$ at height $k$ by
\[T_k(x) 
=\max(x, -k) + \min(x, k) - x= \begin{cases}
k\sign(x) &:{ |x| \geq k}\\
x &: |x| < k.
\end{cases}
\]
Both $T_k\colon L^2(\Gamma) \to L^2(\Gamma)$ and $T_k\colon H^1(\Gamma) \to H^1(\Gamma)$ are bounded continuous maps. The Lipschitz nature of the function $\max(\cdot, 0)$ %We have $|\max(x,0)-\max(y,0)| \leq |x-y|$ and $\max(x,-k) = \max(x+k,0) -k$ which 
implies that $T_k \colon W^{1\slash 2,2}(\Gamma) \to W^{1\slash 2,2}(\Gamma)$ is bounded and by Lemma \ref{lem:nemytskiiH12} it is also continuous. Furthermore, the chain rule for weakly differentiable functions $u$ gives
\begin{equation*}\label{eq:derivativeFact}
\frac{d}{dz}(T_ku(z)) %= \begin{cases}
%\frac{d}{dz}u(z) &\text{if $|u(z)| < k$}\\
%0 &\text{otherwise}
%\end{cases}
= \chi_{\{|u(z)| < k\}}(z)\frac{d}{dz}u(z)
\end{equation*}
for almost every $z$. See \cite[Lemma 2.89]{carl2007nonsmooth}) and  the discussion after Theorem 4.3.6 in \cite{carbone2001unbounded} for these facts on a domain $\Omega$.

Now we discuss truncations over cylinders. Suppose $f \in C^1(\mathbb{R})$ with $f'$ bounded and $f(0)=0$. The chain rule
$\nablabg f(v) = f'(v)\nablabg v$ for $v \in H^1(\C)$ can be proved by the standard argument: approximate $v$ by $v_n \in \mathcal{D}([0,\infty);\mathcal{D}(\Gamma))$, prove the identity for $v_n$ and pass to the limit using continuity of $f'$ and the dominated convergence theorem (DCT). This then allows us to show that
\[\nablabg v^+ = \chi_{\{v \geq 0\}}\nablabg v %\begin{cases}
%\nablabg v &: \text{if $v \geq 0$}\\
%0 &: \text{otherwise}
%\end{cases}
\]
(almost everywhere) by approximating $r \mapsto (r)^+$ by $C^1$ functions with bounded derivatives, the chain rule and then the passage to the limit in the approximations (see \cite[Lemma 1.19]{Heinonen}). This will imply that if $v$ and $w$ are in $H^1(\C)$, then $\max(v,w) \in H^1(\C)$ and
\[\nablabg \max(v,w) =\begin{cases}
\nablabg v &: \text{if $v \geq w$}\\
\nablabg w &: \text{otherwise}.
\end{cases}\]
Since $v = v^+ - v^-$, we have
%$\nablabg v = \nablabg v^+ - \nablabg v^-$, and therefore 
$\nabla v|_{\{v=0\}} = 0$ almost everywhere.  Also, if $v_n, w_n$ are such that $v_n \to v$ and $w_n \to w$ in $H^1(\C)$, then $\max(v_n,w_n) \to \max(v,w)$ in $H^1(\C)$ \cite[Lemma 1.22]{Heinonen}. Therefore, $T_k(v) \in H^1(\C)$ whenever $v \in H^1(\C)$. Furthermore, $T_k(v) \to v$ in $H^1(\C)$ as $k \to \infty$ and $T_k\colon H^1(\C) \to H^1(\C)$ is continuous. If $v \in L^2_{H^1(\C)}$, then $T_k(v) \in L^2_{H^1(\C)}$ too, since $\phi_{\C,-t}T_k(v(t)) = T_k(\phi_{\C,-t}v(t))$ and $T_k(\tilde v) \in L^2(0,T;H^1(\C_0))$ whenever $\tilde v \in L^2(0,T;H^1(\C_0))$.

Clearly, all of this applies if we replace $\C$ with $\C_R$ and in that case we can drop the requirement $f(0)=0$.
\subsection{Integration by parts}We will need the following integration by parts result which is comparable to a result in \cite{gagneux1995analyse} and \cite[Lemma 7.1]{Droniou2007}.
\begin{lem}\label{lem:weakerIBP}Let $f\colon \mathbb{R} \to \mathbb{R}$ with $f(0)=0$ satisfy either
\begin{enumerate}[label=\textup{(\Alph*)}]
\item\label{item:A} $f$ is $C^1$ and Lipschitz, or
\item\label{item:B} $f$ is $C^0$ and piecewise $C^1$ with $f'=0$ outside a compact set $K \subset \subset \mathbb{R},$
\end{enumerate}
and define $F(s) = \int_0^s f(r)\;\mathrm{d}r$. Then for all $u \in \mathbb{W}(W^{1\slash 2,2}, W^{-1\slash 2,2})$, the following formula holds:
\begin{equation}\label{eq:tp}
\int_0^T \langle \dot u(t), f(u(t)) \rangle_{} =\int_{\Gamma(T)}F(u(T)) -\int_{\Gamma_0}F(u_0) - \int_0^T\int_{\Gamma(t)}F(u(t))\sgradt \cdot \mathbf w(t).
\end{equation}
\end{lem}
\begin{proof}
We begin with case \ref{item:A}. Let $u_n \in \mathbb{W}(W^{1\slash 2,2}, L^2) \cap L^\infty_{L^\infty}$ be such that $u_n \to u$ in $\mathbb{W}(W^{1\slash 2,2}, W^{-1\slash 2,2})$. %Such a sequence exists because $C^1([0,T]\times \Gamma_0)$ is dense in $\mathcal{W}(W^{1\slash 2,2}, W^{-1\slash 2,2})$ (eg. see Proposition 23.23 in \cite{ZeidlerIIA} and use density of $C^1(\Gamma_0)$ in $W^{1\slash 2,2}(\Gamma_0)$ \cite[Proposition 3.40]{demengel2012functional}). 
Note that $F(u_n) \in \mathbb{W}(W^{1\slash 2,2}, L^2)$. To see this, observe that
\[|F(s) - F(t)| \leq \int_t^s |f(r)| \leq \norm{f'}{\infty}\max(|s|,|t|)|s-t|,\]
so for almost all $t$, $\int_{\Gamma(t)}|F(u_n(t))|^2 \leq  |\Gamma|\norm{f'}{\infty}^2\norm{u_n(t)}{L^\infty(\Gamma(t))}^4$ and since $\norm{u_n(t)}{L^\infty(\Gamma(t))}$ is bounded almost everywhere by $\norm{u_n}{L^\infty_{L^\infty}}$, we have $F(u_n) \in L^2_{L^2}$. We also see that
\begin{align}
|F(u_n(t,x))-F(u_n(t,y))| &\leq \norm{f'}{\infty}\max(|u_n(t,x)|,|u_n(t,y)|)|u_n(t,x)-u_n(t,y)|\label{eq:ibpf1}
%&\leq \norm{u_n(t)}{L^\infty(\Gamma(t))}\norm{f'}{\infty}|u_n(t,x)-u_n(t,y)|
\end{align}
which shows that $F(u_n) \in L^2_{W^{1\slash 2,2}}$ since $u_n \in L^\infty_{L^\infty}$. Likewise, $\md(F(u_n)) = f(u_n)\dot u_n \in L^2_{L^2}$. This means that the transport theorem is valid and the desired formula \eqref{eq:tp} holds for the $u_n$ and  now
%\[\frac{d}{dt}\int_{\Gamma(t)}F(u_n(t)) = \int_{\Gamma(t)}\md(F(u_n(t))) + \int_{\Gamma(t)}F(u_n(t))\sgrad \cdot \mathbf w\]
%and hence
%\begin{equation}\label{eq:toRef1}
%\int_0^T \langle \dot u_n(t), f(u_n(t)) \rangle = \int_{\Gamma(T)}F(u_n(T)) -\int_{\Gamma_0}F(u_n(0)) - \int_0^T\int_{\Gamma(t)}F(u_n(t))\sgrad \cdot \mathbf w.
%\end{equation}
we must pass to the limit in $n$. For almost every $t$, for a subsequence, $u_{n}(t) \to u(t)$ in $W^{1 \slash 2, 2}(\Gamma(t))$, so by Lemma \ref{lem:nemytskiiH12}, $\norm{f(u_n(t))-f(u(t))}{W^{1 \slash 2, 2}(\Gamma(t))}^2 \to 0$ and 
\[\norm{f(u_n(t))-f(u(t))}{W^{1 \slash 2, 2}(\Gamma(t))}^2 \leq 2\norm{f'}{\infty}^2\left(\norm{u_n(t)}{W^{1 \slash 2, 2}(\Gamma(t))}^2 + \norm{u(t)}{W^{1 \slash 2, 2}(\Gamma(t))}^2\right).\]
The right hand side converges to $4\norm{f'}{\infty}^2\norm{u(t)}{W^{1 \slash 2, 2}(\Gamma(t))}^2$ whilst the integral of the right hand side converges to $4\norm{f'}{\infty}^2\int_0^T\norm{u(t)}{W^{1 \slash 2, 2}(\Gamma(t))}^2$ since $u_n \to u$ in $L^2_{W^{1\slash 2, 2}}$. Then the generalised DCT gives $f(u_n) \to f(u)$ in $L^2_{W^{1 \slash 2, 2}}$. For the remaining terms, we can use \eqref{eq:ibpf1} in conjuncation with $\max(|a|,|b|) \leq |a|+|b|$ and Cauchy--Schwarz.
%\end{proof}
%\begin{lem}\label{lem:piecewiseIBP}Let $f \colon \mathbb{R} \to \mathbb{R}$ . Define $F(s) = \int_0^s f(r)\;\mathrm{d}r$. Then for all $u \in \mathbb{W}(W^{1\slash 2,2}, W^{-1\slash 2,2})$, the following formula holds:
%\begin{equation*}
%\int_0^T \langle \dot u(t), f(u(t)) \rangle = \int_{\Gamma(T)}F(u(T)) - \int_{\Gamma_0}F(u(0))- \int_0^T \int_{\Gamma(t)}F(u_n(t))\sgrad \cdot \mathbf w.
%\end{equation*}
%\end{lem}
%\begin{remark}

For the case \ref{item:B}, note that $f(u) \in L^2_{W^{1\slash 2,2}}$ whenever $u \in L^2_{W^{1\slash 2,2}}$ as $f$ is Lipschitz (since $f'$ is bounded a.e. and $f$ is absolutely continuous). Given $u \in \mathbb{W}(W^{1\slash 2,2}, W^{-1\slash 2,2})$, there exist $u_n \in \mathbb{W}(W^{1\slash 2,2}, L^2)$ such that $u_n \to u$ in $\mathbb{W}(W^{1\slash 2,2}, W^{-1\slash 2,2})$. We have that $F(u_n) \in \mathbb{W}(W^{1\slash 2,2}, L^2)$ 
because $F$ is Lipschitz. So then we can use the standard integration by parts formula to obtain the desired formula for $u_n$. Then again we need to pass to the limit. 
%\begin{align}
%\int_{\Gamma(T)}F(u_n(T)) - \int_{\Gamma_0}F(u_n(0)) %&= \int_0^T \frac{d}{dt}\int_{\Gamma(t)}F(u_n(t)) = \int_0^T \int_{\Gamma(t)}F'(u_n(t))\dot u_n(t) + \int_0^T \int_{\Gamma(t)}F(u_n(t))\sgrad \cdot \mathbf w\\
%&= \int_0^T \langle \dot u_n(t), f(u_n(t))\rangle + \int_0^T \int_{\Gamma(t)}F(u_n(t))\sgrad \cdot \mathbf w\label{eq:26}.
%\end{align}
We have that $u_n \to u$ in $C^0_{L^1}$ by $\mathbb{W}(W^{1\slash 2,2}, W^{-1\slash 2,2}) \hookrightarrow C^0_{L^2} \hookrightarrow C^0_{L^1}$. This implies that $F(u_n) \to F(u)$ in $C^0_{L^1}$ because $F$ is Lipschitz; 
%\[\max_{t \in [0,T]}\norm{F(u_n(t))-F(u(t))}{L^1(\Gamma(t))} \leq \Lip(F)\max_{t \in [0,T]}\norm{u_n(t)-u(t)}{L^1(\Gamma(t))} \to 0.\]
this takes care of the right hand side of the formula. To finish, since $f$ also is Lipschitz, $f(u_n) \to f(u)$ in $L^2_{W^{1\slash 2,2}}$ due to the same reasoning as before.
\end{proof}
\section{The harmonic extension problems on evolving spaces}\label{sec:fractionalLaplacianOnL2X}
In this section, we shall consider \eqref{eq:introL2XHarmonicExtensionProblem} and also the following ``$L^2_{H^1(\C_R)}$ truncated harmonic extension problem": given $u \in L^2_{W^{1\slash 2, 2}}$, find $v \in L^2_{H^1(\C_R)}$ such that 
\begin{equation}\label{eq:truncatedHarmonicExtensionL2XStrongForm}
\begin{aligned}
\Deltabg v &=0, \quad \mathbb{T}_{R,y=0}v = u, \quad \mathbb{T}_{R,y=R}v = 0.
%\Deltabgt v(t) &= 0&&\text{on $\C_R(t)$}\\
%v(t,x,0) &= u(t,x)\\
%v(t,x,R) &= 0
\end{aligned}
\end{equation}
%holds in the weak sense:
As explained in the introduction, we study these problems in order to derive measurability in time of $\ext{\mathcal{E}}_t$ and $\ext{\mathcal{E}}_{R,t}$  which we do not automatically get since each $\ext{\mathcal{E}}_t$ and $\ext{\mathcal{E}}_{R,t}$ is constructed individually in time.
\subsection{The harmonic extension of $u \in L^2_{W^{1 \slash 2, 2}}$}
\begin{lem}
For every $u \in L^2_{W^{1\slash 2,2}}$ with $\int_{\Gamma(t)} u(t) = 0$ for a.a. $t$, there exists a $U \in L^2_{H^{1}(\C)}$ with $\mathbb{T}U = u$ and $\int_{\Gamma(t)}U(t,y) = 0$ a.e. $t$ and for all $y$.
\end{lem}
\begin{proof}
Define 
\[U(t) = \phi_{\C,t}\left(\frac{1}{J^0_t} \mathcal{E}_0\left(\frac{\phi_{\Gamma, -t}u(t)}{\phi_{\Gamma, -t}J^t_0}\right)\right)\]
which satisfies $U \in L^2_{H^1(\C)}$ since %$ \mathcal{E}_0({\phi_{\Gamma, -t}u(t)}\slash {\phi_{\Gamma, -t}J^t_0})\slash {J^0_t} 
$\phi_{\C, -t}U(t) \in L^2(0,T;H^1(\C_0))$ by smoothness on $J^t_0$ and by using \eqref{eq:harmonicExtensionL2Bound} and \eqref{eq:harmonicExtensionGradientBound} (measurability can be inferred from considerations of Nemytskii maps \cite[\S 3.4]{gasinski}). It is easy to check that $U$ verifies the desired properties using Lemma \ref{lem:traceCommutesWithPhi}.
\end{proof}
\begin{theorem}[The harmonic extension problem in the space $L^2_{X}$]\label{thm:wellPosednessHarmonicExtensionL2X}
There exists a map $\ext{\mathbb{E}}\colon L^2_{W^{1\slash 2,2}} \to L^2_{X(\C)}$ such that given $u \in L^2_{W^{1\slash 2,2}}$,  $v=\ext{\mathbb{E}}u$ is the unique weak solution of \eqref{eq:introL2XHarmonicExtensionProblem} satisfying $\ext{\mathbb{T}}v = u$ in $L^2_{W^{1\slash 2,2}}$,
\begin{equation}\label{eq:harmonicExtensionL2XWeakForm}
\begin{aligned}
\int_0^T \int_{\C(t)} \ext{\nablabgt} v(t) \nablabgt \eta(t)  &=0 &&\text{for all $\eta \in L^2_{H^1}(\C)$ with $\mathbb{T}\eta = 0$},
\end{aligned}
\end{equation}
and $\frac{1}{|\Gamma(t)|}\int_{\Gamma(t)} (\ext{\mathbb{E}}u)(t) = \mean{u(t)}$. When $\mean{u(t)}=0$ for a.a. $t$, we write the solution as $\mathbb{E}u$. The map $\ext{\mathbb{E}}$ satisfies $\ext{\mathbb{E}}u = \mathbb{E}(u-\mean u) + \mean u$. 
\end{theorem}
\begin{proof}
First, suppose that $\mean{u(t)} = 0$ for a.e. $t$. Let us transform the equation to one with zero initial trace. 
%We can write $\phi_{-t}u(t) = \sum_{k=0}^\infty (\phi_{-t}u(t), \varphi_k)_{L^2(\Gamma_0)}\varphi_k \in L^2(\Gamma_0)$, so that
%\[u(t) = \sum_{k=0}^\infty (\phi_{-t}u(t), \varphi_k)_{L^2(\Gamma_0)}\phi_{t}\varphi_k.\]
%Define $U$ by 
%\[U(t,x,y) := e^{-y}\sum_{k=0}^\infty (\phi_{-t}u(t), \varphi_k)_{L^2(\Gamma_0)}(\phi_t \varphi_k)(x).\]
%We see that 
%\begin{align*}
%\norm{U(t)}{H^1(\C(t))}^2 &= \int_0^\infty\int_{\Gamma(t)}|e^{-y}\sum_{k=0}^\infty (\phi_{-t}u(t), \varphi_k)_{L^2(\Gamma_0)}(\phi_t \varphi_k)(x)|^2 + |e^{-y}\sum_{k=0}^\infty (\phi_{-t}u(t), \varphi_k)_{L^2(\Gamma_0)}(\nabla_\Gamma \phi_t \varphi_k)(x)|^2\\
%&\quad + |e^{-y}\sum_{k=0}^\infty (\phi_{-t}u(t), \varphi_k)_{L^2(\Gamma_0)}(\phi_t \varphi_k)(x)|^2\\
%&= 2\int_0^\infty\int_{\Gamma(t)}|e^{-y}\sum_{k=0}^\infty (\phi_{-t}u(t), \varphi_k)_{L^2(\Gamma_0)}(\phi_t \varphi_k)(x)|^2 + |e^{-y}\sum_{k=0}^\infty (\phi_{-t}u(t), \varphi_k)_{L^2(\Gamma_0)}(\nabla_\Gamma \phi_t \varphi_k)(x)|^2\\
%&= 2\int_0^\infty e^{-2y}\int_{\Gamma(t)}|\sum_{k=0}^\infty (\phi_{-t}u(t), \varphi_k)_{L^2(\Gamma_0)}(\phi_t \varphi_k)(x)|^2 + |\sum_{k=0}^\infty (\phi_{-t}u(t), \varphi_k)_{L^2(\Gamma_0)}(\nabla_\Gamma \phi_t \varphi_k)(x)|^2\\
%&= \norm{u(t)}{L^2(\Gamma(t))}^2 + \norm{\sgrad u(t)}{L^2(\Gamma(t))}^2
%\end{align*}
%Since $t \mapsto \phi_{-t}U(t)$ is measurable, we see that $U \in L^2_{H^1(\C)}$. 
By the previous lemma, there exists a $U \in L^2_{H^1(\C)}$ with $\mathbb{T}U = u$ and crucially
%\[\int_{\Gamma(t)} U(t,x,y)\;\mathrm{d}x = 0,\]
%thus 
$\mean{U(t,y)} = 0$ for a.a. $t$ and all $y$. Set $d:= v-U \in L^2_{H^1(\C)}$ which satisfies
\begin{equation*}
\begin{aligned}
\Delta_{\bar g} d &=-\Delta_{\bar g}U\quad\text{and}\quad \mathbb{T}d = 0.
\end{aligned}
\end{equation*}
The space $\hat X := \{ d \in L^2_{H^1(\C)} \mid \mathbb{T}d = 0 \text{ and } \mean{d(t,y)}= 0\text{ for all $y$ and a.a. $t$}\}$, being a closed linear subspace of $L^2_{H^1(\C)}$ (thanks to the continuity of $\mathbb{T}$ and $y \mapsto \mean{d(t,y)}$), is a separable Hilbert space. Define $J\colon \hat X \to \mathbb{R}$ by
\[J(d) = \frac{1}{2}\int_0^T\int_{\C(t)} |\nablabgt d(t)|^2 + \int_0^T \int_{\C(t)}\nablabgt U(t)\nablabgt d(t),\]
and observe that $J$ is coercive through the use of Poincar\'e's and Young's inequalities. 
%\begin{align*}
%J(d) %&\geq \frac{1}{2}\int_0^T\int_{\C(t)} |\nabla_{\Gamma} d(t)|^2 + d_y(t)^2 - \int_0^T \int_{\C(t)}C_\epsilon|\nablabgt U(t)|^2 + \frac{\epsilon}{2}|\nablabgt d(t)|^2\\
%&\geq C_1\int_0^T\int_{\C(t)} d(t)^2 + |\sgradt d(t)|^2 + d_y(t)^2 - \int_0^T \int_{\C(t)}C_\epsilon|\nablabgt U(t)|^2 + \frac{\epsilon}{2}|\nablabgt d(t)|^2\\
%&\geq C_2\int_0^T\int_{\C(t)} d(t)^2 + |\nablabgt d(t)|^2 - C_3\int_0^T \int_{\C(t)}|\nablabgt U(t)|^2\\
%&\geq C_2\norm{d}{L^2_{H^1(\C)}}^2 - C_3\norm{\nablabgt U}{L^2_{L^2(\C)}}^2.
%\end{align*}
%which clearly implies $J(d_n) \to \infty$ if $\norm{d_n}{L^2_{H^1(\C)}} \to \infty$. 
Since $J$ is also continuous, by \cite[Theorem 5.25]{demengel2012functional}, $J$ has a unique minimiser $d$ satisfying $J'(d,w) = (\nablabg  d + \nablabg U, \nablabg w)_{L^2_{L^2(\C)}}=0$ %and it satisfies $J'(d, w) =0$ for all $w \in \hat X$, that is,
%\[J'(d,w) = \int_0^T\int_{\C(t)} \nablabgt  d(t)\nablabgt w(t) + \int_0^T \int_{\C(t)}\nablabgt U(t)\nablabgt w(t)=0\]
for all $w \in \hat X$. Recalling $v=d+U$, we find that $v \in L^2_{H^1(\C)}$ with $\mathbb{T}v = u$ and $\mean v = 0$ satisfies
\[\int_0^T\int_{\C(t)} \nablabgt  v(t)\nablabgt w(t) =0\quad\text{for all $w \in \hat X$.}\]
To remove the mean value condition on the test functions, let $\eta \in L^2_{H^1(\C)}$ with $\mathbb{T}\eta = 0$ and test with $w(t) := \eta(t) - \mean{\eta(t)}$ (this satisfies $\mathbb{T}w =0$ and $\mean{w(t)} = \mean{\eta(t)} - \mean {\eta(t)}=0$, so is admissible):
\begin{align*}
0 %&= \int_0^T\int_{\C(t)} \nablabgt  v(t)\nablabgt \eta(t) - \int_0^T\int_{\C(t)} v_y(t)\mean \eta_y(t)\\
%&= \int_0^T\int_{\C(t)} \nablabgt  v(t)\nablabgt \eta(t) - \int_0^T\int_0^\infty\int_{\Gamma(t)}  v_y(t)\partial_y\left(\frac{1}{|\Gamma(t)|}\int_{\Gamma(t)}\eta(t)\right)\\
&= \int_0^T\int_{\C(t)} \nablabgt  v(t)\nablabgt \eta(t) - \int_0^T\int_0^\infty\partial_y\left(\frac{1}{|\Gamma(t)|}\int_{\Gamma(t)} v(t)\right)\partial_y\left(\int_{\Gamma(t)}\eta(t)\right)\\
&= \int_0^T\int_{\C(t)} \nablabgt  v(t)\nablabgt \eta(t)
\end{align*}
since $\mean{v(t)} = 0$ for a.a. $t$ and all $y$. This settles the problem for the case $\mean{u(t)} = 0$. For general $u \in L^2_{W^{1\slash 2,2}}$, 
%\end{proof}
%\begin{cor}Given $u \in L^2_{W^{1\slash 2,2}}$ (with any mean value), there exists a map $\ext{\mathbb{E}}\colon L^2_{W^{1\slash 2,2}} \to L^2_{X(\C)}$ such that $\Delta \ext{\mathbb{E}}u = 0$, $\ext{\mathbb{T}}\ext{\mathbb{E}}u = u$ and $\frac{1}{|\Gamma(t)|}\int_{\Gamma(t)} (\ext{\mathbb{E}}u)(t) = \mean u(t)$. 
%\end{cor}
%\begin{proof}
%Now suppose $u \in L^2_{W^{1\slash 2,2}}$ does not have mean value zero. 
define $\ext{\mathbb{E}}u := \mathbb{E}(u-\mean u) + \mean{u}\in L^2_{X(\C)}$ which satisfies $\frac{1}{|\Gamma(t)|}\int_{\Gamma(t)}\ext{\mathbb{E}}u(t) = \mean{u(t)}$ and
\begin{align*}
\int_0^T\int_{\C(t)}\ext{\nablabgt} (\ext{\mathbb{E}}u)(t) \nablabgt \eta(t) = \int_0^T\int_{\C(t)}\nablabgt  (\mathbb{E}(u- \mean u))(t) \nablabgt \eta = 0\quad\text{for all $\eta \in L^2_{H^1(\C)}$ with $\mathbb T\eta = 0$}.
\end{align*}
%So we have shown the existence of a map $\ext{\mathbb{E}}\colon L^2_{W^{1\slash 2,2}} \to L^2_{X(\C)}$ such that $\Delta \ext{\mathbb{E}}u = 0$, $\ext{\mathbb{T}}\ext{\mathbb{E}}u = u$ and $\frac{1}{|\Gamma(t)|}\int_{\Gamma(t)}\ext{\mathbb{E}}u(t) = \mean u(t)$. 
\end{proof}
We need to elucidate the link between $\mathbb{E}$ and the family of maps $\{\mathcal{E}_t\}_{t \in [0,T]}$ from Definition \ref{defn:harmonicExtensionsWitht}.
\begin{lem}\label{lem:EagreesPointwise}Let $u \in L^2_{W^{1\slash 2,2}}$. For almost all $t$, $(\ext{\mathbb{E}}u)(t) = \ext{\mathcal{E}}_tu(t)$ in $X(\C(t))$.
\end{lem}
\begin{proof}
%We begin with
%\begin{align*}
%\int_0^T \int_{\C(t)} \nablabgt (\ext{\mathbb{E}}u)(t)\nablabgt \eta(t) = 0 \quad \text{for all $\eta \in L^2_{H^1(\C)}$ with $\mathbb{T}\eta = 0$}.
%\end{align*}
Pick $\psi \in C_c^\infty(0,T)$ and $v_0 \in H^1(\C_0)$ with $\mathcal{T}_0v_0 = 0$, then $\psi \phi_{\C,t}v_0 \in L^2_{H^1(\C)}$ with $\mathbb{T}(\psi\phi_{\C,t}v_0) = 0$, so it is an admissible test function in \eqref{eq:harmonicExtensionL2XWeakForm} and testing with it gives %(via the fundamental lemma of the calculus of variations)
\[\int_{\C(t)} \nablabgt (\ext{\mathbb{E}}u)(t)\nablabgt \phi_{\C,t} v_0 = 0 \quad \text{for all $v_0 \in H^1(\C_0)$ with $\mathcal{T}_0v_0 = 0$, for almost all $t$.}\]
By the homeomorphism properties of $\phi_{\C,t}$, this is same as
\[\int_{\C(t)} \nablabgt (\ext{\mathbb{E}}u)(t)\nablabgt v_t = 0 \quad \text{for all $v_t \in H^1(\C(t))$ with $\mathcal{T}_tv_t = 0$, for almost all $t$,}\]
and since also $\ext{\mathcal{T}}_t(\ext{\mathbb{E}}u(t)) = u(t)$, we have $(\ext{\mathbb{E}}u)(t) = \ext{\mathcal{E}}_tu(t)$ by the uniqueness in Theorem \ref{thm:harmonicExtensionExistence}.
\end{proof}
Thanks to the the previous lemma, we can use the bound \eqref{eq:harmonicExtensionL2Bound} and Lemma \ref{lem:boundOnGradientEL2C} in conjunction with the eigenvalue estimate \eqref{eq:eigenvalueEstimate} to obtain the next result.
\begin{cor}\label{lem:boundsOnMathbbE}We have for $u \in L^2_{W^{1\slash 2,2}}$ with $\mean u=0$,
\[\norm{\mathbb{E}u}{L^2_{L^2(\C)}} \leq C\norm{u}{L^2_{L^2}}\quad\text{and}\quad \norm{\nablabg \mathbb{E}u}{L^2_{L^2(\C)}} \leq C\norm{u}{L^2_{W^{1\slash 2,2}}}.\]
\end{cor}
\subsection{The truncated harmonic extension of $u \in L^2_{W^{1\slash 2, 2}}$}
\begin{theorem}[The truncated harmonic extension problem in the space $L^2_{H^1(\C_R)}$]\label{thm:wellPosednessTruncatedHarmonicExtensionL2X}
%Given $u \in L^2_{W^{1\slash 2, 2}}$, 
There exists a map $\ext{\mathbb{E}}_R\colon L^2_{W^{1\slash 2, 2}} \to L^2_{H^1(\C_R)}$ such that given $u \in L^2_{W^{1\slash 2, 2}}$, $\ext{\mathbb{E}}_R u$ is the unique weak solution of \eqref{eq:truncatedHarmonicExtensionL2XStrongForm} satisfying $\mathbb{T}_{R,y=0}v = u$ and $\mathbb{T}_{R,y=R} =0$ in $L^2_{W^{1\slash 2,2}}$ and
\begin{equation}\label{eq:truncatedHarmonicExtensionL2XWeakForm}
\begin{aligned}
\int_0^T \int_{\C_R(t)}\nablabgt v(t) \nablabgt \eta(t) &= 0 &&\text{for all $\eta \in L^2_{H^1_0(\C_R)}$},
\end{aligned}
\end{equation}
and $\frac{1}{|\Gamma(t)|}\int_{\Gamma(t)} (\ext{\mathbb{E}}_Ru)(t) = \mean u(t)$. When $\mean{u(t)}=0$ for a.a. $t$, we write the solution as $\mathbb{E}_Ru$. %The map $\ext{\mathbb{E}}_R$ satisfies $\ext{\mathbb{E}}u = \mathbb{E}(u-\mean u) + \mean u$. 
\end{theorem}
\begin{proof}
We transform \eqref{eq:truncatedHarmonicExtensionL2XStrongForm} to having zero boundary conditions by setting $w=v-{(R-y)\ext{\mathbb{E}}u}\slash {R} \in L^2_{H^1(\C_R)}$; then 
\begin{equation*}
\begin{aligned}
\Deltabgt w(t) &= -\frac{1}{R}\Deltabgt\left((R-y)\ext{\mathbb{E}}u(t)\right) &&\text{on $\C_R(t)$}\\
w(t,x,0) &= w(t,x,R) = 0,
\end{aligned}
\end{equation*}
which, by Lax--Milgram, has a unique solution $w \in L^2_{H^1_0(\C_R)}$ satisfying
\begin{equation*}\label{eq:weakFormTruncatedZeroInitialDataL2L2}
\int_0^T \int_{\C_R(t)}\nablabgt w(t) \nablabgt \eta(t) = -\int_0^T \int_{\C_R(t)}\nablabgt \left(\frac{R-y}{R}\ext{\mathcal{E}}_t(u(t))\right)\nablabgt \eta(t)\quad \forall \eta \in L^2_{H^1_0(\C_R)}.
\end{equation*}
Indeed, define the bounded and coercive bilinear form $a\colon L^2_{H^1_0(\C_R)}\times L^2_{H^1_0(\C_R)} \to \mathbb{R}$ by the left hand side of the above equality %It is clearly bounded and coercive. %follows from Poincar\'e's inequality which holds for the following reason. Since 
%\[\norm{\eta}{L^2_{H^1_0(\C_R)}}^2 = \int_0^T  \int_{\Gamma(t)}\int_0^R|\eta(t)|^2 + |\sgrad \eta(t)|^2 + |\eta_y(t)|^2 < \infty\]
%(the interchange of integrals over $[0,R]$ and $\Gamma(t)$ is justified for the following reason. Suppose $w \in H^1(\C_R(t))$. This means that $w$, $\sgrad w$ and $w_y$ belong to $L^2((0,R)\times \Gamma(t))$, so are measurable in the product space, and Fubini--Tonelli allows to integrate in any order) it follows that for a.a. $t$, a.a. $x$, $\eta(t,x,\cdot) \in H^1_0(0,R)$ (recall that $\eta$ vanishes at $y=R$ and $y=0$). Thus the Poincar\'e inequality on $[0,R]$ implies
%\[\int_0^R |\eta(t,x)|^2 \leq C_P \int_0^R |\eta_y(t,x)|^2.\]
%Using this fact in the definition of the norm of $L^2_{H^1_0(\C_R)}$ gives $\norm{\eta}{L^2_{H^1_0(\C_R)}} \leq C\norm{\nablabg \eta}{L^2_{L^2(\C_R)}}$ so that $\norm{\nablabg \cdot}{L^2_{L^2(\C_R)}}$ is an equivalent norm on $L^2_{H^1_0(\C_R)}.$ 
and define $l\colon L^2_{H^1_0(\C_R)} \to \mathbb{R}$ by the right hand side, which is a bounded linear functional due to \eqref{eq:harmonicExtensionL2Bound} and Lemma \ref{lem:boundOnGradientEL2C}. %, we see that $l$ is in the dual space of $L^2_{H^1_0(\C_R)}$. 
%\begin{align*}
%|l(\eta)| %&= \int_0^T \int_0^R \int_{\Gamma(t)}\frac{R-y}{R}\nablabgt \ext{\mathcal{E}}_t(u(t))\nablabgt \eta(t) + \ext{\mathcal{E}}_t(u(t))\nablabgt \left(\frac{R-y}{R}\right)\nablabgt \eta(t) \\
%&\leq \int_0^T \int_{\C(t)}2|\nablabgt \ext{\mathcal{E}}_t(u(t))||\nablabgt \eta(t)| + \frac{1}{R}|\ext{\mathcal{E}}_t(u(t))||\partial_y \eta(t)|
%&\leq C(R)\left(\norm{\nablabg \ext{\mathbb{E}}u}{L^2_{L^2(\C_R)}} + \norm{\ext{\mathbb{E}}u}{L^2_{L^2(\C_R)}}\right)\norm{\nablabg \eta}{L^2_{L^2(\C_R)}}\\
%\leq C(R)\norm{\ext{\mathbb{E}}u}{L^2_{H^1(\C_R)}}\norm{\nablabg \eta}{L^2_{L^2(\C_R)}}.
%\end{align*}
%Therefore, Lax--Milgram gives the result. 
It follows that $\ext{\mathbb{E}}_Ru := v:= w+{(R-y)\ext{\mathbb{E}}(u)}\slash {R} \in L^2_{H^1(\C_R)}$
satisfies $\mathbb{T}_{R,y=0}v = u$, $\mathbb{T}_{R,y=R}=0$ and \eqref{eq:truncatedHarmonicExtensionL2XWeakForm}.
%\[\int_0^T \int_0^R \int_{\Gamma(t)}\nablabgt v(t) \nablabgt \eta(t) = 0\quad \forall \eta \in L^2_{H^1_0(\C_R)}.\]
\end{proof}
\begin{lem}\label{lem:ERagreesPointwise}Let $u \in L^2_{H^{1\slash 2}}$. For almost all $t$, $(\ext{\mathbb{E}}_Ru)(t) = \ext{\mathcal{E}_{R,t}}u(t)$ in $H^1(\C_R(t))$.
\end{lem}
This lemma follows just like Lemma \ref{lem:EagreesPointwise} since $\phi_{\C,t}\colon H^1_0(\C_R(0)) \to H^1_0(\C_R(t))$ is a homeomorphism, and it, along with Lemma \ref{lem:graddBoundOnTruncatedWithT}, implies the following.
\begin{cor}\label{lem:gradBoundTruncatedProblemL2H12}There exist constants $C_1$, $C_2 >0$ independent of $R$ such that
\begin{align*}
\norm{\ext{\mathbb{E}}_Ru}{L^2_{L^2(\C_R)}} &\leq C_1\norm{u}{L^2_{L^2}} + 2\sqrt{R|\Gamma|}\norm{\mean u}{L^2(0,T)}\\
\norm{\nablabg \ext{\mathbb{E}}_Ru}{L^2_{L^2(\C_R)}} &\leq C_2\norm{u}{L^2_{W^{1\slash 2, 2}}}\text{ if $R \geq 1$}.
\end{align*}
%where $C$ is independent of $R$.
\end{cor}
%\begin{proof}
%This easily follows from the previous lemma and Lemma \ref{lem:graddBoundOnTruncatedWithT}.% (since $R > 1$, the dependence on $R$ given in the latter lemma disappears).
%
%We have
%\begin{align*}
%\norm{\nablabg v}{L^2_{L^2(\C_R)}}^2 &= \norm{\nablabg w + \nablabg(\frac{R-y}{R}\ext{\mathbb{E}}u)}{L^2_{L^2(\C_R)}}^2\\
%&\leq 2\norm{\nablabg w}{L^2_{L^2(\C_R)}}^2 + 2\norm{\nablabg(\frac{R-y}{R}\ext{\mathbb{E}}u)}{L^2_{L^2(\C_R)}}^2.
%\end{align*}
%Now, the weak formulation for $w$ implies
%\begin{align*}
%\int_0^T \int_0^R \int_{\Gamma(t)}|\nablabgt w(t)|^2 &\leq \int_0^T \int_0^R \int_{\Gamma(t)}|\nablabgt \left(\frac{R-y}{R}\ext{\mathcal{E}}_t(u(t))\right)||\nablabgt w(t)|\\
%&\leq\frac{1}{2}\int_0^T \int_0^R \int_{\Gamma(t)}|\nablabgt \left(\frac{R-y}{R}\ext{\mathcal{E}}_t(u(t))\right)|^2 + |\nablabgt w(t)|^2
%\end{align*}
%which is
%\begin{align*}
%\int_0^T \int_0^R \int_{\Gamma(t)}|\nablabgt w(t)|^2 &\leq\int_0^T \int_0^R \int_{\Gamma(t)}\left|\frac{R-y}{R}\nablabgt \ext{\mathcal{E}}_t(u(t)) - \frac{1}{R}\ext{\mathcal{E}}_t(u(t))\right|^2\\
%&\leq2\int_0^T \int_0^R \int_{\Gamma(t)}|\nablabgt \ext{\mathcal{E}}_t(u(t))|^2 + |\ext{\mathcal{E}}_t(u(t))|^2\\
%&\leq C\norm{u}{L^2_{H^{1\slash 2}}}^2
%\end{align*} 
%where we replaced the interval of integration $(0,R)$ with $(0,\infty)$ and used the estimates for $\mathbb{E}$. Therefore,
%\begin{align*}
%\norm{\nablabg v}{L^2_{L^2(\C_R)}} &\leq C\norm{u}{L^2_{H^{1\slash 2}}}.
%\end{align*}
%\end{proof}
A third way to interpret the map $\mathscr{Z}_R$ from \eqref{defn:Z} is as a map $\mathscr{Z}_R\colon \{ \eta \in L^2_{H^1(\C_R)} \mid \mathbb{T}_{R,y=R}\eta = 0\} \to L^2_{H^1(\C)}$, and again this is an isometry.
\begin{lem}\label{lem:convergenceTruncationsGradientL2L2}We have $\mathscr{Z}_R\ext{\mathbb{E}}_R u \to \ext{\mathbb{E}}u$ in $L^2_{X(\C)}$.
\end{lem}
\begin{proof}
Lemma \ref{lem:expDecay} gives for almost all $t$ $\norm{\nablabgt \mathscr{Z}_R\ext{\mathcal{E}}_{R,t} u(t) - \nablabgt \ext{\mathcal{E}}_tu(t)}{L^2(\C(t))} \to 0$ and to use the DCT it suffices to find an integrable uniform in $R$ bound on the above norm which follows from %. We have, using %Lemmas \ref{lem:convergenceTruncationsL2} and
Lemma \ref{lem:graddBoundOnTruncatedWithT}.
\end{proof}
\section{The non-degenerate problem: proof of Theorem \ref{thm:wellPosednessOfNonDegenerateProblem}}\label{sec:nondegenerateProblem}
Let $\beta\colon \mathbb{R} \to \mathbb{R}$ be a function satisfying \eqref{eq:assumptionsOnBeta} on p.~\pageref{eq:assumptionsOnBeta}. We will prove Theorem \ref{thm:wellPosednessOfNonDegenerateProblem} in this section, that of the well-posedness of problem \eqref{eq:betaProblem}. For easier reading, we will shorten the duality products \label{notation:dualityProducts} $\langle \cdot, \cdot \rangle_{W^{-1\slash 2, 2}(\Gamma(t)), W^{1\slash 2, 2}(\Gamma(t))}$ to $\langle \cdot, \cdot \rangle_{}$ (an abuse of notation) and $\langle \cdot, \cdot \rangle_{W^{-1\slash 2, 2}(\Gamma_0), W^{1\slash 2, 2}(\Gamma_0)}$ to $\langle \cdot, \cdot \rangle_0.$
%Let $\beta\colon \mathbb{R} \to \mathbb{R}$ be $C^2$ and Lipschitz with $\beta(0) = 0,$ $%\beta', (\beta^{-1})', (\beta^{-1})'' \in L^\infty(\mathbb{R})$, and
%\begin{align*}
%\beta' &\geq C_{\beta'}\\
%(\beta^{-1})' &\geq C_{\beta'_{inv}} 
%\end{align*}
%for constants $C_{\beta'}$ and $C_{\beta'_{inv}}.$

\subsection{Existence of solutions to the truncated problem}\label{sec:truncatedProblem}
This subsection is devoted to the proof of the following theorem. 
\begin{theorem}\label{thm:existenceTruncatedProblem}
For each $R \geq 1$, there exists a unique weak solution $u_R \in \mathbb{W}(W^{1\slash 2,2}, W^{-1\slash 2,2})$ to \eqref{eq:betaRProblem} with $\nablabg\ext{\mathbb{E}}_R(\beta(u_R)) \in L^2_{L^2(\C_R)}$ and $u_R(0)=u_0$ satisfying 
\[\int_0^T\langle \dot u_R(t), \eta(t)\rangle_{} + \int_0^T\int_{\Gamma(t)}u_R(t)\eta(t) \sgradt \cdot \mathbf w(t) + \int_0^T\int_{\C_R(t)}\nablabgt \ext{\mathcal{E}}_{R,t}(\beta(u_R(t)))\nablabgt (E_R\eta)(t) = 0\]
for all $\eta \in L^2_{W^{1\slash 2,2}}$, where $E_R\eta \in L^2_{H^1(\C_R)}$ satisfies $\mathbb{T}_{R,y=0}E_R\eta = \eta$ and $\mathbb{T}_{R,y=R}E_R\eta = 0$.
\end{theorem}
We hide the subscript $R$ in $u_R$ and write just $u$ for simpler notation. Define $a_R(t;\cdot,\cdot)\colon W^{1\slash 2,2}(\Gamma(t))\times W^{1\slash 2,2}(\Gamma(t)) \to \mathbb{R}$ by
\[a_R(t;u,\eta) = \int_{\C_R(t)}\nablabgt \ext{\mathcal{E}}_{R,t}(\beta(u))\nablabgt E_R(t)\eta\]
where $E_R(t)\colon W^{1\slash 2,2}(\Gamma(t)) \to H^1(\C_R(t))$ is an (arbitrary) extension that satisfies $\mathcal{T}_{R,t,y=0}(E_R(t)\eta) = \eta$ and $\mathcal{T}_{R,t,y=R}(E_R(t)\eta) = 0$; the choice of $E_R$ does not matter (see Remark \ref{rem:truncatedHarmonicExtensionArbitrary}).
%since for any two such extensions $E_R$ and $\tilde E_R$
%\begin{align*}
%\int_0^R  \int_{\Gamma(t)}\nablabgt \ext{\mathcal{E}}_{R,t}(\beta(u))\nablabgt E_R(t)\eta &- \int_0^R  \int_{\Gamma(t)}\nablabgt \ext{\mathcal{E}}_{R,t}(\beta(u))\nablabgt \tilde E_R(t)\eta\\
%& = \int_0^R  \int_{\Gamma(t)}\nablabgt \ext{\mathcal{E}}_{R,t}(\beta(u))\nablabgt (E_R(t)\eta-\tilde E_R(t)\eta)\\
%&= 0
%\end{align*}
%by definition of the weak solution, because $E_R(t)\eta-\tilde E_R(t)\eta \in H^1(\C_R(t))$ and $\mathcal{T}_{R,t,y=0}(E_R(t)\eta-\tilde E_R(t)\eta) = \eta-\eta = 0$ and $\mathcal{T}_{R,t,y=R}(E_R(t)\eta-\tilde E_R(t)\eta) = 0$, i.e., $E_R(t)\eta-\tilde E_R(t)\eta \in H^1_0(\C_R(t))$.
To derive the Galerkin approximation, we pull back the first two terms in the equation onto $\Gamma_0$ and then make a substitution in order to put the Jacobian term $J^0_t$ onto the elliptic form. 
%We define $\tilde u = \phi_{\Gamma, -(\cdot)}u$ and $\tilde \eta = \phi_{\Gamma, -(\cdot)}\eta$ and rewrite the equation
%%\[\int_{\Gamma(t)}\dot u(t) \eta(t) + \int_{\Gamma(t)}u(t)\eta(t) \sgrad \cdot \mathbf w + a_R(t;u(t),\eta(t)) = 0\]
%in terms of $\tilde u$ and $\tilde \eta$:
%\[\int_{\Gamma_0}\tilde u'(t) \tilde \eta(t) J^0_t + \int_{\Gamma_0}\tilde u(t)\tilde \eta(t) J^0_t \phi_{-t}(\sgrad \cdot \mathbf w) + a_R(t;\phi_t \tilde u(t),\phi_t \tilde \eta(t)) = 0\]
%(for ease of reading we wrote $\phi_t$ instead of %$\phi_{\Gamma, t}$). Now substitute $\tilde \psi = \tilde \eta J^0_t$ and use $1\slash {\phi_tJ^0_t} = J^t_0$:
%\[\int_{\Gamma_0}\tilde u'(t) \tilde \psi(t)+ \int_{\Gamma_0}\tilde u(t)\tilde \psi(t) \phi_{-t}(\sgrad \cdot \mathbf w) + a_R(t;\phi_t \tilde u(t),J^t_0\phi_t \tilde \psi(t)) = 0.\]
%We seek a Galerkin approximation of this equation. 
Let $\{b_j\}$ be an orthonormal basis of $L^2(\Gamma_0)$ that is orthogonal in $W^{1\slash 2,2}(\Gamma_0)$ and let $\tilde u(t) = \phi_{\Gamma, -t}u(t)$. The Galerkin approximation is the system
\begin{equation}\label{eq:galerkinEquation}
\begin{aligned}
\int_{\Gamma_0}\tilde u_n'(t) b_j+ \int_{\Gamma_0}\tilde u_n(t)b_j\phi_{-t}(\sgradt \cdot \mathbf w(t)) + a_R(t;\phi_t \tilde u_n(t),J^t_0\phi_t b_j) &= 0 \quad\forall j=1, ..., n\\
\tilde u_n(0) &= \tilde u_{0n}
\end{aligned}
\end{equation}
for an ansatz $\tilde u_n(t) = \sum_{i=1}^n \alpha_i(t)b_i$ with unknown coefficients $\alpha_i = \alpha_i^n$ and $\tilde u_{0n} \in V_n(0) := \text{span}\{b_1, ..., b_n\}$ is such that $\tilde u_{0n} \to u_0$ in $W^{1\slash 2,2}(\Gamma_0)$ and $\norm{\tilde u_{0n}}{W^{1\slash 2,2}(\Gamma_0)} \leq C\norm{u_0}{W^{1\slash 2,2}(\Gamma_0)}$.
\begin{remark}\label{rem:projectionOperator}We pulled back the equation onto a reference domain in order to facilitate the procurement of a bound on $\tilde u_n'$ which is needed for a strong convergence result. This transformation to the reference domain $\Gamma_0$ could have been avoided if we knew that the orthogonal projection operator $P_n^t \colon L^2(\Gamma(t)) \to V_n(t):= \phi_t(V_n(0))$ defined by 
\[(P_n^t u-u, v_n)_{L^2(\Gamma(t))} =0\quad\text{for all $v_n \in V_n(t)$}\]
is bounded as a map $P_n^t\colon V(t) \to V(t)$. Such a bound is true when $t=0$ because of the special choice of basis functions, but for arbitrary $t$ the desired bound appears elusive. Of course, such a result would be of fundamental use generally in evolutionary equations on evolving domains.
\end{remark}
\begin{lem}The Galerkin equation \eqref{eq:galerkinEquation} has a solution $\tilde u_n \in H^1(0,T;V_n(0))$. %\textcolor{red}{check it's the right space}
\end{lem}
\begin{proof}
The equation \eqref{eq:galerkinEquation} leads to
%\[\sum_{i=1}^n \alpha_i'(t)\int_{\Gamma_0}b_ib_j+ \sum_{i=1}^n \alpha_i(t)\int_{\Gamma_0}b_ib_j\phi_{-t}(\sgrad \cdot \mathbf w) + a_R(t;\sum_{i=1}^n \alpha_i(t)\phi_t b_i,J^t_0\phi_t b_j) = 0 \quad\forall j=1, ..., n\]
%and using the orthonormality of the basis
\begin{equation}\label{eq:galerkin1}
\begin{aligned}
\alpha_j'(t)+ \sum_{i=1}^n \alpha_i(t)\int_{\Gamma_0}b_ib_j\phi_{-t}(\sgradt \cdot \mathbf w(t)) + a_R(t;\sum_{i=1}^n \alpha_i(t)\phi_t b_i,J^t_0\phi_t b_j) &= 0 && \forall j=1, ..., n\\
\alpha_j(0) &= (\tilde u_{0n}, b_j)_{L^2(\Gamma_0)}  &&\forall j=1, ..., n.
\end{aligned}
\end{equation}
Define $\bm{\alpha}(t)=(\alpha_1(t), ..., \alpha_n(t))^{\intercal},$ $\mathbf{b}(t) = (\phi_tb_1, \hdots, \phi_tb_n)^\intercal,$ $\mathbf{a}(t,\bm{\alpha}) = (a_R(t;\bm{\alpha}\cdot \bm{b}(t),J^t_0\phi_t b_1), \hdots, a_R(t;\bm{\alpha}\cdot \bm{b}(t),J^t_0\phi_t b_n))^{\intercal}$ and the matrix $(\mathbf{W}(t))_{ij} = \int_{\Gamma_0}b_jb_i\phi_{-t}(\sgradt \cdot \mathbf w(t)).$
%\[\mathbf{a}(t,\bm{\alpha}(t)) = \begin{pmatrix}a_R(t;\sum_{i=1}^n \alpha_i(t)\phi_t b_i,J^t_0\phi_t b_1)\\\vdots\\a_R(t;\sum_{i=1}^n \alpha_i(t)\phi_t b_i,J^t_0\phi_t b_n) \end{pmatrix}=\begin{pmatrix}a_R(t;\bm{\alpha}(t)\cdot \bm{b}(t),J^t_0\phi_t b_1)\\\vdots\\a_R(t;\bm{\alpha}(t)\cdot \bm{b}(t),J^t_0\phi_t b_n) \end{pmatrix}.\]
The system of equations \eqref{eq:galerkin1} is then written
\[\bm{\alpha}'(t) = F(t,\bm{\alpha}(t)) := - \mathbf{W}(t)\bm{\alpha}(t) - \mathbf{a}(t,\bm{\alpha}(t))\]
with initial data $\bm{\alpha}(0) = ((\tilde u_{0n}, b_1)_{L^2(\Gamma_0)}, ..., (\tilde u_{0n}, b_n)_{L^2(\Gamma_0)})^{\intercal}$. We need to show that $t \mapsto F(t,\bm{\alpha})$ is measurable for fixed $\bm \alpha \in \mathbb{R}^n$. The term with the matrix is clear. For the remaining term, we have
\begin{align*}
a_R(t;\bm{\alpha} \cdot \bm{b}(t),J^t_0\phi_t b_j) &= \int_0^R  \int_{\Gamma(t)}\nablabgt \ext{\mathcal{E}}_{R,t}(\beta(\bm{\alpha}\cdot \bm{b}(t)))\nablabgt E_R(t)(J^t_0\phi_t b_j)\\
%&= \int_0^R  \int_{\Gamma(t)}\sgradt \ext{\mathcal{E}}_{R,t}(\beta(\bm{\alpha}\cdot \bm{b}(t)))\sgradt E_R(t)(J^t_0\phi_t b_j) + \partial_y\ext{\mathcal{E}}_{R,t}(\beta(\bm{\alpha}\cdot \bm{b}(t)))\partial_y E_R(t)(J^t_0\phi_t b_j)\\
&= \int_0^R  \int_{\Gamma_0}J^0_t\sgrads \phi_{-t}[\ext{\mathcal{E}}_{R,t}(\beta(\bm{\alpha}\cdot \bm{b}(t)))](\mathbf{D}\Phi^0_t)^{-1}(\mathbf{D}\Phi^0_t )^{-\intercal} \sgrads \phi_{-t}[E_R(t)(J^t_0\phi_t b_j)]\\
&\quad + \int_0^R  \int_{\Gamma_0}J^0_t\partial_y\phi_{-t}\ext{\mathcal{E}}_{R,t}(\beta(\bm{\alpha}\cdot \bm{b}(t)))\partial_y \phi_{-t}E_R(t)(J^t_0\phi_t b_j),
\end{align*}
and %since 
%\[\phi_{-t}\ext{\mathcal{E}}_{R,t}(\beta(\bm{\alpha}\cdot \bm{b}(t))) = \phi_{-t}{\mathcal{E}}_R(t)(\beta(\bm{\alpha}\cdot \bm{b}(t))- \mean{(\beta(\bm{\alpha}\cdot \bm{b}(t))}) + \frac{R-y}{R}\mean{(\beta(\bm{\alpha}\cdot \bm{b}(t))}\]
%and 
we know that $\ext{\mathcal{E}}_{R,t}(\beta(\bm{\alpha}\cdot \bm{b}(t))) = \ext{\mathbb{E}}_R(\beta(\bm{\alpha}\cdot \bm{b}))(t)$ for almost all $t$ (Lemma \ref{lem:ERagreesPointwise}), and the pullback of the latter is measurable as a function of $t$ since $\ext{\mathbb{E}}_R(\beta(\bm{\alpha}\cdot \bm{b})) \in L^2_{H^1(\C_R)}$; the same argument can be used to deal with the $E_R(t)$ term. Now suppose that $\bm \alpha^j \to \bm\alpha$ in $\mathbb{R}^n$. We see that
\begin{align*}
\norm{\mathbf{a}(t,\bm{\alpha}^j)-\mathbf{a}(t,\bm{\alpha})}{\mathbb{R}^n}^2 &= \sum_i |a_R(t;\bm{\alpha}^j\cdot \bm{b}(t),J^t_0\phi_t b_i)-a_R(t;\bm{\alpha}\cdot \bm{b}(t),J^t_0\phi_t b_i)|^2\\
&\leq \sum_{i}\left(\int_0^R  \int_{\Gamma(t)}|\nablabgt \ext{\mathcal{E}}_{R,t}[\beta(\bm{\alpha}^j\cdot \bm{b}(t))-\beta(\bm{\alpha}\cdot \bm{b}(t))]\nablabgt E_R(t)(J^t_0\phi_t b_i)|\right)^2\\
&\leq C(R)\norm{\beta(\bm{\alpha}^j\cdot \bm{b}(t))-\beta(\bm{\alpha}\cdot \bm{b}(t))}{W^{1\slash 2,2}(\Gamma(t))}^2\sum_{i}\norm{\nablabgt E_R(t)(J^t_0\phi_t b_i)}{L^2(\C_R)}^2%\\
\end{align*}
by Lemma \ref{lem:graddBoundOnTruncatedWithT}, and this tends to zero by Lemma \ref{lem:nemytskiiH12} since $\bm{\alpha}^j\cdot \bm{b}(t) \to \bm{\alpha}\cdot \bm{b}(t)$ in $W^{1\slash 2,2}(\Gamma(t))$. Therefore, $\bm{\alpha} \mapsto \mathbf{a}(t,\bm\alpha)$ is continuous and so $F$ is a Carath\'eodory function. The uniform bound that we shall derive in the next subsection shows that $\norm{\bm{\alpha}(t)}{\mathbb{R}^n} \leq c$ for all $t$ if $\bm{\alpha}$ satisfies the ODE \eqref{eq:galerkin1}. Let us now prove that there exists $f \in L^1(0,T)$ with $\norm{F(t,\bm{\alpha})}{\mathbb{R}^n} \leq f(t)$ for every $\bm{\alpha} \in \{ \bm{\alpha} \in \mathbb{R}^n \mid \norm{\bm{\alpha}}{\mathbb{R}^n} \leq 2c\}$. %Let $\norm{\cdot}{F}$ denote the Frobenius matrix norm (i.e. the Euclidean norm) on $\mathbb{R}^n\times \mathbb{R}^n$. %, and let $\norm{\cdot}{\mathbb{R}^n}$ denote the Euclidean norm on $\mathbb{R}^n$. 
Note that
\begin{align*}
\norm{\mathbf{a}(t,\bm\alpha)}{\mathbb{R}^n}^2 
&\leq \sum_j \norm{\nablabgt \ext{\mathcal{E}}_{R,t}(\beta(\bm \alpha \cdot \bm b(t)))}{L^2(\C_R(t))}^2\norm{\nablabgt E_R(t)(J^t_0\phi_t b_j)}{L^2(\C_R(t))}^2\\
&\leq C_1\sum_j \norm{\beta(\bm \alpha \cdot \bm b(t))}{W^{1\slash 2,2}(\Gamma(t))}^2\norm{\nablabgt E_R(t)(J^t_0\phi_t b_j)}{L^2(\C_R(t))}^2\tag{by Lemma \ref{lem:graddBoundOnTruncatedWithT}}\\
&\leq C_1\norm{\beta'}{\infty}^2 \norm{\bm{\alpha}}{\mathbb{R}^n}^2\sum_i\norm{b_i(t)}{W^{1\slash 2,2}(\Gamma(t))}^2\sum_j\norm{\nablabgt E_R(t)(J^t_0\phi_t b_j)}{L^2(\C_R(t))}^2\\ 
&\leq C_2\norm{\bm \alpha}{\mathbb{R}^n}^2 \sum_{i,j} \norm{\phi_t b_i}{W^{1\slash 2,2}(\Gamma(t))}^2\norm{\nablabgt E_R(t)(J^t_0\phi_t b_j)}{L^2(\C_R(t))}^2
\end{align*}
so that overall (because the Frobenius norm $\norm{\cdot}{F}$ is compatible with the Euclidean vector norm),
\begin{align*}
\norm{F(t,\bm\alpha)}{\mathbb{R}^n} %&\leq \norm{\bm W(t)}{F}\norm{\bm\alpha}{\mathbb{R}^n} + \norm{\mathbf{a}(t,\bm\alpha)}{\mathbb{R}^n}\\
&\leq 2c\left(\norm{\bm W(t)}{F} + \sqrt{C_2}\sum_{i,j} \norm{\phi_t b_i}{W^{1\slash 2,2}(\Gamma(t))}\norm{\nablabgt E_R(t)(J^t_0\phi_t b_j)}{L^2(\C_R(t))}\right) =: f(t)
\end{align*}
and the term in the brackets on the right hand side is integrable over $(0,T)$. Now an application of the ODE theory in \cite[Problem 30.2]{ZeidlerIIB} gives global existence of a solution $\tilde u_n \colon [0,T] \to V_n(0)$.
\end{proof}
\subsubsection{Uniform estimates (in $n$)}
Multiply the first equality in \eqref{eq:galerkinEquation} by $\alpha_j(t)$ and sum up to get (using an arbitrary linear extension)
\[\int_{\Gamma_0}\tilde u_n'(t) \tilde u_n(t) + \int_{\Gamma_0}\tilde u_n(t)^2\phi_{-t}(\sgradt \cdot \mathbf w(t)) + a_R(t;\phi_t \tilde u_n(t),J^t_0\phi_t \tilde u_n(t)) = 0.\]
Now, in 
\[a_R(t;\phi_t \tilde u_n(t),J^t_0\phi_t \tilde u_n(t))=\int_{\C_R(t)}\nablabgt \ext{\mathcal{E}}_{R,t}(\beta(\phi_t \tilde u_n(t)))\nablabgt E_R(t)(J^t_0\phi_t \tilde u_n(t))\]
let us pick $E_R(t)(J^t_0\phi_t \tilde u_n(t)) = J^t_0\beta^{-1}(\ext{\mathcal{E}}_{R,t}(\beta(\phi_t \tilde u_n(t))))$, which is valid since 
$\mathcal{T}_{R,t,y=0}E_R(t)(J^t_0\phi_t \tilde u_n(t)) = J^t_0\beta^{-1}(\beta(\phi_t \tilde u_n(t))) = J^t_0\phi_t \tilde u_n(t)$ and $\mathcal{T}_{R,t,y=R}E_R(t)(J^t_0\phi_t \tilde u_n(t))= 0,$ and this gives
\begin{align*}
&\frac{1}{2}\frac{d}{dt}\int_{\Gamma_0}\tilde u_n(t)^2  + C_1C_{\beta'_{inv}} \int_{\C_R(t)}|\nablabgt \ext{\mathcal{E}}_{R,t}(\beta(\phi_t \tilde u_n(t)))|^2\\
&\quad\leq \frac{1}{2}\norm{\sgrad \cdot \mathbf w}{\infty}\int_{\Gamma_0}\tilde u_n(t)^2+ C_2\int_{\C_R(t)}C_{\epsilon}|\beta^{-1}(\ext{\mathcal{E}}_{R,t}(\beta(\phi_t \tilde u_n(t))))|^2 +  \epsilon|\sgrad \ext{\mathcal{E}}_{R,t}(\beta(\phi_t \tilde u_n(t)))|^2
\end{align*}
wherein we note that
\begin{align*}
\int_0^R\int_{\Gamma(t)}|\beta^{-1}(\ext{\mathcal{E}}_{R,t}(\beta(\phi_t \tilde u_n(t))))|^2 %&\leq \norm{(\beta^{-1})'}{\infty}^2\int_0^R\int_{\Gamma(t)}|\ext{\mathcal{E}}_{R,t}(\beta(\phi_t \tilde u_n(t)))|^2\\
&\leq C_3\norm{(\beta^{-1})'}{\infty}^2\norm{\beta(\phi_t \tilde u_n(t))}{L^2(\Gamma(t))}^2
%&\leq C_5(R,\lambda_1)\norm{(\beta^{-1})'}{\infty}^2\norm{\beta'}{\infty}^2\norm{\phi_t \tilde u_n(t)}{L^2(\Gamma(t))}^2\\
\leq C_4\norm{\tilde u_n(t)}{L^2(\Gamma_0)}^2
\end{align*}
by \eqref{eq:truncatedHarmonicExtensionL2Bound}, where $C_3$ and $C_4$ depend on $R$ and $\lambda_1$. Then Gronwall's inequality implies %This implies
\[\max_{t \in [0,T]}\norm{\tilde u_n(t)}{L^2(\Gamma_0)} + \norm{\nablabg \ext{\mathbb{E}}_R(\beta(\phi_{(\cdot)} \tilde u_n))}{L^2_{L^2(\C_R)}} \leq C.\]
%for a constant $C$ independent of $n$. 
\begin{remark}\label{rem:needForTruncationOfCylinder}
We needed to truncate the domain in order to obtain the previous bounds. If the domain was instead the full cylinder $\C(t)$, the extension of the test function would have to include a cut-off function so that it belongs to $L^2_{H^1(\C)}$, for example, if $\psi_\rho$ is as in Definition \ref{defn:cutoffFunction}, then we could choose
\[E(t)(J^t_0\phi_t \tilde u_n(t)) = J^t_0\beta^{-1}[{\mathcal{E}}_{t}(\beta(\phi_t \tilde u_n(t))-\mean{\beta(\phi_t \tilde u_n(t))}) + \psi_\rho\mean{\beta(\phi_t \tilde u_n(t))}]\]
but this leads to a residual term of the type 
\[\int_0^\infty \int_{\Gamma(t)}\beta^{-1}[{\mathcal{E}}_{t}(\beta(\phi_t \tilde u_n(t))-\mean{\beta(\phi_t \tilde u_n(t))}) + \psi_\rho\mean{\beta(\phi_t \tilde u_n(t))}]\sgrad \ext{\mathcal{E}}_{t}(\beta(\phi_t \tilde u_n(t)))\sgrad J^t_0\]
and we would have to make restrictive assumptions on the evolution to neglect this term as we send $\rho \to \infty$.
\end{remark}
Writing $\beta(u_n) = \mathbb{T}_{R,y=0} \ext{\mathbb{E}}_R(\beta(u_n))$ and using the trace inequality,
\begin{align*}
\norm{\beta(u_n)}{L^2_{W^{1\slash 2,2}}}^2 
\leq C_1\norm{\ext{\mathbb{E}}_R(\beta(u_n))}{L^2_{H^1(\C_R)}}^2
\leq C_2\norm{\beta'}{\infty}^2\norm{u_n}{L^{2}_{L^2}}^2 + C_3
\leq C_4
\end{align*}
by Corollary \ref{lem:gradBoundTruncatedProblemL2H12} ($C_2$ and $C_4$ will depend on $R$) and the energy estimates. Since $\beta^{-1}$ is Lipschitz, this implies
\[\norm{u_n}{L^2_{W^{1\slash 2,2}}} \leq C\]
independent of $n$ (using the boundedness result of Lemma \ref{lem:nemytskiiH12}). The bound on the time derivative follows too: take $\eta \in W^{1\slash 2, 2}(\Gamma_0)$, recall \eqref{eq:galerkinEquation} and that $P^0_n$ is self-adjoint:
\begin{align*}
\int_{\Gamma_0}\tilde u_n'(t)\eta %&= (P_n^0\tilde u_n'(t), \eta)_{L^2(\Gamma_0)} 
%&= (\tilde u_n'(t), P_n^0\eta)_{L^2(\Gamma_0)}\\
%&= -\int_{\Gamma_0}\tilde u_n(t)P_n^0(\eta)\phi_{-t}(\sgrad \cdot \mathbf w)- a_R(t;\phi_t \tilde u_n(t),J^t_0\phi_t P_n^0(\eta))\\
&= -\int_{\Gamma_0}\tilde u_n(t)P_n^0(\eta)\phi_{-t}(\sgradt \cdot \mathbf w(t))- \int_0^R  \int_{\Gamma(t)}\nablabgt \ext{\mathcal{E}}_{R,t}(\beta(\phi_t \tilde u_n(t)))\nablabgt E_R(t)(J^t_0\phi_t P_n^0(\eta))
%&\leq C_1\norm{\tilde u_n(t)}{L^2(\Gamma_0)}\norm{P_n^0(\eta)}{L^2(\Gamma_0)}\\
%&\quad + \norm{\nablabgt \ext{\mathcal{E}}_{R,t}(\beta(\phi_t \tilde u_n(t)))}{L^2(\C_R(t))}\norm{\nablabgt [J^t_0\ext{\mathcal{E}}_{R,t}(\phi_t P_n^0(\eta))]}{L^2(\C_R(t))}
\end{align*}
assuming a linear extension. Now picking $E_R(t)(J^t_0\phi_t P_n^0(\eta))=J^t_0\ext{\mathcal{E}}_{R,t}(\phi_t P_n^0(\eta))$, observe that
\begin{align*}
\norm{\nablabgt [J^t_0\ext{\mathcal{E}}_{R,t}(\phi_t P_n^0(\eta))]}{L^2(\C_R(t))}^2 %&= \int_0^R \int_{\Gamma(t)} |J^t_0\nablabgt \ext{\mathcal{E}}_{R,t}(\phi_t P_n^0(\eta)) + \ext{\mathcal{E}}_{R,t}(\phi_t P_n^0(\eta))\sgrad J^t_0|^2\\
&\leq C_1\int_0^R \int_{\Gamma(t)} |\nablabgt \ext{\mathcal{E}}_{R,t}(\phi_t P_n^0(\eta))|^2 + |\ext{\mathcal{E}}_{R,t}(\phi_t P_n^0(\eta))|^2\\
&\leq C_2\norm{\phi_t P_n^0(\eta)}{W^{1\slash 2,2}(\Gamma(t))}^2\tag{by Lemma \ref{lem:graddBoundOnTruncatedWithT} and \eqref{eq:truncatedHarmonicExtensionL2Bound}}%\\
%&\leq C_3\norm{P_n^0(\eta)}{W^{1\slash 2,2}(\Gamma_0)}^2
\end{align*}
(where again $C_2$ depends on $R$ and $\lambda_1$) which implies that
\begin{align*}
\int_0^T \langle \tilde u_n'(t), \eta \rangle_{0}
%&\leq C_5\int_0^T \norm{\tilde u_n(t)}{L^2(\Gamma_0)}\norm{P_n^0(\eta)}{L^2(\Gamma_0)}+ \norm{\nablabgt \ext{\mathcal{E}}_{R,t}(\beta(\phi_t \tilde u_n))}{L^2(\C_R(t))}\norm{P_n^0(\eta)}{W^{1\slash 2,2}(\Gamma_0)}\\
&\leq C_3\norm{P_n^0(\eta)}{L^2(0,T;W^{1\slash 2,2}(\Gamma_0))}\left(\norm{\tilde u_n}{L^2(0,T;L^2(\Gamma_0))} + \norm{\nablabgt \ext{\mathbb{E}}_R(\beta(\phi_t \tilde u_n))}{L^2_{L^2(\C_R)}}\right)\\
%&\leq C_6\norm{P_n^0(\eta)}{L^2(0,T;W^{1\slash 2,2}(\Gamma_0))}\\
&\leq C_4\norm{\eta}{L^2(0,T;W^{1\slash 2,2}(\Gamma_0))}
\end{align*}
by using the uniform estimates. Now taking the supremum over $\eta \in L^2(0,T;W^{1\slash 2,2}(\Gamma_0))$ shows that
\[\norm{\tilde u_n'}{L^2(0,T;W^{-1\slash 2,2}(\Gamma_0))} \leq C.\]
%where $C$ is independent of $n$.
\subsubsection{Passage to the limit in the Galerkin approximation}\label{sec:passToLimitGalerkin}
We have as $n \to \infty$
\begin{equation}\label{eq:convergencesGalerkin}
\begin{aligned}
\tilde u_n &\weaklyto \tilde u &&\text{in $L^2(0,T;W^{1\slash 2,2}(\Gamma_0))$}\\
\tilde u_n' &\weaklyto \tilde u' &&\text{in $L^2(0,T;W^{-1\slash 2,2}(\Gamma_0))$}\\
\tilde u_n &\to \tilde u &&\text{in $L^2(0,T;L^2(\Gamma_0))$}\\
\underline{D}_i \ext{\mathbb{E}}_R \beta (u_n) &\weaklyto \theta_i &&\text{in $L^2_{L^2(\C_R)}$}\\
\partial_y \ext{\mathbb{E}}_R \beta(u_n) &\weaklyto \theta_y &&\text{in $L^2_{L^2(\C_R)}$}
\end{aligned}
\end{equation}
where $\underline{D}_i = (\sgrad)_i$ is the $i$-th component of the tangential gradient and Aubin--Lions yielded the strong convergence. Therefore we have $u_n \to u$ in $L^2_{L^2}$ and $\beta(u_n) \to \beta(u)$ in $L^2_{L^2}$ thanks to the Lipschitz continuity of $\beta$. Using the boundedness of $\ext{\mathbb{E}}_R$ in the space $L^2_{L^2(\C_R)}$ from Corollary \ref{lem:gradBoundTruncatedProblemL2H12}, we obtain by linearity that%Due to the boundedness result \eqref{eq:truncatedHarmonicExtensionL2Bound} on $\ext{\mathcal{E}}_{R,t}$ with the constants in the bound depending on $\lambda_1(t)$ (which can be bounded), we see that
\begin{equation}\label{eq:conv1}
\ext{\mathbb{E}}_R(\beta(u_n)) \to \ext{\mathbb{E}}_R(\beta(u)) \quad\text{in $L^2_{L^2(\C_R)}$}.
\end{equation}
\paragraph{Identification of the spatial term}Take the test function 
\begin{equation}\label{eq:testFunction}
\eta(t,y,x) = \psi(t)(\phi_t v_0)(x) h(y)\quad\text{where $\psi \in C_c^\infty(0,T)$, $v_0 \in C_c^1(\Gamma_0)$ and $h \in C^\infty_c(0,R)$}
\end{equation} 
belonging to $L^2_{H^1(\C_R)}$ in the spatial integration by parts formula \cite[\S 2.1]{Alphonse2014a} on $\Gamma(t)$ integrated over $y$ and $t$:
\begin{align*}
\int_0^T \int_{\C_R(t)}(\underline{D}_i \ext{\mathbb{E}}_R\beta(u_n))\eta &= -\int_0^T \int_{\C_R(t)}(\ext{\mathbb{E}}_R\beta(u_n))\underline{D}_i\eta + \int_0^T \int_{\C_R(t)}(\ext{\mathbb{E}}_R\beta(u_n))\eta H\nu^\Gamma_i%\\
%&\to 
%-\int_0^T \int_{\C_R(t)}(\ext{\mathbb{E}}_R\beta(u))\underline{D}_i\eta + \int_0^T \int_{\C_R(t)}(\ext{\mathbb{E}}_R\beta(u))\eta H\nu^\Gamma_i
\end{align*}
where $H$ is the mean curvature. Using \eqref{eq:convergencesGalerkin} and \eqref{eq:conv1}, we have
\[\int_0^T \int_{\C_R(t)} \theta_i \eta = -\int_0^T \int_{\C_R(t)}(\ext{\mathbb{E}}_R\beta(u))\underline{D}_i\eta + \int_0^T \int_{\C_R(t)}(\ext{\mathbb{E}}_R\beta(u))\eta H\nu^\Gamma_i,\]
whence it follows that for almost every $t$, for almost every $y$, 
\[\int_{\Gamma(t)} \theta_i(t,y) \phi_t v_0= -\int_{\Gamma(t)}(\ext{\mathbb{E}}_R\beta(u))(t,y)\underline{D}_i\phi_t v_0 + \int_{\Gamma(t)}(\ext{\mathbb{E}}_R\beta(u))(t,y)\phi_t v_0 H(t)\nu^\Gamma_i(t).\]
Since this holds for all $\phi_t v_0 \in C_c^1(\Gamma_0)$, it also holds for all $v \in C_c^1(\Gamma(t))$, which implies that $\underline{D}_i(\ext{\mathbb{E}}_R\beta(u)) = \theta_i$ by definition.
\paragraph{Identification of the $y$ term}Again take $\eta \in L^2_{H^1(\C_R)}$ as in \eqref{eq:testFunction} and consider the integration by parts formula
\begin{align*}
\int_0^T \int_{\C_R(t)}(\partial_y \ext{\mathbb{E}}_R\beta(u_n))\eta &= -\int_0^T \int_{\C_R(t)}(\ext{\mathbb{E}}_R\beta(u_n))\partial_y\eta. %\\
%&\to -\int_0^T \int_{\C_R(t)}(\ext{\mathbb{E}}_R\beta(u))\partial_y\eta
\end{align*}
%This formula holds because
%\begin{align*}
%\int_0^T \int_{\C_R(t)}(\partial_y \ext{\mathbb{E}}_R\beta(u_n))\eta &= \int_0^T \psi(t)\int_0^R h(y) \int_{\Gamma(t)}\partial_y \ext{\mathbb{E}}_R\beta(u_n)\phi_t v_0\\
%&= \int_0^T \psi(t)\int_0^R h(y) \frac{d}{dy}\int_{\Gamma(t)} \ext{\mathbb{E}}_R\beta(u_n)\phi_t v_0\tag{because $\ext{\mathbb{E}}_R\beta(u_n) \in H^1(\C_R(t))$ so the inner product over $\Gamma(t)$ is absolutely continuous}\\
%&= -\int_0^T \psi(t)\int_0^R \partial_y h(y) \int_{\Gamma(t)} \ext{\mathbb{E}}_R\beta(u_n)\phi_t v_0.
%\end{align*}
As before, passing to the limit we find
%\[\int_0^T \int_{\C_R(t)} \theta_y \eta = -\int_0^T \int_{\C_R(t)}(\ext{\mathbb{E}}_R\beta(u))\partial_y\eta.\]
%It follows 
%that for almost every $t$
%\[ \int_{\Gamma(t)}\phi_t v_0 \int_0^R \theta_y h(y) = - \int_{\Gamma(t)}\phi_t v_0\int_0^R (\ext{\mathbb{E}}_R\beta(u))\partial_yh(y).\]
%which holds for all $C_c^1(\Gamma)$, implying 
for almost every $t$ and almost every $x$ (again since $\phi_t v_0$ ranges over all of $C_c^1(\Gamma(t))$) that
\[\int_0^R \theta_y(t,x) h = - \int_0^R (\ext{\mathbb{E}}_R\beta(u))(t,x)\partial_yh\]
and thus we identify $\theta_y = \partial_y \ext{\mathbb{E}}_R\beta(u)$.
\paragraph{Conclusion of the proof of Theorem \ref{thm:existenceTruncatedProblem}}Therefore, the last two convergences listed in \eqref{eq:convergencesGalerkin} can be replaced with $\nablabg \ext{\mathbb{E}}_R \beta(u_n) \weaklyto \nablabg \ext{\mathbb{E}}_R \beta(u)$ in $L^2_{L^2(\C_R)}$. Recall that $V_n(t) := \text{span}\{\phi_t b_1, ..., \phi_t b_n\}$. Given $\eta \in L^2_{W^{1\slash 2,2}},$ by density, there is a sequence $\eta_l(t) = \sum_{j=1}^l \gamma^l_j(t)\phi_t b_j$ with $\eta_l \in L^2_{V_l}$ such that $\eta_l \to \eta$ in $L^2_{W^{1\slash 2,2}}.$ %and
%\begin{align*}
 %\qquad \text{and} %\qquad \norm{\eta_l - \eta }{L^2_{W^{1\slash 2,2}}} \to 0.
%\end{align*}
If $l \leq n$, then $\eta_l \in L^2_{V_n}$ and we multiply \eqref{eq:galerkinEquation} by $\gamma^l_j(t)$ and sum up to get
\begin{equation*}\label{eq:ic2}
\int_{\Gamma_0}\tilde u_n'(t) \tilde \eta_l(t)+ \int_{\Gamma_0}\tilde u_n(t)\tilde \eta_l(t)\phi_{-t}(\sgradt \cdot \mathbf w(t)) + \int_{\C_R(t)}\nablabgt \ext{\mathcal{E}}_{R,t}(\beta(u_n(t)))\nablabgt E_R(t)(J^t_0\eta_l(t)) = 0,
\end{equation*}
where $\tilde \eta_l(t) := \phi_{-t}\eta_l(t)$. We obtain after integrating the above equation and sending $n \to \infty$ the equation
\begin{align}
\nonumber \int_0^T\langle \tilde u'(t) , \tilde \eta_l(t)\rangle_{0} &+ \int_0^T\int_{\Gamma_0}\tilde u(t)\tilde \eta_l(t)\phi_{-t}(\sgradt \cdot \mathbf w(t))\\
&+ \int_0^T\int_{\C_R(t)}\nablabgt \ext{\mathcal{E}}_{R,t}(\beta(u(t)))\nablabgt E_R(t)(J^t_0\eta_l(t)) = 0.\label{eq:recal1}
\end{align}
Let us prove that $\phi_{-t}(J^t_0)\tilde \eta_l \to \phi_{-t}(J^t_0) \tilde \eta$ in $L^2(0,T;W^{1\slash 2,2}(\Gamma_0))$. For the seminorm, we have
\begin{align*}
&\seminorm{\phi_{-t}(J^t_0)(\tilde \eta_l(t) - \tilde \eta(t))}{W^{1\slash 2,2}(\Gamma_0)}^2 %= \int_{\Gamma_0}\int_{\Gamma_0} \frac{|[\phi_{-t}(J^t_0)(\tilde \eta_l(t) - \tilde \eta(t))](x) - [\phi_{-t}(J^t_0)(\tilde \eta_l(t) - \tilde \eta(t))](y)|^2}{|x-y|^n}\\
\leq 2\int_{\Gamma_0}\int_{\Gamma_0} \frac{|\phi_{-t}(J^t_0)(x)\left([\tilde \eta_l(t) - \tilde \eta(t)](x) - [\tilde \eta_l(t) - \tilde \eta(t)](y)\right)|^2}{|x-y|^n}\\
&\quad  + 2\int_{\Gamma_0}\int_{\Gamma_0} \frac{|[\tilde \eta_l(t)-\tilde \eta(t)](y)(\phi_{-t}(J^t_0)(x)-\phi_{-t}(J^t_0)(y)|^2}{|x-y|^n}\\ 
%&\leq C_1\int_{\Gamma_0}\int_{\Gamma_0} \frac{|[\tilde \eta_l(t) - \tilde \eta(t)](x) - [\tilde \eta_l(t) - \tilde \eta(t)](y)|^2}{|x-y|^n} + C_2\int_{\Gamma_0}\int_{\Gamma_0} \frac{|[\tilde \eta_l(t)-\tilde \eta(t)](y)|^2}{|x-y|^{n-2}}}\\
&\leq C_1\seminorm{\tilde \eta_l(t) - \tilde \eta(t)}{W^{1\slash 2,2}(\Gamma_0)}^2 + C_2\int_{\Gamma_0}|[\tilde \eta_l(t)-\tilde \eta(t)](y)|^2\int_{\Gamma_0} \frac{1}{|x-y|^{n-2}}\;\mathrm{d}x\mathrm{d}y\tag{$\phi_{-t}J^t_0$ is Lipschitz with uniform Lipschitz constant}\\
%&\leq C_1\seminorm{\tilde \eta_l(t) - \tilde \eta(t)}{W^{1\slash 2,2}(\Gamma_0)}^2 + C_3\int_{\Gamma_0}|[\tilde \eta_l(t)-\tilde \eta(t)](y)|^2\tag{for example, see Lemma 5.2 in \cite{Alphonse2014a}}\\
&\leq C_3\norm{\tilde \eta_l(t) - \tilde \eta(t)}{W^{1\slash 2,2}(\Gamma_0)}^2,
\end{align*}
see Lemma 5.2 in \cite{Alphonse2014a} for last line. Integrating over time and passing to the limit shows the result. Thus $J^t_0\eta_l \to J^t_0\eta$ in $L^2_{W^{1\slash 2,2}}$ and it follows from Corollary \ref{lem:gradBoundTruncatedProblemL2H12} that $\nablabg \ext{\mathbb{E}}_R(\eta_lJ^t_0) \to \nablabg \ext{\mathbb{E}}_R(\eta J^t_0)$ in $L^2_{L^2(\C_R)}$. With this in mind, taking limits $l \to \infty$ in \eqref{eq:recal1} with $E_R = \ext{\mathbb{E}}_R$, %yields
%\[\int_0^T\langle \tilde u'(t),  \tilde \eta(t)\rangle_{0} + \int_0^T\int_{\Gamma_0}\tilde u(t)\tilde \eta(t)\phi_{-t}(\sgrad \cdot \mathbf w) + \int_0^T\int_{\C_R(t)}\nablabgt \ext{\mathcal{E}}_{R,t}(\beta(u(t)))\nablabgt \ext{\mathcal{E}}_{R,t}(J^t_0\eta(t)) = 0\]
and then because the extension can be arbitrary (Remark \ref{rem:truncatedHarmonicExtensionArbitrary}), we obtain 
\begin{equation*}\label{eq:ic1}
\int_0^T \langle \tilde u'(t), \tilde \eta(t) \rangle_{0} + \int_0^T\int_{\Gamma_0}\tilde u(t)\tilde \eta(t)\phi_{-t}(\sgradt \cdot \mathbf w(t)) + \int_0^T\int_{\C_R(t)}\nablabgt \ext{\mathcal{E}}_{R,t}(\beta(u(t)))\nablabgt E_R(J^{(\cdot)}_0\eta)(t) = 0
\end{equation*}
for all $\tilde \eta \in L^2(0,T;W^{1\slash 2,2}(\Gamma_0))$.
%The continuity and linearity wrt. the test function allows us to deduce that
%\[\int_0^T\langle \tilde u', \eta\rangle + \int_0^T\int_{\Gamma_0}\tilde u\eta\phi_{-t}(\sgrad \cdot \mathbf w) + \int_0^T\int_{\C_R(t)}\nablabgt \ext{\mathcal{E}}_{R,t}(\beta(u))\nablabgt E_R(t)(J^t_0\phi_t \tilde \eta) = 0\]
%for all $\tilde \eta \in L^2(0,T;W^{1\slash 2,2}(\Gamma_0))$. 
Now, pushing forward the first two integrals, recalling from the proof of Theorem 2.33 of \cite{Alphonse2014b} that $\dot u(t) = \phi_{-t}^*(J^0_t \tilde u'(t))$ (where $\phi_{-t}^*$ is the adjoint of $\phi_{-t}\colon W^{1\slash 2, 2}(\Gamma(t)) \to W^{1\slash 2, 2}(\Gamma_0)$) and using
\begin{align*}
\langle \tilde u'(t), \tilde \eta(t)\rangle_{0}%&= \langle (J^0_t)^{-1}J^0_t \tilde u'(t), \tilde \eta(t)\rangle_{0}
=\langle J^0_t \tilde u'(t), (J^0_t)^{-1}\tilde \eta(t)\rangle_{0}
=\langle \phi_{-t}^*(J^0_t \tilde u'(t)), \phi_t((J^0_t)^{-1}\tilde \eta(t))\rangle_{}
=\langle \dot u(t), J_0^t\phi_t\tilde \eta(t)\rangle_{}
\end{align*}
gives
\begin{align*}
\int_0^T\langle \dot u(t), J_0^t\phi_t \tilde \eta(t)\rangle_{} &+ \int_0^T\int_{\Gamma(t)}u(t)\phi_t \tilde \eta(t) J^t_0 \sgradt \cdot \mathbf w(t)\\
&+ \int_0^T\int_{\C_R(t)}\nablabgt \ext{\mathcal{E}}_{R,t}(\beta(u(t)))\nablabgt E_R(J^{(\cdot)}_0\phi_{(\cdot)} \tilde \eta)(t) = 0.
\end{align*}
Picking $\tilde \eta(t) = \phi_{-t}\eta(t) \slash\phi_{-t} J^t_0$ yields that $u$ satisfies the equality given in Theorem \ref{thm:existenceTruncatedProblem} %\[\int_0^T\langle \dot u(t), \eta(t)\rangle_{} + \int_0^T\int_{\Gamma(t)}u(t)\eta (t)\sgrad \cdot \mathbf w + \int_0^T\int_{\C_R(t)}\nablabgt \ext{\mathcal{E}}_{R,t}(\beta(u(t)))\nablabgt E_R(t)\eta(t) = 0\]
for each $\eta \in L^2_{W^{1\slash 2,2}}$. A standard argument involving integration by parts of the equation satisfied by $u$ and the equation satisfied by $\tilde u_n$ and then passage to the limit in $n$ shows that the initial condition is satisfied, see \cite[\S 5.3]{Alphonse2014b}. %This concludes the proof of Theorem \ref{thm:existenceTruncatedProblem}.
\subsection{Existence of solutions to the non-degenerate problem}\label{sec:passToLimitInR}
Therefore, for each $R \geq 1$, we have a function $u_R \in \mathbb{W}(W^{1\slash 2,2}, W^{-1\slash 2,2})$ with  $u_R(0)=u_0$ and $\nablabg\ext{\mathbb{E}}_{R}(\beta(u_R)) \in L^2_{L^2(\C_R)}$ satisfying 
\begin{equation}\label{eq:pointwiseWeakFormuR}
\langle \dot u_R(t), \eta(t)\rangle_{} + \int_{\Gamma(t)}u_R(t)\eta(t) \sgradt \cdot \mathbf w(t) + \int_{\C_R(t)}\nablabgt \ext{\mathcal{E}}_{R,t}(\beta(u_R(t)))\nablabgt (E_R\eta)(t) = 0
\end{equation}
for all $\eta \in L^2_{W^{1\slash 2,2}}$ and for almost all $t$. We now want some estimates independent of $R$. 
\subsubsection{Uniform estimates (in $R$)}
Let us pick $\eta = u_R$ and use $E_R\eta = \beta^{-1}(\ext{\mathbb{E}}_R(\beta(u_R)))$:
\begin{align*}
\frac{1}{2}\frac{d}{dt}\int_{\Gamma(t)}|u_R(t)|^2 + \frac{1}{2}\int_{\Gamma(t)}|u_R(t)|^2\sgradt \cdot \mathbf w(t) + \int_0^R \int_{\Gamma(t)}(\beta^{-1})'[\ext{\mathcal{E}}_{R,t}(\beta(u_R(t)))]|\nablabgt \ext{\mathcal{E}}_{R,t}(\beta(u_R(t)))|^2=0.
\end{align*}
Since $(\beta^{-1})' \geq C_{\beta'_{inv}}$ and $u_R(0) = u_0$, we immediately obtain via Gronwall's inequality that
\[\max_{t \in [0,T]}\norm{u_R(t)}{L^2(\Gamma(t))} + \norm{\nablabg \ext{\mathbb{E}}_R(\beta(u_R))}{L^2_{L^2(\C_R)}} \leq C\]
independent of $R$. Choosing $E_R = \ext{\mathbb{E}}_R$ in \eqref{eq:pointwiseWeakFormuR}, 
%Rearranging \eqref{eq:pointwiseWeakFormuR}, we find
%\begin{align*}
%\int_0^T\langle \dot u_R(t), \eta(t)\rangle_{}% &= -\int_0^T\int_{\Gamma(t)}u_R(t)\eta(t) \sgrad \cdot \mathbf w - \int_0^T\int_0^R  \int_{\Gamma(t)}\nablabgt \ext{\mathcal{E}}_{R,t}(\beta(u_R(t)))\nablabgt E_R(t) \eta(t)\\
%&\leq C\norm{u_R}{L^2_{L^2}}\norm{\eta}{L^2_{L^2}} + \norm{\nablabg \ext{\mathbb{E}}_R(\beta(u_R))}{L^2_{L^2(\C_R)}}\norm{\nablabg \ext{\mathbb{E}}_R\eta }{L^2_{L^2(\C_R)}}\tag{picking $E_R = \ext{\mathbb{E}}_R$}\\
%&\leq C_1\norm{\eta}{L^2_{L^2}} + C_2\norm{\eta}{L^2_{W^{1\slash 2,2}}}
%\end{align*}
using Corollary \ref{lem:gradBoundTruncatedProblemL2H12} and the uniform estimates, and taking supremums over $\eta \in L^2_{W^{1\slash 2,2}}$ gives
\[\norm{\dot u_R}{L^2_{W^{-1\slash 2,2}}} \leq C.\]
Lemma \ref{lem:traceBoundedBySeminorm} implies that $\seminorm{\mathcal{T}_tv}{W^{1\slash 2,2}(\Gamma(t))} \leq  C\norm{\nablabgt v}{L^2(\C(t))}$ for all $v \in H^1(\C(t))$; we claim the constant is independent of $t$. Indeed, an inspection of the proof of the lemma reveals that we need to check whether the trace map $\mathcal{T}_t \colon H^1(\C(t)) \to W^{1\slash 2, 2}(\Gamma(t))$ is bounded uniformly in $t$ and whether the constant in the Poincar\'e inequality on $\Gamma(t)$ is independent of $t$. The first question has been settled in \S \ref{sec:functionSpaces} and the second is also affirmative due to \cite[Lemma 5.9]{Alphonse2014a}. Using this inequality, we find
%Let us apply Lemma \ref{lem:traceBoundedBySeminorm} (\textcolor{red}{explain why constant independent of $t$ in the lemma!}):
\begin{align*}
\seminorm{\beta(u_R(t))}{W^{1\slash 2,2}(\Gamma(t))}=\seminorm{\mathcal{T}_t \mathscr{Z}_R \ext{\mathcal{E}}_{R,t} \beta(u_R(t))}{W^{1\slash 2,2}(\Gamma(t))}  \leq %C\norm{\nablabg \mathscr{Z}_R\ext{\mathcal{E}}_{R,t}(\beta(u_R(t)))}{L^2(\C(t))} = 
C\norm{\nablabg \ext{\mathcal{E}}_{R,t}(\beta(u_R(t)))}{L^2(\C_R(t))},
\end{align*}
which implies  that $\seminorm{\beta(u_R)}{L^2_{W^{1\slash 2,2}}} \leq C$. This gives boundedness of $u_R$ in the fractional seminorm, and thus %we have that
\[\norm{u_R}{L^2_{W^{1\slash 2,2}}} \leq C.\]
\subsubsection{Passage to the limit in $R$}
Therefore, we have
\begin{equation*}
\begin{aligned}
u_R &\weaklyto u &&\text{in $L^2_{W^{1\slash 2,2}}$}\\
\dot u_R &\weaklyto \dot u &&\text{in $L^2_{W^{-1\slash 2,2}}$}\\
u_R &\to u &&\text{in $L^2_{L^2}$}\\
\underline{D}_i \mathscr{Z}_R\ext{\mathbb{E}}_R(\beta(u_R)) &\weaklyto \theta_i &&\text{in $L^2_{L^2(\C)}$}\\
\partial_y \mathscr{Z}_R\ext{\mathbb{E}}_R(\beta(u_R)) &\weaklyto \theta_y &&\text{in $L^2_{L^2(\C)}$}
\end{aligned}
\end{equation*}
and we need to identify the limits $\theta_i$ and $\theta_y$. Our first task is to show that $\mathscr{Z}_R\ext{\mathbb{E}}_R(\beta(u_R)-\mean{\beta(u_R)}) \to \mathbb{E}(\beta(u)-\mean{\beta(u)})$ in $L^2_{L^2(\C)}$. Set $w_R = \beta(u_R)$ and $w = \beta(u)$; since $w_R(t) - \meanu{w_R(t)} \to w(t)-\mean{w(t)}$ in $L^2(\Gamma(t))$ for almost every $t$, by Lemma \ref{lem:ctsConvergence},
%\[\mathscr{Z}_R{\mathcal{E}}_{R,t}(w_R(t)-\mean{w}_R(t)) \to {\mathcal{E}_t}(w(t)-\mean{w(t)})\quad\text{in $L^2(\C(t))$}.\]
%So we have for a.e. $t$ 
\[f_R(t) := \norm{\mathscr{Z}_R{\mathcal{E}}_R(w_R(t)-\mean{w_R(t)}) -{\mathcal{E}}(w(t)-\mean{w(t)})}{L^2(\C(t))}^2 \to 0.\]
By virture of $\mathscr{Z}$ being an isometry, and using \eqref{eq:harmonicExtensionL2Bound} and \eqref{eq:truncatedHarmonicExtensionL2Bound},
\begin{align*}
|f_R(t)| %&\leq 2\norm{\mathscr{Z}_R{\mathcal{E}}_R(w_R(t)-\mean{w_R(t)})}{L^2(\C(t))}^2 + 2\norm{{\mathcal{E}}(w(t)-\mean{w(t)})}{L^2(\C(t))}^2\\
%&\leq 2\norm{{\mathcal{E}}_R(w_R(t)-\mean{w_R(t)})}{L^2(\C_R(t))}^2 + 2\norm{{\mathcal{E}}(w(t)-\mean{w(t)})}{L^2(\C(t))}^2\\
&\leq C_1\left(\norm{w_R(t)-\mean{w_R(t)}}{L^2(\Gamma(t))}^2+ \norm{w(t)-\mean{w(t)}}{L^2(\Gamma(t))}^2\right) =: g_R(t).
\end{align*}
Now, for almost all $t$, $g_R(t) \to 2C_1\norm{w(t)-\mean{w(t)}}{L^2(\Gamma(t))}^2$
and $\int_0^T g_R(t) %=C_1\left(\norm{w_R-\meanu{w_R}}{L^2_{L^2}}^2+ \norm{w-\mean{w}}{L^2_{L^2}}^2\right) 
\to %C_1(\norm{w-\mean{w}}{L^2_{L^2}}^2+ \norm{w-\mean{w}}{L^2_{L^2}}^2) = 
2C_1\norm{w-\mean{w}}{L^2_{L^2}}^2,$
so that by the generalised DCT, $\lim_{R \to \infty}\int_0^T f_R(t) = \int_0^T \lim_{R \to \infty}f_R(t) =0$, giving
\[\mathscr{Z}_R{\mathbb{E}}_R(w_R-\meanu{w_R}) \to {\mathbb{E}}(w-\mean{w}) \quad \text{in $L^2_{L^2(\C)}$}\]
as desired. Now, choosing $\eta$ as in \eqref{eq:testFunction} except with $h \in C_c^\infty(0,\infty)$, we have
\begin{align*}
\int_0^T\int_{\C(t)} \underline{D}_i\mathscr{Z}_R\ext{\mathbb{E}}_Rw_R \eta &= \int_0^T\int_{\C(t)} \underline{D}_i\mathscr{Z}_R\left({\mathbb{E}}_R(w_R-\meanu{w_R}) + \frac{R-y}{R}\meanu{w_R}\right) \eta\\
%&= \int_0^T\int_{\C(t)} \underline{D}_i\mathscr{Z}_R{\mathbb{E}}_R(w_R-\meanu{w_R})\eta\\
&= -\int_0^T\int_{\C(t)} \underline{D}_i\eta \mathscr{Z}_R{\mathbb{E}}_R(w_R-\meanu{w_R}) + \int_0^T\int_{\C(t)}  \mathscr{Z}_R{\mathbb{E}}_R(w_R-\meanu{w_R})\eta H\nu^\Gamma_i
\end{align*}
and passing to the limit on both sides gives
\begin{align*}
\int_0^T\int_{\C(t)} \theta_i \eta=-\int_0^T\int_{\C(t)} \underline{D}_i\eta {\mathbb{E}}(w-\mean{w}) + \int_0^T\int_{\C(t)}  {\mathbb{E}}(w-\mean{w})\eta H\nu^\Gamma_i,
\end{align*}
and then an argument similar to that in \S \ref{sec:passToLimitGalerkin} shows that $\underline{D}_i \mathbb{E}(w-\mean{w}) = \theta_i$. For the $y$ term,
\begin{align}
\nonumber \int_0^T\int_{\C(t)} \partial_y \mathscr{Z}_R\ext{\mathbb{E}}_Rw_R \eta &= \int_0^T\int_{\C(t)} \partial_y \mathscr{Z}_R\left({\mathbb{E}}_R(w_R-\meanu{w_R}) + \frac{R-y}{R}\meanu{w_R}\right) \eta\\
%\nonumber &= \int_0^T\int_{\C(t)} (\partial_y \mathscr{Z}_R{\mathbb{E}}_R(w_R-\meanu{w_R}) - \chi_{y \leq R}\frac{\meanu{w_R}}{R})\eta\\
&= -\int_0^T\int_{\C(t)} \mathscr{Z}_R{\mathbb{E}}_R(w_R-\meanu{w_R})\partial_y\eta + \chi_{y \leq R}\frac{\meanu{w_R}}{R}\eta\label{eq:conv2},
\end{align}
where the last term on the right hand side
\begin{align*}
\int_0^T\int_{\C(t)} \chi_{y \leq R}\frac{\meanu{w_R}}{R}\eta &= \int_0^T \frac{\mean{w_R(t)}}{R}\psi(t)\int_0^\infty \chi_{y \leq R}(y)h(y) \int_{\Gamma(t)} \phi_tv_0 \to 0
\end{align*}
since 
$\int_0^T {\mean{w_R(t)}\psi(t)}\slash {R} \to 0$ and $\int_0^\infty \chi_{y \leq R}(y)h(y) \to \int_0^\infty h(y)$ both due to the DCT (recall that $\mean{w_R(t)} \to \mean{w(t)}$ a.e.). Then taking the limit in \eqref{eq:conv2}, we get
\begin{align*}
\int_0^T\int_{\C(t)} \theta_y \eta 
&= -\int_0^T\int_{\C(t)} \mathbb{E}(w -\mean{w}) \partial_y \eta
\end{align*}
which again gives $\partial_y \mathbb{E}(w -\mean{w}) = \theta_y$ by similar reasoning to \S \ref{sec:passToLimitGalerkin}.
Now, integrating \eqref{eq:pointwiseWeakFormuR} in time, we can pass to the limit %in the equation
%\[\int_0^T\langle \dot u_R(t), \eta(t)\rangle_{} + \int_0^T\int_{\Gamma(t)}u_R(t)\eta(t) \sgrad \cdot \mathbf w + \int_0^T\int_{\C_R(t)}\nablabgt \ext{\mathcal{E}}_{R,t}(\beta(u_R(t)))\nablabgt E_R(t)\eta(t) = 0\]
by first of all taking $E_R\eta = \mathscr{Z}_1\ext{\mathbb{E}}_1\eta$ (this satisfies $E_R\eta|_{y=0} = \eta$ and $E_R\eta|_{y=R} = 0$ since $R \geq 1$). Replace the integral over $\C_R(t)$ by one over $\C(t)$:
\[\int_0^T\langle \dot u_R(t), \eta(t)\rangle_{} + \int_0^T\int_{\Gamma(t)}u_R(t)\eta(t) \sgradt \cdot \mathbf w(t) + \int_0^T\int_{\C(t)}\nablabgt \mathscr{Z}_R\ext{\mathcal{E}}_{R,t}(\beta(u_R(t)))\nablabgt \mathscr{Z}_1\ext{\mathbb{E}}_1\eta(t) = 0,\]
and then using the above convergence results %and also Lemma \ref{lem:convergenceTruncationsGradientL2L2} which tells us that $\nablabg \mathscr{Z}_R\ext{\mathbb{E}}_R\eta \to \nablabg \ext{\mathbb{E}}\eta$:
%\[\int_0^T\langle \dot u(t), \eta(t)\rangle_{} + \int_0^T\int_{\Gamma(t)}u(t)\eta(t) \sgrad \cdot \mathbf w + \int_0^T\int_{\C(t)}\nablabgt \ext{\mathcal{E}}_t(\beta(u(t)))\nablabgt \mathscr{Z}_1\ext{\mathbb{E}}_1\eta(t) = 0\]
and recalling that the elliptic form can have an arbitary extension, we find exactly the weak formulation \eqref{eq:nondegenerateProblem} of Theorem \ref{thm:wellPosednessOfNonDegenerateProblem}. For the conservation of mass, note that $\int_{\Gamma(t)}u_R(t) = \int_{\Gamma_0}u_0$ holds simply by testing with $\eta= E_R\eta \equiv 1$ and then we can use the strong convergence of $u_R$ to $u$ in $L^2_{L^2}$ and the continuity of $t \mapsto (u(t),1)_{L^2(\Gamma(t))}$ to get the result for all $t$.
%\[\int_0^T\langle \dot u(t), \eta(t)\rangle_{} + \int_0^T\int_{\Gamma(t)}u(t)\eta(t) \sgrad \cdot \mathbf w + \int_0^T\int_{\C(t)}\nablabgt \ext{\mathcal{E}}_t(\beta(u(t)))\nablabgt E(t)\eta(t) = 0\]
%for any extension $E\eta \in L^2_{H^1(\C)}$ which has trace at $y=0$ equal to $\eta$.
\subsection{Contraction principle}
Let $u_{01}$, $u_{02} \in L^\infty(\Gamma_0)$ be initial data and consider the respective solutions $u_{1R}$ and $u_{2R}$ to the truncated problem \eqref{eq:nondegenerateProblem}.  %The difference of the solutions satisfies 
%\begin{align*}
%\langle \dot u_{1R}(t)-\dot u_{2R}(t), \eta(t)\rangle_{} &+ \int_{\Gamma(t)}(u_{1R}(t)-u_{2R}(t))\eta(t) \sgrad \cdot \mathbf w\\
%&+ \int_{\C_R(t)}\nablabgt %(\ext{\mathcal{E}}_{R,t}(\beta(u_{1R}(t)))-\ext{\mathcal{E}}_{R,t}(\beta(u_{2R}(t))))\nablabgt E_R(t) \eta(t) = 0
%\end{align*}
The contractivity can be proved with a sensible choice of test function (for example, see \cite{BlanchardPorretta} for a continuous dependence argument).  Take the difference of the two weak formulations, set $v_{iR} = \ext{\mathbb{E}}_R(\beta(u_{iR}))$, pick $\eta = \frac{1}{\epsilon}T_\epsilon((u_{1R}-u_{2R})^+)$ and integrate over time:
\begin{align*}
\nonumber &\frac{1}{\epsilon}\int_0^t\langle \md  (u_{1R} - u_{2R}), T_\epsilon(u_{1R} - u_{2R})^+ \rangle +  \frac{1}{\epsilon}\int_0^t\int_{\Gamma(s)}(u_{1R} - u_{2R}) T_\epsilon(u_{1R} - u_{2R})^+\sgradt \cdot \mathbf{w}(t)\\
& +  \int_0^t\int_{\C_R(s)}\nablabgs (v_{1R}(s)-v_{2R}(s))\nablabgs E_R(s) \eta(s)= 0.
\end{align*}
Defining $S_\epsilon(s) := \int_0^s {T_\epsilon(r^+)}\slash {\epsilon}\;\mathrm{d}r$, applying Lemma \ref{lem:weakerIBP}, taking the limit inferior  
%\begin{align*}
%\frac{1}{\epsilon}&\int_0^t\langle \md  (u_{1R} - u_{2R}), T_\epsilon(u_{1R} - u_{2R}) \rangle = \int_{\Gamma(t)}S_\epsilon(u_{1R}(t)-u_{2R}(t))- \int_{\Gamma_0} S_\epsilon(u_{1R}(0)-u_{2R}(0))\\
%&\quad - \int_0^t\int_{\Gamma(s)}S_\epsilon (u_{1R}(s)-u_{2R}(s))\sgrad \cdot \mathbf w,
%\end{align*}
and using $S_\epsilon(\cdot) \to (\cdot)^+$, we obtain
\begin{align}
\nonumber &\int_{\Gamma(t)}(u_{1R}(t)-u_{2R}(t))^+  +\liminf_{\epsilon \to 0}\frac{1}{\epsilon}\int_0^t\int_{\C_R(s)}\nablabgs (v_{1R}(s)-v_{2R}(s))\nablabgs E_R(s) \eta(s)\\
&= \int_{\Gamma_0}(u_{1R}(0)-u_{2R}(0))^+.\label{eq:19}
\end{align}
Let us pick $E_R\eta = T_{\epsilon}((\beta^{-1}(v_{1R})-\beta^{-1}(v_{2R}))^+)\slash {\epsilon} \in L^2_{H^1(\C_R)}$ which satisfies $\mathbb{T}_{R, y=0}(E_R\eta) = T_\epsilon((u_{1R}-u_{2R})^+)\slash {\epsilon}$ and $\mathbb{T}_{R,y=R}(E_R\eta) = 0$ so is an admissible test function. Here we used that, for example, $\mathcal{T}_{R,y=0}T_\epsilon(w^+) = T_\epsilon(\mathcal{T}_{R,y=0}w)^+$ for all $w \in H^1(\C_R(0))$; this holds due to a density argument using the continuity of $T_\epsilon \circ (\cdot)^+$ between $H^1(\C_R(0))$ (see \S \ref{sec:truncations}) and between $W^{1\slash 2, 2}(\Gamma_0)$ (by Lemma \ref{lem:nemytskiiH12}). We also used that $\mathcal{T}_{R,y=0}\beta^{-1}(w) = \beta^{-1}(\mathcal{T}_{R,y=0}w)$ for all $w \in H^1(\C_R(0))$, which follows again by Lemma \ref{lem:nemytskiiH12} and the continuity of the map $\beta^{-1}\colon H^1(\C_R(0)) \to H^1(\C_R(0))$ (which is a consequence of the boundedness and continuity of $(\beta^{-1})'$).
%Here we used, for example, that $\mathcal{T}_{R,y=0}(T_\epsilon(\beta^{-1}(w))) = T_\epsilon(\beta^{-1}(\mathcal{T}_{R,y=0}w))$ for all $w \in H^1(\C_R(0))$, which holds by density, using the Lipschitz property of $T_\epsilon\colon \mathbb{R} \to \mathbb{R}$ with Lemma \ref{lem:nemytskiiH12} and the fact that $(\beta^{-1})'$ is bounded and continuous.

Setting $B_\epsilon(s) := \{(x,y) \in \Gamma(s) \times [0,R] \mid 0 \leq \beta^{-1}(v_{1R}(s,x,y))-\beta^{-1}(v_{2R}(s,x,y)) < \epsilon\},$ the elliptic form in \eqref{eq:19} is (see \S \ref{sec:truncations})
\begin{align}
\nonumber &\int_{\C_R(s)}\nablabgs (v_{1R}(s)-v_{2R}(s))\nablabgs E_R(s) \eta(s)\\
%&\frac{1}{\epsilon}\int_{\C_R(s)}\nablabgs (v_{1R}(s)-v_{2R}(s))\nablabgs T_{\epsilon}[\beta^{-1}(v_{1R}(s))-\beta^{-1}(v_{2R}(s))]\\
\nonumber &=\frac{1}{\epsilon} \int_{B_\epsilon(s)}\nablabgs (v_{1R}(s)-v_{2R}(s))\nablabgs (\beta^{-1}(v_{1R}(s))-\beta^{-1}(v_{2R}(s)))\\
%\nonumber &=\frac{1}{\epsilon} \int_{B_\epsilon(s)}\nablabgs (v_{1R}(s)-v_{2R}(s))\left((\beta^{-1})'(v_{1R})(\nablabgs v_{1R}-\nablabgs v_{2R}) +((\beta^{-1})'(v_{1R})-(\beta^{-1})'(v_{2R}))\nablabgs v_{2R}\right)\\
%\end{align*}
% Note that
%\begin{align*}
%\nablabgs (\beta^{-1}(v_{1R})-\beta^{-1}(v_{2R})) 
%&= (\beta^{-1})'(v_{1R})(\nablabgs v_{1R}-\nablabgs v_{2R})\\
%&\quad +((\beta^{-1})'(v_{1R})-(\beta^{-1})'(v_{2R}))\nablabgs v_{2R}
%\end{align*}
%so the above is
%\begin{align}
%\nonumber &\int_{\C_R(s)}\nablabgs (v_{1R}(s)-v_{2R}(s))\nablabgs E_R(s) \eta(s)\\
&\geq \frac{1}{\epsilon}\int_{B_\epsilon(s)}((\beta^{-1})'(v_{1R}(s))-(\beta^{-1})'(v_{2R}(s)))\nablabgs (v_{1R}(s)-v_{2R}(s))\nablabgs v_{2R}.\label{eq:18}
\end{align}
Here, $\beta^{-1}(v_{1R})-\beta^{-1}(v_{2R})=(\beta^{-1})'(v_{1R})(\nablabgs v_{1R}-\nablabgs v_{2R}) +((\beta^{-1})'(v_{1R})-(\beta^{-1})'(v_{2R}))\nablabgs v_{2R}$ was used to derive \eqref{eq:18}. The right hand side of \eqref{eq:18} can be estimated as follows:
\begin{align}
\nonumber &\left|\frac{1}{\epsilon}  \int_{B_\epsilon(s)}((\beta^{-1})'(v_{1R}(s))-(\beta^{-1})'(v_{2R}(s)))\nablabgs (v_{1R}(s)-v_{2R}(s))\nablabgs v_{2R}(s)\right|\\
%\nonumber &\leq \frac{1}{\epsilon}  \int_{B_\epsilon(s)}|(\beta^{-1})'(v_{1R}(s))-(\beta^{-1})'(v_{2R}(s))||\nablabgs (v_{1R}(s)-v_{2R}(s))||\nablabgs v_{2R}(s)|\\
\nonumber &\leq \frac{1}{\epsilon}\norm{(\beta^{-1})''}{\infty}  \int_{B_\epsilon(s)}|v_{1R}(s)-v_{2R}(s)||\nablabgs (v_{1R}(s)-v_{2R}(s))||\nablabgs v_{2R}(s)|\\
\nonumber &\leq \frac{1}{\epsilon}\norm{(\beta^{-1})''}{\infty}\norm{\beta'}{\infty} \int_{B_\epsilon(s)}|\beta^{-1}(v_{1R}(s))-\beta^{-1}(v_{2R}(s))||\nablabgs (v_{1R}(s)-v_{2R}(s))||\nablabgs v_{2R}(s)|\\
%\nonumber &\leq \norm{(\beta^{-1})''}{\infty}\norm{\beta'}{\infty} \int_{B_\epsilon(s)}|\nablabgs (v_{1R}(s)-v_{2R}(s))||\nablabgs v_{2R}(s)|\\
&\leq \norm{(\beta^{-1})''}{\infty}\norm{\beta'}{\infty}\int_{\C_R(s)}\chi_{B_\epsilon(s)}|\nablabgs (v_{1R}(s)-v_{2R}(s))||\nablabgs v_{2R}(s)|\label{eq:rhsOfElliptic}.
\end{align}
Now we show that this expression tends to zero as $\epsilon$ tends to zero. 
By DCT the integral on the right hand side of \eqref{eq:rhsOfElliptic} converges to the integral of the limit, so we shall focus on the pointwise limit of $\chi_{B_\epsilon(s)}$, namely,
\begin{equation}\label{eq:new1}
%\chi_{B_\epsilon(s)}(z) \to 
\chi_{\{z \in \C_R(s) \mid \beta^{-1}(v_{1R}(s,z))-\beta^{-1}(v_{2R}(s,z)) =0\}} = \chi_{\{z \in C_R(s) \mid v_{1R}(s,z)-v_{2R}(s,z) =0\}}.
\end{equation}
%pointwise a.e. $z \in C_R(s)$ and  %and we obtain
%\begin{align}
%\nonumber &\left|\frac{1}{\epsilon}  \int_{B_\epsilon(s)}((\beta^{-1})'(v_{1R}(s))-(\beta^{-1})'(v_{2R}(s)))\nablabgs (v_{1R}(s)-v_{2R}(s))\nablabgs v_{2R}(s)\right|\\
%&\leq \norm{(\beta^{-1})''}{\infty}\norm{\beta'}{\infty}\int_0^R  \int_{\Gamma(s)}\chi_{\{(x,y) \in C_R(s) \mid \beta^{-1}(v_{1R}(s,x,y))-\beta^{-1}(v_{2R}(s,x,y)) =0\}(z)}|\nablabgs (v_{1R}(s)-v_{2R}(s))||\nablabgs v_{2R}(s)|\label{eq:new2}
%\end{align}
%The set in the indicator function is
%\[\{z \in C_R(s) \mid \beta^{-1}(v_{1R}(s,z))-\beta^{-1}(v_{2R}(s,z)) =0\} = \{z \in C_R(s) \mid v_{1R}(s,z)-v_{2R}(s,z) =0\},\] and 
Observe that $\nablabgs(v_{1R}(s)-v_{2R}(s))|_{\{v_{1R}(s)-v_{2R}(s) =0\}} =0$ a.e. on $[0,R]\times \Gamma(s)$ by a theorem of Stampacchia (see \S \ref{sec:truncations}), so if $\{\beta^{-1}(v_{1R}(s))-\beta^{-1}(v_{2R}(s)) =0\}$ has positive measure, then the limit of the integral on the right hand side of \eqref{eq:rhsOfElliptic} vanishes. So then let us suppose that $\beta^{-1}(v_{1R})-\beta^{-1}(v_{2R})=0$ only on a set of measure zero. In this case, \eqref{eq:new1} is exactly 0, so again the limit vanishes. This implies in \eqref{eq:18} that
\begin{align*}
\liminf_{\epsilon \to 0} \int_{\C_R(s)}\nablabgs (v_{1R}(s)-v_{2R}(s))\nablabgs E_R(s) \eta(s) \geq 0.
\end{align*}
%(This follows from the fact that if $a_\epsilon \geq b_\epsilon$ and $|b_\epsilon| \leq c_\epsilon$ and $c_\epsilon \to 0$, then $\liminf a_e \geq 0$). 
Plugging this back into \eqref{eq:19}, we obtain the desired contractivity for $u_{1R}-u_{2R}$ at each point in time.
%\begin{align}\label{eq:21}
%\int_{\Gamma(t)}(u_{1R}(t)-u_{2R}(t))^+&\leq \int_{\Gamma_0}(u_{1R}(0)-u_{2R}(0))^+.
%\end{align}
Now, by the work in the previous subsections, thanks to the strong $L^2_{L^2}$ convergence, it follows that for almost all $t$, $u_{1R}(t) \to u_1(t)$ and $u_{2R}(t) \to u_2(t)$ in $L^2(\Gamma(t))$ for a subsequence. Therefore we can pass to the limit and we will obtain for almost all $t$
%We know that $u_{1R} \to u_1$ (strongly) in $L^2_{L^2}$ and thus we have the strong convergence $u_{1R}(t) \to u_1(t)$ in $L^2(\Gamma(t))$ for a.a. $t$ for a subsequence (and likewise for $u_2$).  Thus we can pass to the limit above to find for a.a. $t$
\begin{align*}
\int_{\Gamma(t)}(u_{1}(t)-u_{2}(t))^+ \leq \int_{\Gamma_0}(u_{01}-u_{02})^+.
\end{align*}
%This concludes the proof of Theorem \ref{thm:wellPosednessOfNonDegenerateProblem}.
%%\beta\beta\beta\beta\beta\varphi_k\varphi_k\varphi_k\varphi_k\varphi_k\varphi_k\varphi_k\varphi_k\varphi_k\varphi_k\varphi_k\varphi_k\varphi_k\varphi_k\varphi_k\varphi_k\varphi_k\varphi_k\varphi_k\varphi_k\varphi_k\varphi_k\varphi_k\varphi_k\varphi_k\varphi_k\varphi_k\varphi_k\varphi_k\varphi_k\varphi_k\varphi_k\varphi_k\varphi_k\varphi_k\varphi_k\varphi_k\varphi_k\varphi_k\varphi_k\varphi_k\varphi_k\varphi_k\varphi_k\varphi_k\varphi_k\varphi_k\varphi_k\varphi_k\varphi_k\varphi_k\varphi_k\varphi_k\varphi_k\varphi_k\varphi_k\varphi_k\varphi_k\varphi_k\varphi_k\varphi_k\varphi_k\varphi_k\varphi_k\varphi_k\varphi_k\varphi_k\varphi_k\varphi_k\varphi_k\varphi_k
\section{The fractional porous medium equation: proof of Theorem \ref{thm:wellPosednessOfFPME}}\label{sec:fpme}
We pick (see \cite[p.~102]{Eden1991}) a sequence of smooth functions $\pme_k$ such that $\pme_k(0) = 0$, $\Psi_k \to \Psi$ in $C^0_{\text{loc}}(\mathbb{R})$, $|\pme^{-1}_k(r)| \leq C_1|r| + C_2$, $|(\pme^{-1}_k)''| \leq C_k,$ and $1\slash C_k \leq \pme_k' \leq k.$
%\begin{align*}
%\frac{1}{C_k} \leq &\pme_k' \leq k\\
%&\pme^{-1}_k \to \pme^{-1} \qquad \text{in $C^0_{\text{loc}}(\mathbb{R})$}\\
%&|\pme^{-1}_k(r)| \leq C_1|r| + C_2\\
%\frac{1}{k} \leq &(\pme^{-1}_k)' \leq C_k\\
%&|(\pme^{-1}_k)''| \leq C_k.
%\end{align*}
%{Note that this implies} by the inverse function theorem (since $\pme^{-1}_k$ is $C^1$ and has non-zero derivative) that 
%\begin{align*}
%\pme_k \text{ is $C^1$}
%\end{align*}
%and from the formula $\pme_k'(b) = \frac{1}{(\pme^{-1}_k)'(\pme_k(b))}$ that
%\begin{align*}
%\end{align*}
The previous section gives us existence and uniqueness of $u_k \in \mathbb{W}(W^{1\slash 2,2}, W^{-1\slash 2,2})$ satisfying  
\begin{equation}\label{eq:toTestWith}
\langle \dot u_k(t), \eta(t)\rangle_{} + \int_{\Gamma(t)}u_k(t)\eta(t) \sgradt \cdot \mathbf w(t) + \int_{\C(t)}\nablabgt \ext{\mathcal{E}}_t(\pme_k(u_k(t)))\nablabgt E(t)\eta(t) = 0.
\end{equation}
Now we obtain appropriate estimates independent of $k$ and pass to the limit for the last time. We first look for a weak maximum principle.
%\subsection{Uniform estimates (in $k$)}
Let us set $w_k(t) = u_k(t)e^{-\lambda t}$ (note that $\dot u_k(t) = e^{\lambda t}(\dot w_k(t) + \lambda w_k(t))$) and pick $\eta = (w_k-M)^+$ where $M:= \norm{u_0}{L^\infty(\Gamma_0)}$. We would like to pick the extension of $\eta = (w_k-M)^+$ to be
\[\left(\pme_k^{-1}(\mathbb{E}(\pme_k(u_k)-\mean{\pme_k(u_k)}))e^{-\lambda t} -M\right)^+\]
but this is not possible since the bracketed term is not square integrable. Therefore we  %introduce the following cut-off function.
%The following cut-off function will come in use below.
define 
\[g_k(u_k,\rho) := \mathbb{E}(\pme_k(u_k)-\mean{\pme_k(u_k)}) + \psi_\rho \mean{\pme_k(u_k)}\]
and pick $E(w_k-M)^+ = \left(\pme_k^{-1}(g_k(u_k,\rho))e^{-\lambda t} -M\psi_\rho\right)^+ \in L^2_{H^1(\C)}$ which satisfies $\mathbb{T}E\eta = (w_k(t)-M)^+$ and 
\[\nablabgt E\eta = \begin{cases}
(\pme_k^{-1})'(g_k(u_k,\rho))e^{-\lambda t}\nablabgt g_k(u_k,\rho) -M\psi_\rho' &: \text{if $\pme_k^{-1}(g_k(u_k,\rho))e^{-\lambda t} -M\psi_\rho \geq 0$}\\
0 &:\text{otherwise}.
\end{cases}\]
%\[\nablabgt E\eta = \begin{cases}
%(\pme_k^{-1})'(g_k(u_k,\rho))e^{-\lambda t}(\nablabgt \mathbb{E}(\pme_k(u_k)-\mean{\pme_k(u_k)}) + \psi_\rho' \mean{\pme_k(u_k)}) -M\psi_\rho' &: \text{if $\pme_k^{-1}(g_k(u_k,\rho))e^{-\lambda t} -M\psi_\rho \geq 0$}\\
%0 &:\text{otherwise}
%\end{cases}\]
Equation \eqref{eq:toTestWith} reads 
\begin{align*}
\langle \dot w_k(t), (w_k(t)-M)^+\rangle_{} &+ \int_{\Gamma(t)}w_k(t)(w_k(t)-M)^+ (\lambda + \sgradt \cdot \mathbf w(t))\\
&+ e^{-\lambda t}\int_{\C(t)}\nablabgt \ext{\mathcal{E}}_t(\pme_k(u_k(t)))\nablabgt E(t)(w_k(t)-M)^+ = 0,
\end{align*}
and the gradient term on the set $\{\pme_k^{-1}(g_k(u_k,\rho))e^{-\lambda t} -M\psi_\rho \geq 0\}$ is
\begin{align}
\nonumber &\int_{\C(t)}\nablabgt \ext{\mathcal{E}}_t(\pme_k(u_k(t)))\nablabgt E(t)\eta(t)\\
\nonumber &= \int_{\C(t)}\nablabgt \ext{\mathcal{E}}_t(\pme_k(u_k(t)))\left((\pme_k^{-1})'(g_k(u_k,\rho))e^{-\lambda t}(\nablabgt \mathbb{E}(\pme_k(u_k(t))-\mean{\pme_k(u_k(t))}) + \psi_\rho' \mean{\pme_k(u_k(t))}) -M\psi_\rho'\right)\\
\nonumber &= \int_{\C(t)}(\pme_k^{-1})'(g_k(u_k,\rho))e^{-\lambda t}\left(|\nablabgt \ext{\mathcal{E}}_t(\pme_k(u_k(t)))|^2+ \partial_y\ext{\mathcal{E}}_t(\pme_k(u_k(t)))\psi_\rho' \mean{\pme_k(u_k(t))}\right)\\
\nonumber &\quad -\int_{\C(t)}M\partial_y \ext{\mathcal{E}}_t(\pme_k(u_k(t))) \psi_\rho'\\
\nonumber &\geq e^{-\lambda t}k^{-1}\int_{\C(t)}|\nablabgt \ext{\mathcal{E}}_t(\pme_k(u_k(t)))|^2 - e^{-\lambda t}C_k\int_{\C(t)}|\partial_y\ext{\mathcal{E}}_t(\pme_k(u_k(t)))||\psi_\rho' \mean{\pme_k(u_k(t))}|\label{eq:g2}\\
\nonumber &\geq \frac{1}{2} e^{-\lambda t}C_1(k)\int_{\C(t)}|\nablabgt \ext{\mathcal{E}}_t(\pme_k(u_k(t)))|^2 - \frac{1}{2} e^{-\lambda t}C_2(k)|\Gamma(t)|| \mean{\pme_k(u_k(t))}|^2\frac{1}{\rho}
\end{align}
where the last term in the antepenultimate line vanished because $\partial_y \mathcal{E}_t(\pme_k(u_k(t))-\mean{\pme_k(u_k(t))})$ has mean value zero and to derive the last line we used Young's inequality and that $|\psi_{\rho}'| \leq C\slash \rho$ with $\supp(\rho) \subset [\rho,2\rho]$:
\begin{align*}
\int_{\C(t)}  |\partial_y\ext{\mathcal{E}}_t(\pme_k(u_k(t)))||\psi_\rho' \mean{\pme_k(u_k(t))}|
%&\leq e^{-2\lambda t}C_k\int_{\C(t)} \epsilon|\nablabgt \ext{\mathcal{E}}_t(\pme_k(u_k(t)))|^2 + C_\epsilon |\psi_\rho' \mean{\pme_k(u_k(t))}|^2\\
%&\leq \int_{\C(t)} \epsilon|\nablabgt \ext{\mathcal{E}}_t(\pme_k(u_k(t)))|^2 + C_\epsilon|\Gamma(t)|| \mean{\pme_k(u_k(t))}|^2\int_0^\infty |\psi_\rho'|^2\\
&\leq \int_{\C(t)} \epsilon|\nablabgt \ext{\mathcal{E}}_t(\pme_k(u_k(t)))|^2 + C_\epsilon C|\Gamma(t)|| \mean{\pme_k(u_k(t))}|^2\int_\rho^{2\rho} \frac{1}{\rho^2}.
%&\leq \int_{\C(t)} \epsilon|\nablabgt \ext{\mathcal{E}}_t(\pme_k(u_k(t)))|^2 + C_\epsilon C|\Gamma(t)|| \mean{\pme_k(u_k(t))}|^2\frac{1}{\rho}.
\end{align*}
%hence if $\epsilon$ is small enough, \eqref{eq:g2} becomes
%\begin{align*}
%&\int_{\C(t)}\nablabgt \ext{\mathcal{E}}_t(\pme_k(u_k(t)))\nablabgt E(t)\eta(t)\\ %\geq \frac{1}{2} e^{-2\lambda t}C_k\int_{\C(t)}|\nablabgt \ext{\mathcal{E}}_t(\pme_k(u_k(t)))|^2 - C_\epsilon(k) |\psi_\rho' \mean{\pme_k(u_k(t))}|^2\\
%&\geq \frac{1}{2} e^{-\lambda t}C_1(k)\int_{\C(t)}|\nablabgt \ext{\mathcal{E}}_t(\pme_k(u_k(t)))|^2 - \frac{1}{2} e^{-\lambda t}C_2(k)|\Gamma(t)|| \mean{\pme_k(u_k(t))}|^2\frac{1}{\rho}.
%\end{align*}
Thus, we have
%\begin{align*}
%&\langle \dot w_k(t), (w_k(t)-M)^+\rangle_{} + \int_{\Gamma(t)}w_k(t)(w_k(t)-M)^+(\lambda + \sgrad \cdot \mathbf w)\\
%&+ \chi_{\{\pme_k^{-1}(g_k(u_k,\rho))e^{-\lambda t} -M\psi_\rho \geq 0\}}\frac{1}{2} e^{-2\lambda t}\left(C_1(k)\int_{\C(t)}|\nablabgt \ext{\mathcal{E}}_t(\pme_k(u_k(t)))|^2- C_2(k)|\Gamma(t)|| \mean{\pme_k(u_k(t))}|^2\frac{1}{\rho}\right) \leq 0.
%\end{align*}
\begin{align*}
&\langle \dot w_k(t), (w_k(t)-M)^+\rangle_{} + \int_{\Gamma(t)}w_k(t)(w_k(t)-M)^+(\lambda + \sgradt \cdot \mathbf w(t))\\
&-\frac{1}{2}\chi_{\{\pme_k^{-1}(g_k(u_k,\rho))e^{-\lambda t} -M\psi_\rho \geq 0\}} e^{-2\lambda t}C_2(k)|\Gamma(t)|| \mean{\pme_k(u_k(t))}|^2\frac{1}{\rho}\leq 0.
\end{align*}
Choosing $\lambda := \norm{\sgrad \cdot \mathbf w}{\infty}$ and sending $\rho \to \infty$, we can discard the last two terms and we will find
\begin{align*}
\frac{1}{2}\frac{d}{dt}\int_{\Gamma(t)}((w_k(t)-M)^+)^2-\frac{1}{2}\int_{\Gamma(t)}((w_k(t)-M)^+)^2\sgradt \cdot \mathbf w(t)%=\langle \dot w_k(t), (w_k(t)-M)^+\rangle_{} 
\leq 0.
\end{align*}
Gronwall's inequality implies boundedness of $w_k$ and hence
\begin{equation}\label{eq:ukBoundedLinfty}
\esssup_{t \in [0,T]}\norm{u_k(t)}{L^\infty(\Gamma(t))} + \esssup_{t \in [0,T]}\norm{\pme_k(u_k(t))}{L^\infty(\Gamma(t))} \leq C.
\end{equation}
The second bound holds because  $\esssup_{t \in [0,T]}\norm{\pme_k(u_k(t))}{L^\infty(\Gamma(t))} \leq \max\left(|\pme_k(C)|, |\pme_k(-C)|\right)$ (since $\pme_k$ is increasing) and the right hand side is bounded since $\Psi_k \to \Psi$.

Now we focus on obtaining a bound on $\pme_k(u_k)$ in $L^2_{W^{1\slash 2,2}}$. To this end, let us define the antiderivatives
\begin{equation*}
\begin{aligned}
H_k(r) &= \int_0^r \pme_k(s)\;\mathrm{d}s \quad\text{and}\quad G_k(r) = \int_0^r \pme^{-1}_k(s)\;\mathrm{d}s%,\\
%H(r) &= \int_0^r \pme(s)\;\mathrm{d}s &&G(r) = \int_0^r \pme^{-1}(s)\;\mathrm{d}s.
\end{aligned}
\end{equation*}
and also the antiderivatives $H$ and $G$ by the obvious formulae. 
%and the Legendre transform of $G_k$
%\[G_k^*(\tau) = \sup_{r \in \mathbb{R}}(\tau r- G_k(r)).\]
%\end{defn}
%Note that $G_k^*(\pme^{-1}_k(\tau)) + G_k(\tau) = \tau\pme^{-1}_k(\tau)$ and in fact $G_k^*(r) = H_k(r)$. 
If $u \in L^2(M)$, then $G_k(u) \in L^1(M)$ and $H_k(\pme^{-1}_k(u)) \in L^1(M)$; this follows from
\begin{align}
|G_k(u)| %&= \left| \int_0^u \pme^{-1}_k(s)\right|
 \leq \max(|\pme^{-1}_k(u)||u|, |\pme^{-1}_k(-u)||u|) \leq (C_1|u| + C_2)|u|\label{eq:boundOnGk}
\end{align}
and
\begin{equation}\label{eq:boundOnHk}
H_k(\pme^{-1}_k(u)) = u\pme^{-1}_k(u) - G_k(u) \leq |u|(C_1|u| + C_2) + |G_k(u)| \leq (C_3|u| + C_4)|u|.
\end{equation}
%we also have that $H_k(\pme^{-1}_k(u))$ is in $L^1(M)$. (
These properties are also true for $G$ and $H$.
\begin{remark}
{We could have generalised the porous medium nonlinearity $\pme(r) = |r|^{m-1}r$ to simply having $\pme$ as a continuous increasing function. In this case $\pme$ is no longer invertible so we would have to use Legendre transforms \cite{Eden1991}.}
\end{remark}
Test the equation \eqref{eq:toTestWith} with $\eta = \pme_k(u_k)$, pick $E(\pme_k(u_k)) = \mathbb{E}(\pme_k(u_k)-\mean{\pme_k(u_k)}) + \psi_\rho \mean{\pme_k(u_k)}$ and use the integration by parts formula of Lemma \ref{lem:weakerIBP}:
%\begin{align*}
%&\int_{\Gamma(T)}H_k(u_k(T)) -\int_{\Gamma_0}H_k(u_0) - \int_0^T\int_{\Gamma(t)}H_k(u_k)\sgrad \cdot \mathbf w + \int_0^T\int_{\Gamma(t)}u_k\pme_k(u_k)\sgrad \cdot \mathbf w\\
%&+ \int_0^T\int_0^\infty  \int_{\Gamma(t)}\nablabgt \ext{\mathcal{E}}_t(\pme_k(u_k))\nablabgt E(t)\pme_k(u_k(t)) = 0
%\end{align*}
%Since $H_k$ is always positive, the first term can be thrown away, and if we p, we obtain
\begin{align*}
&\int_0^T\int_0^\infty  \int_{\Gamma(t)}|\nablabgt \mathcal{E}_t(\pme_k(u_k)-\mean{\pme_k(u_k)})|^2 + \partial_y \mathcal{E}_t(\pme_k(u_k)-\mean{\pme_k(u_k)})\psi'_\rho\mean{\pme_k(u_k)}\\
&\leq \int_{\Gamma_0}H_k(u_0) + \int_0^T\int_{\Gamma(t)}H_k(u_k)\sgradt \cdot \mathbf w(t)  - \int_0^T\int_{\Gamma(t)}u_k\pme_k(u_k)\sgradt \cdot \mathbf w(t)
\end{align*}
where we threw away the $H_k(u_k(T))$ term since $H_k \geq 0$. The second term on the LHS disappears since the harmonic extension of a mean value zero function has mean value zero too. Then we finally get after using \eqref{eq:boundOnHk}
%\[|H_k(r)| = |H_k(\pme_k^{-1}(\pme_k(r)))| \leq C_1|\pme_k(r)|^2 + C_2\]
that $|H_k(u_k)| \leq C_1\norm{\pme_k(u_k)}{L^\infty_{L^\infty}}^2 + C_2$. This takes care of the second term on the right hand side, and as for the initial data, we note that $|H_k(u_0)| \leq  C_1\norm{\pme_k(u_0)}{L^\infty(\Gamma_0)}^2+C_2$
and $\norm{\pme_k(u_0)}{L^\infty(\Gamma_0)} \leq \max\left(|\pme_k(M)|, |\pme_k(-M)|\right)$, and the right hand side is bounded like before. Thus
\[\norm{\nablabg \mathbb{E}(\pme_k(u_k)-\mean{\pme_k(u_k)})}{L^2_{L^2(\C)}} \leq C.\]
%Rearranging the equaiton, we find
%\[\norm{\dot u_k}{L^2_{W^{-1\slash 2,2}}} \leq C.\]
Writing $\pme_k(u_k) = \ext{\mathbb{T}}\ext{\mathbb{E}}(\pme_k(u_k))$ and using Lemma \ref{lem:existenceOfTraceMapOnL2H1AndL2X} and the previous uniform bounds, we have
\begin{align*}
\norm{\pme_k(u_k)}{L^2_{W^{1\slash 2,2}}} %&= \norm{\ext{\mathbb{T}}\ext{\mathbb{E}}(\pme_k(u_k))}{L^2_{W^{1\slash 2,2}}} \leq C\norm{\ext{\mathbb{E}}(\pme_k(u_k))}{L^2_X}
%&= C\left(\norm{\nablabg \mathbb{E}(\pme_k(u_k)-\mean{\pme_k(u_k)})}{L^2_{L^2(\C)}} + \norm{\pme_k(u_k)}{L^2_{L^2}}\right)\\
\leq C.
\end{align*}
%due to the above bound and the $L^\infty$ bound.
Finally, integrating and rearranging \eqref{eq:toTestWith}:
\begin{align*}
\int_0^T \langle \dot u_k(t), \eta(t)\rangle_{} %&= -\int_0^T \int_{\Gamma(t)}u_k(t)\eta(t) \sgrad \cdot \mathbf w - \int_0^T \int_0^\infty  \int_{\Gamma(t)}\nablabgt \ext{\mathcal{E}}_t(\pme_k(u_k(t)))\nablabgt E(t)\eta(t)\\
&\leq \norm{\sgrad \cdot \mathbf w}{\infty}\norm{u_k}{L^2_{L^2}}\norm{\eta}{L^2_{L^2}} + \norm{\nablabg \ext{\mathbb{E}}(\pme_k(u_k))}{L^2_{L^2(\C)}}\norm{\nablabg E\eta}{L^2_{L^2(\C)}},
\end{align*}
choosing $E\eta = \mathscr{Z}_\rho \ext{\mathbb{E}}_\rho \eta$ for some $\rho > 1$ and using 
%\begin{align*}
$\norm{\nablabg \mathscr{Z}_\rho \ext{\mathbb{E}}_\rho \eta}{L^2_{L^2(\C)}} =  \norm{\nablabg \ext{\mathbb{E}}_\rho \eta}{L^2_{L^2(\C_\rho)}} \leq C\norm{\eta}{L^2_{W^{1\slash 2,2}}}$
%\end{align*}
with the last inequality by Corollary \ref{lem:gradBoundTruncatedProblemL2H12}, it easily follows that
\begin{equation}\label{eq:dotukBounded}
\norm{\dot u_k}{L^2_{W^{-1\slash 2,2}}} \leq C
\end{equation}
independent of $k$. Therefore, we have
\begin{equation}\label{eq:listOfConvergences}
\begin{aligned}
u_k &\weaklyto u &&\text{in $L^2_{L^2}$}\\
%\dot u_k &\weaklyto w &&\text{in $L^2_{W^{-1\slash 2,2}}$}\\
u_k &\to u &&\text{in $L^2_{W^{-1\slash 2,2}}$}\\
v_k:= \pme_k(u_k) &\weaklyto v &&\text{in $L^2_{W^{1\slash 2,2}}$}\\
\nablabg \mathbb{E}(v_k-\mean{v_k}) &\weaklyto \theta &&\text{in $L^2_{L^2(\C)}$}
%\mathbb{E}(v_k-\mean v_k) &\weaklyto \gamma &&\text{in $L^2_{L^2(C)}$}\\
\end{aligned}
\end{equation}
with the strong convergence by Aubin--Lions. Now the question is whether $v = \pme(u)$. If so, then we can also identify $\theta$: indeed,
%\[v_k \weaklyto \pme(u) \quad\text{in $L^2_{W^{1\slash 2,2}}$}.\]% and since the mean value operator $M$ ($Mv_k = \mean v_k$) is linear and continuous between $L^2_{L^2}$, it follows that $\mean v_k \weaklyto \mean{\pme(u)}$; thus
%\[v_k -\mean v_k \weaklyto \pme(u) - \mean{\pme(u)} \quad \text{in $L^2_{L^2}$}.\]
we know that the map $\mathbb{G}\colon L^2_{W^{1\slash 2,2}} \to L^2_{L^2(\C)}$  defined by $\mathbb{G}w = \nablabg \mathbb{E}(w-\mean{w})$ is linear and also continuous by Corollary \ref{lem:boundsOnMathbbE}: \[\norm{\mathbb{G}w}{L^2_{L^2(\C)}} \leq C_1\norm{w-\mean{w}}{L^2_{W^{1\slash 2,2}}} \leq C_2\norm{w}{L^2_{W^{1\slash 2,2}}},\]
and this implies that $\mathbb{G}v_k \weaklyto \mathbb{G}\pme(u)$ in $L^2_{L^2(\C)}$, i.e., $\nablabg \mathbb{E}(v_k-\mean{v_k}) \weaklyto \nablabg \mathbb{E}(\pme(u)-\mean{\pme(u)})$ in $L^2_{L^2(\C)}$. Now we show that indeed $v = \pme(u)$.
\subsection{Identification of $v \equiv \pme(u)$}
Let us define 
\begin{align*}
J_k(v) = \begin{cases}
\int_0^T \int_{\Gamma(t)} G_k(v) &\text{if $G_k(v) \in L^1_{L^1}$}\\
0 &\text{otherwise}
\end{cases} \quad \text{and}\quad
J(v) = \begin{cases}
\int_0^T \int_{\Gamma(t)} G(v) &\text{if $G(v) \in L^1_{L^1}$}\\
0 &\text{otherwise}.
\end{cases}
\end{align*}
Note that if $v \in L^2_{L^2}$ then $G_k(v)$, $G(v) \in L^1_{L^1}$ (see \eqref{eq:boundOnGk}).
\begin{lem}\label{lem:semicontinuity}The map
\[v \mapsto \int_0^T \int_{\Gamma(t)} G(v)\]
from $L^2_{L^2}$ into $\mathbb{R}$ is lower semicontinuous.
\end{lem}
\begin{proof}
First, observe that $G\colon \mathbb{R} \to \mathbb{R}$ is convex, proper and continuous, hence (for example by adapting Proposition 8.1 in \cite[Chapter II]{showalter}) the map 
\[w \mapsto \int_{\Gamma(t)}G(w) \quad\text{for $w \in L^2(\Gamma(t))$}\]
(which is well-defined, for example, see \eqref{eq:boundOnGk}) is lower semicontinuous for each fixed $t$. If $v_n \to v$ in $L^2_{L^2}$, we have $v_{n_j}(t) \to v(t)$ in $L^2(\Gamma(t))$ for almost all $t$, so 
\begin{equation}\label{eq:22}
\int_{\Gamma(t)}G(v(t)) \leq \liminf_{{n_j} \to \infty}\int_{\Gamma(t)}G(v_{n_j}(t)).
\end{equation}
Integrating \eqref{eq:22}, and since $\int_{\Gamma(t)}G(v_{n_j}(t)) \geq 0$ and $t \mapsto \int_{\Gamma(t)}G(v_{n_j}(t)) = \int_{\Gamma_0}G(\tilde v_{n_j}(t))J^0_t$ is measurable, we can apply Fatou's lemma to give
%\[\int_0^T \liminf_{{n_j} \to \infty}\int_{\Gamma(t)}G(v_{n_j}(t)) \leq \liminf_{{n_j} \to \infty}\int_0^T \int_{\Gamma(t)} G(v_{n_j}(t)).\]
%Integrating \eqref{eq:22} and applying this fact gives
\[\int_0^T\int_{\Gamma(t)}G(v(t)) \leq \int_0^T \liminf_{{n_j} \to \infty}\int_{\Gamma(t)}G(v_{n_j}(t))\leq \liminf_{{n_j} \to \infty}\int_0^T\int_{\Gamma(t)}G(v_{n_j}(t)).\]
Thus far we have shown that for any sequence $v_n \to v$ converging in $L^2_{L^2}$, $J(v) \leq \liminf_{j  \to \infty}J(v_{n_j})$ holds for a subsequence $n_j$. Now, if $v_n \to v$ in $L^2_{L^2}$, then it follows that there is a subsequence $v_{n_j}$ such that 
\begin{equation}\label{eq:23}
\liminf_{n \to \infty}J(v_n) = \lim_{j \to \infty}J(v_{n_j})
\end{equation}
by definition of the $\liminf$ ($J$ is non-negative, so either $\liminf J(v_n) = \infty$ or $\liminf J(v_n) = C \geq 0$; the former case makes the problem trivial). We know that there is a subsequence $n_{j_{k}}$ of $n_{j}$ such that $J(v) \leq \liminf_{k \to \infty}J(v_{n_{j_k}}) = \lim_{j \to \infty}J(v_{n_j}) = \liminf_{n \to \infty}J(v_n)$ with the first equality because the limit of $J(v_{n_{j_k}})$ is the same as the limit of $J(v_{n_j})$ and the second equality from \eqref{eq:23}.
%So putting this together
%\[J(v) \leq \lim_{j \to \infty}J(v_{n_j}) = \liminf_{n \to \infty}J(v_n)\]
%from \eqref{eq:23} .
\end{proof}
\begin{lem}We have $u = \pme^{-1}(v)$.
\end{lem}
\begin{proof}
By convexity of $G_k$ and $G$, $J_k$ and $J$ are also convex (see \cite[\S 2.4]{Blanchard1988}). If the G\^ateaux derivative of $J_k$ or $J$ exists at a particular point, then the set of subdifferentials of $J_k$ or $J$ coincides with the set of G\^ateaux derivatives at that point \cite[Proposition 3.33]{hemi}. By a direct calculation, the subdifferentials are
\begin{align*}
\nonumber \partial J_k(v_k) = \{w \in L^2_{L^2} \mid w &= \pme^{-1}_k(v_k) \text{ in $L^2_{L^2}$}\}\quad\text{and}\quad
\partial J(v) = \{w \in L^2_{L^2} \mid w = \pme^{-1}(v) \text{ in $L^2_{L^2}$}\}\label{eq:subdifferentialOfJ}.
\end{align*}
%(where the equality $a=b$ a.e. in the set means that for almost all $t \in [0,T]$, $a(t)=b(t)$ almost everywhere in $\Gamma(t)$).
%To see this, for example, we have by definition
%\[\partial J(u) = \{v \in L^2_{L^2} \mid \lim_{\lambda \to 0} \frac{J(u+\lambda w)-J(u)}%{\lambda} \geq (v,w)_{L^2_{L^2}}\}.\]
%But  ...
%And because the condition is that $(\pme^{-1}(u) - v, w)_{L^2_{L^2}} \geq 0$ for all $w$, it follows that $v=\pme^{-1}(u)$ a.e.
% To see this, for example, taking $v$, $h \in L^2_{L^2}$,
%\begin{align*}
%\frac{J(v+\lambda h)-J(v)}{\lambda} &=  \int_0^T \int_{\Gamma(t)}\frac{G(v+\lambda h) - G(v)}{\lambda} = \frac{1}{\lambda} \int_0^T \int_{\Gamma(t)}\int_0^1 \frac{d}{ds}(G(v+\lambda hs))\;\mathrm{d}s\\
%%&= \frac{1}{\lambda} \int_0^T \int_{\Gamma(t)}\int_0^1 \pme^{-1}(v+\lambda hs)\lambda h\;\mathrm{d}s\\
%&= \int_0^T \int_{\Gamma(t)}\int_0^1 \pme^{-1}(v+\lambda hs)h\;\mathrm{d}s,
%\end{align*}
%and now we can take limits (using dominated %convergence):
%\begin{align*}
%\lim_{\lambda \to 0}\frac{J(v+\lambda h)-J(v)}{\lambda} 
%= \int_0^T \int_{\Gamma(t)}\int_0^1 \pme^{-1}(v)h\;\mathrm{d}s
%=  \int_0^T \int_{\Gamma(t)}\pme^{-1}(v)h = (\pme^{-1}(v), h)_{L^2_{L^2}}.
%\end{align*}
%So the G\^ateaux derivative of $J(v)$ is $\pme^{-1}(v)$. 
By definition (see \cite[Definition 3.31]{hemi}), since $\pme^{-1}_k(v_k) \in \partial J_k(v_k)$, %the subdifferential $\pme^{-1}_k(v_k)$ of $J_k$ at $v_k$ satisfies 
 for all $w \in L^2_{L^2}$,
\begin{equation}\label{eq:pme1}
 \int_0^T\int_{\Gamma(t)}G_k(v_k) + \int_0^T\int_{\Gamma(t)}\pme^{-1}_k(v_k)w \leq  \int_0^T\int_{\Gamma(t)}G_k(w) + \int_0^T\int_{\Gamma(t)}\pme^{-1}_k(v_k)v_k.
\end{equation}
%Now we want to pass to the limit in this equation and we will get
%\[ \int_0^T\int_{\Gamma(t)}\Psi(u) + \int_0^T\int_{\Gamma(t)}u w \leq  \int_0^T\int_{\Gamma(t)}\Psi(w) + \int_0^T\int_{\Gamma(t)}u u\]
We want to pass to the limit in this inequality using \eqref{eq:listOfConvergences} and the methods of \cite{Eden1991}. For the first term on the right hand side: for almost all $t$ and almost all $x \in \Gamma(t)$,  $G_k(w(t,x)) \to  G(w(t,x))$ by the convergence of $\pme^{-1}_k \to \pme^{-1}$. We also have by \eqref{eq:boundOnGk} $|G_k(w(t,x))| \leq C(|w(t,x)|^2 + |w(t,x)|)$, and the right hand side is in $L^1_{L^1}$, so by the DCT, $G_k(w) \to G(w)$ in $L^1_{L^1}$, which obviously implies
\[\int_0^T\int_{\Gamma(t)}G_k(w) \to \int_0^T\int_{\Gamma(t)}G(w).\]
For the second term on the right hand side, since $u \in L^2_{L^2},$ 
\begin{align*}
\int_0^T\int_{\Gamma(t)}\pme^{-1}_k(v_k)v_k = \langle \pme^{-1}_k(v_k), v_k \rangle_{L^2_{W^{-1\slash 2,2}}, L^2_{W^{1\slash 2,2}}} \to \langle u, v \rangle_{L^2_{W^{-1\slash 2,2}}, L^2_{W^{1\slash 2,2}}} = \int_0^T\int_{\Gamma(t)}u v.
\end{align*}
For the first term on the left hand side, we first show an intermediary step, that 
\begin{equation}\label{eq:pme2}
\lim_{k \to \infty}\int_0^T\int_{\Gamma(t)}G_k(v_k) - G(v_k) = 0.
\end{equation}
To see this, note that
\begin{align*}
|G_k(v_k(t,x)) - G(v_k(t,x))| &= \bigg|\int_0^{v_k(t,x)}(\pme^{-1}_k(s) - \pme^{-1}(s))\bigg| 
%&\leq \sup_{s \in [-\norm{v_k}{\infty},\norm{v_k}{\infty}]}\norm{v_k}{\infty}|\pme^{-1}_k(s)-\pme^{-1}(s)|\\
\leq C\sup_{s \in [-C, C]}|\pme^{-1}_k(s)-\pme^{-1}(s)|,
\end{align*}
hence
\begin{align*}
\left|\int_0^T\int_{\Gamma(t)}G_k(v_k) - G(v_k)\right| %&\leq \int_0^T\int_{\Gamma(t)}|G_k(v_k) - G(v_k)|\\
%&\leq C\int_0^T\int_{\Gamma(t)}\sup_{s \in [-C, C]}|\pme^{-1}_k(s)-\pme^{-1}(s)|\\
&\leq |\Gamma|T\sup_{s \in [-C, C]}|\pme^{-1}_k(s)-\pme^{-1}(s)|C \to 0.
\end{align*}
%by the convergence of $\pme^{-1}_k \to \pme^{-1}$.
By weak lower semicontinuity of the map $v \mapsto \int_0^T \int_{\Gamma(t)} G(v)$ (Lemma \ref{lem:semicontinuity}), we have
\begin{align*}
\int_0^T \int_{\Gamma(t)}G(v) &\leq \liminf_{k \to \infty} \int_0^T \int_{\Gamma(t)}G(v_k) = \liminf_{k \to \infty} \int_0^T \int_{\Gamma(t)}G_k(v_k)
\end{align*}
with the equality by \eqref{eq:pme2}. Lastly, the second term on the left hand side is obvious. Now we can take $\liminf_{k \to \infty}$ in \eqref{eq:pme1} and use the above facts to get
\[ \int_0^T\int_{\Gamma(t)}G(v) + \int_0^T\int_{\Gamma(t)}u w \leq  \int_0^T\int_{\Gamma(t)}G(w) + \int_0^T\int_{\Gamma(t)}u v,\]
which is exactly the statement $u \in \partial J(v)$, i.e., $u = \pme^{-1}(v)$.
\end{proof}
That $u \in L^\infty_{L^\infty}$ follows from the strong convergence in $L^2_{W{-1\slash 2, 2}}$ and the $L^\infty$ estimate \eqref{eq:ukBoundedLinfty},  see \cite[\S 3(b)]{Alphonse2014} for more details. Integrating \eqref{eq:toTestWith} by parts over time and letting $\eta \in \mathbb{W}(W^{1\slash 2,2}, L^2)$ with $\eta(T)=0$, the equation we want to pass to the limit in is 
\begin{align*}
-\int_0^T \int_{\Gamma(t)}\dot \eta(t)u_k(t) + \int_0^T \int_{\C(t)}\nablabgt \ext{\mathcal{E}}_t(\pme_k(u_k(t)))\nablabgt E(t)\eta(t) = \int_{\Gamma_0}u_0\eta(0), 
\end{align*}
and this is easily done using the convergence results and will result in the equation in Definition \ref{defn:weakSolutionToFPME}.
%\[-\int_0^T \int_{\Gamma(t)}\dot \eta(t)u(t) + \int_0^T \int_0^\infty  \int_{\Gamma(t)}\nablabgt \mean{\mathcal{E}}_t(\pme(u))\nablabgt E(t)\eta(t) = \int_{\Gamma_0}u_0\eta(0).\]
\subsection{Contraction principle and conservation of mass}
We know that the solutions $u_{1k}$ and $u_{2k}$ of the non-degenerate problem (with nonlinearity $\pme_k$) and initial data $u_{01}$ and $u_{02}$ respectively satisfy 
%We know that the solutions $u_1$ and $u_2$ to the FPME with initial data $u_{01}$ and $u_{02}$ respectively are the limit of approximations $u_{1k}$, $u_{2k}$ to the non-degenerated problem, and these satisfy
\begin{equation}\label{eq:21a}
\int_{\Gamma(t)}(u_{1k}(t)-u_{2k}(t))^+ \leq \int_{\Gamma_0}(u_{01} - u_{02})^+\quad\text{for all $k$}
\end{equation}
by Theorem \ref{thm:wellPosednessOfNonDegenerateProblem}. We have shown that (for a subsequence) $u_{ik}$ converges to $u_i$, the solution of the fractional porous medium equation with initial data $u_{0i}$. Now, with $\tilde u_{ik} := \phi_{-(\cdot)}u_{ik}$, the bounds \eqref{eq:ukBoundedLinfty} and \eqref{eq:dotukBounded} translate into (see \cite[Lemma 2.8]{Alphonse2014} and \cite[Proof of Theorem 2.33]{Alphonse2014b}) 
\[\norm{\tilde u_{ik}}{L^\infty(0,T;L^\infty(\Gamma_0))} + \norm{\tilde u_{ik}'}{L^2(0,T;W^{-1\slash 2,2}(\Gamma_0))} \leq C,\]
thus by Aubin--Lions (Theorem II.5.16 in \cite{Boyer}), for a subsequence and for every $t \in [0,T]$, % $\tilde u_{ik} \to \tilde u_i$ in $C^0([0,T];W^{-1\slash 2,2}(\Gamma_0))$, therefore, 
\begin{equation}\label{eq:strongConvergenceCtsDep}
\begin{aligned}
&\tilde u_{ik}(t) \to \tilde u_i(t) &&\text{in $W^{-1\slash 2,2}(\Gamma_0)$}.%\\
%&\tilde u_{2g(h(k))}(t) \to \tilde u_2(t) &&\text{in $W^{-1\slash 2,2}(\Gamma_0)$}.
\end{aligned}
\end{equation}
By the uniform bound, we have for almost all $t$ that %$\tilde u_{il_t(k)}(t) \weakstar \tilde u_i(t)$ in $L^\infty(\Gamma_0)$ % and $\tilde u_{i l_t (k)}(t) \weaklyto \theta_i(t)$ in $L^p(\Gamma_0)$ for any $p\geq 1$. 
%and we can identify $\theta_i(t) = \tilde u_i(t)$ 
%. Therefore, 
$\tilde u_{1 l_t(k)}(t) - \tilde u_{2 l_t  (k)}(t)\weaklyto \tilde u_1(t) - \tilde u_2(t)$ in $L^1(\Gamma_0)$ (the identification is thanks to the strong convergence \eqref{eq:strongConvergenceCtsDep}). Since $(\cdot)^+$ is a convex function, $I_t\colon L^1(\Gamma(t)) \to \mathbb{R}$ defined by $I_t(w) = \int_{\Gamma(t)}w^+$ is convex, and clearly it is also continuous. Then, by a corollary of Mazur's lemma \cite[Corollary 3.8 and Remark 5]{brezis2010functional}, $I_t$ is weakly lower semicontinuous, which from \eqref{eq:21a} gives
\begin{equation*}\label{eq:ctsDependenceForAAt}
\int_{\Gamma(t)}(u_1(t)-u_2(t))^+\leq \int_{\Gamma_0}(u_{01}-u_{02})^+\quad\text{for almost all $t \in [0,T]$}.
\end{equation*}
In fact this holds for every $t \in [0,T]$. Take an arbitrary $t \in [0,T]$ and a sequence $t_j \to t$ such that $\tilde u_{ik}(t_j)$ is bounded in $L^\infty(\Gamma_0)$. This gives $\tilde u_{ik}(t_j) \weakstar \tilde u_{ik}(t)$ in $L^\infty(\Gamma_0)$ since $\tilde u_{ik} \in C^0([0,T];W^{-1\slash 2, 2}(\Gamma_0))$. The weak-star lower semicontinuity of norms gives %with \eqref{eq:w1} and \eqref{eq:w2} give
\[\norm{\tilde u_{ik}(t)}{L^\infty(\Gamma_0)} \leq C\quad\text{for every $t \in [0,T]$},\]
and the argument previously given can be repeated and we will get $\tilde u_{ik}(t) \weakstar \tilde u_{i}(t)$, and then we can pass to the limit in the contraction result satisfied by $\tilde u_{1k}(t_j)-\tilde u_{2k}(t_j)$, first in $j$ and then in $k$.

The conservation of mass follows easily by passing to the limit in $\int_{\Gamma(t)}u_k(t) = \int_{\Gamma_0}u_0$.% holds for all $t$ and all $k$, and since (for a subsequence) $u_k(t) \weakstar u(t)$ for all $t$, we can easily pass to the limit.
\section{Concluding remarks}\label{sec:conclusions}
The (non-fractional) porous medium equation on an evolving surface can be also tackled in this way, as a limit of approximations; of course the problem is easier in that case as we would not need \S \ref{sec:fractionalLaplacianOnCompactManifolds}, \S \ref{sec:fractionalLaplacianOnL2X} and parts of \S \ref{sec:functionSpaces}, and the non-degenerate problem in \S \ref{sec:nondegenerateProblem} can be handled with a fixed point argument using the linear theory in \cite{Alphonse2014b}, as done in \cite{Alphonse2014} for a Stefan problem. We name a few of the many interesting open issues left. We required bounded initial data for the results above but the $L^1$-continuous dependence result leaves us in good position to extend the results to integrable data if we manage to obtain a smoothing effect (for which the work \cite{BonforteGrillo} by Bonforte and Grillo may be useful). There is also the fast diffusion or the singular case where $m \in (0,1)$ which we have not addressed. A fundamental property enjoyed by solutions of the fractional porous medium equation on a stationary domain is regularity in time \cite[Theorem 2.3]{DePablo2011}, that is, the solution has a time derivative in $L^1$. In the stationary case, this regularity is obtained partially by a rescaling argument of \cite{BenilanCrandall} and using the $L^1$-continuous dependence applied to a solution and its rescaled version. This does not work in our setting since rescaled solutions live on a different evolving hypersurface, so the continuous dependence inequality cannot be applied. This result would be useful because it would allow us to study qualitative properties such as the effect the geometry of the hypersurface has on the solution. An obvious further extension is to study this theory of weak solutions with a general exponent in the fractional Laplacian $(-\Delta_{\Gamma(t)})^s$: for this of course \cite{CaffarelliSilvestre} is the obvious starting point and the methodology we used in this paper should work. %One can also study further qualitative properties of solutions to do with the domain evolution. %We hope to do this in future work (?).
\section*{Acknowledgements}
A.A. was supported by the EPSRC grant EP/H023364/1 within the MASDOC Centre for Doctoral Training. The authors thank the anonymous referee for some helpful suggestions. A.A. is also grateful to Juan Luis V{\'a}zquez and James Robinson for useful comments and additional references.
%\section{Bibliography styles}
%There are various bibliography styles available. You can select the style of your choice in the preamble of this document. These styles are Elsevier styles based on standard styles like Harvard and Vancouver. Please use Bib\TeX\ to generate your bibliography and include DOIs whenever available.
%
%Here are two sample references: \cite{Feynman1963118,Dirac1953888}.
%
%\section*{References}
\bibliographystyle{abbrv}
\bibliography{FPMENewBib}

\end{document}